\documentclass[preprint,12pt]{elsarticle}

\usepackage{amssymb}
\usepackage{amsmath}
\usepackage{color}
\usepackage{subfigure}

\newcommand{\be}{\begin{eqnarray}}
\newcommand{\ee}{\end{eqnarray}}
\newcommand{\ba}{\begin{align}}
\newcommand{\ea}{\end{align}}
\newcommand{\bi}{\begin{itemize}}
\newcommand{\ei}{\end{itemize}}

\newcommand{\N}{\mathbb N}

\newcommand{\R}{\mathbb R}

\newcommand{\beq}[1]{\begin{equation} \label{#1}}
\newcommand{\eeq}{\end{equation}}
\newcommand{\beqa}{\begin{eqnarray}}
\newcommand{\eeqa}{\end{eqnarray}}
\newcommand{\bal}{\begin{align}}
\newcommand{\eal}{\end{align}}
\newcommand{\bsub}{\begin{subequations}}
\newcommand{\esub}{\end{subequations}}
\newcommand{\eqlab}[1]{\label{eq:#1}}
\renewcommand{\eqref}[1]{(\ref{eq:#1})}
\newcommand{\eqsref}[2]{(\ref{eq:#1}) and~(\ref{eq:#2})}

\newcommand{\figref}[1]{Fig.~\ref{fig:#1}}
\newcommand{\figlab}[1]{\label{fig:#1}}

\newcommand{\secref}[1]{Section~\ref{sec:#1}}
\newtheorem{mainres}{Main result}
\newcommand{\seclab}[1]{\label{sec:#1}}
\newcommand{\defnref}[1]{Definition~\ref{definition:#1}}
\newcommand{\defnlab}[1]{\label{definition:#1}}

\newcommand{\remref}[1]{Remark~\ref{remark:#1}}
\newcommand{\remlab}[1]{\label{remark:#1}}

\newcommand{\lemmaref}[1]{Lemma~\ref{lemma:#1}}
\newcommand{\lemmalab}[1]{\label{lemma:#1}}
\newcommand{\propref}[1]{Proposition~\ref{proposition:#1}}
\newcommand{\proplab}[1]{\label{proposition:#1}}

\newcommand{\appref}[1]{\ref{app:#1}}
\newcommand{\applab}[1]{\label{app:#1}}

\newcommand{\tablab}[1]{\label{tab:#1}}
\newcommand{\tabref}[1]{Table~\ref{tab:#1}}

\newcommand{\conref}[1]{Conjecture~\ref{conjecture:#1}}
\newcommand{\conlab}[1]{\label{conjecture:#1}}
  \newtheorem{prop}{Proposition}
\newtheorem{cor}{Corollary}

\newtheorem{con}{Conjecture}

\newtheorem{theorem}{Theorem}
\newtheorem{lemma}[theorem]{Lemma}
\newdefinition{remark}{Remark}
\newdefinition{definition}{Definition}
\newproof{proof}{Proof}
\newproof{pot}{Proof of Theorem \ref{thm2}}
\definecolor{rred}{rgb}{0.7,0.0,0.2}
\definecolor{bblue}{rgb}{0.2,0.0,0.7}
\newcommand{\xqed}[1]{%
  \leavevmode\unskip\penalty9999 \hbox{}\nobreak\hfill
  \quad\hbox{\ensuremath{#1}}}
  \newcommand{\xqedhere}[2]{%
  \rlap{\hbox to#1{\hfil\llap{\ensuremath{#2}}}}}

\journal{Journal of Differential Equations}

\begin{document}

\begin{frontmatter}




\title {Periodic orbits near a bifurcating slow manifold}

%
%
%

\author{K. Uldall Kristiansen}

\address{Department of Mathematics and Computer Science, Technical University of Denmark, 2800 Kgs. Lyngby, DK}

\begin{abstract}
This paper studies a class of $1\frac12$-degree-of-freedom Hamiltonian systems with a slowly varying phase that unfolds a Hamiltonian pitchfork bifurcation. The main result of the paper is that there exists an order of $\ln^2\epsilon^{-1}$-many periodic orbits that all stay within an $\mathcal O(\epsilon^{1/3})$-distance from the union of the normally elliptic slow manifolds that occur as a result of the bifurcation. Here $\epsilon\ll 1$ measures the time scale separation. These periodic orbits are predominantly unstable. 
The proof is based on averaging of two blowup systems, allowing one to estimate the effect of the singularity, combined with results on asymptotics of the second Painleve equation.
The stable orbits of smallest amplitude that are {persistently} obtained by these methods remain slightly further away from the slow manifold being distant by an order $\mathcal O(\epsilon^{1/3}\ln^{1/2}\ln \epsilon^{-1})$. 
\end{abstract}


\end{frontmatter}


\section{Introduction}

{This paper considers a class of slow-fast $1\frac12$-degrees-of-freedom (d.o.f.) Hamiltonian systems, that includes the following example:
\begin{align}
 \dot x &=y,\eqlab{toy}\\
 \dot y&=-x(-\sin u+2x^2),\nonumber\\
 \dot u&=\epsilon,\nonumber
\end{align}
where $u\in S^1=\mathbb R/(2\pi \mathbb Z)$ is a slowly varying phase. As the systems considered in this paper, the example \eqref{toy} is symmetric with respect to the reflection 
\begin{align}
\mathcal R:\quad (x,y,u)\mapsto (-x,-y,u),\eqlab{reflection}
\end{align}
and with respect to a time-reversible symmetry 
\begin{align}
\mathcal T_\tau:\quad (x,y,u)(t)\mapsto (x,-y,\tau-u)(-t),\eqlab{Ttau}
\end{align}
for $$\tau = \pi.$$ The more general symmetry $\mathcal T_\tau$ for $0<\tau<2\pi$ will be used later on. Example \eqref{toy} also possesses a slow manifold of normally elliptic critical points of \eqref{toy}$_{\epsilon=0}$:
\begin{align}
 S=\left\{(x,y,u)\vert y=0,\,x=\left\{\begin{array}{cc}
                                      0 & u\notin [0,\tau],\\
                                      \pm \kappa(u)&u\in (0,\tau)
                                     \end{array}\right.,\,u\ne 0,\,\tau\right\},\eqlab{M0toy}
\end{align}
where $\kappa^2(u) = \frac12 \sin u>0$ for $u\in (0,\pi)$ and $\tau=\pi$. See also \figref{slowman}. The set $S$ is not uniformly normally elliptic because of the pitchfork bifurcations of \eqref{toy}$_{\epsilon=0}$ at $u=0$ and $u=\tau$. The main aim of this paper is to investigate the existence and stability of periodic orbits that remain close to $S$ and therefore pass close to the bifurcation points for $\epsilon\ll 1$.}

\begin{figure}[h!]
\begin{center}
\subfigure[]{\includegraphics[width=.95\textwidth]{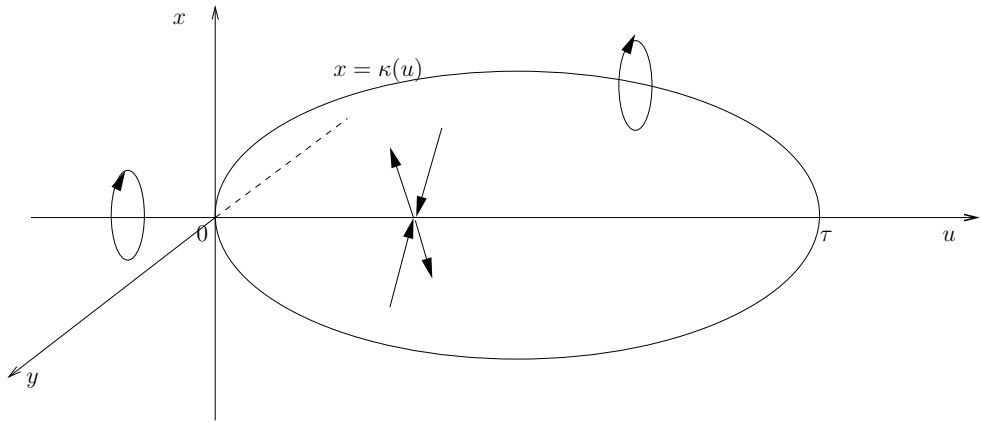}}
\subfigure[]{\includegraphics[width=.95\textwidth]{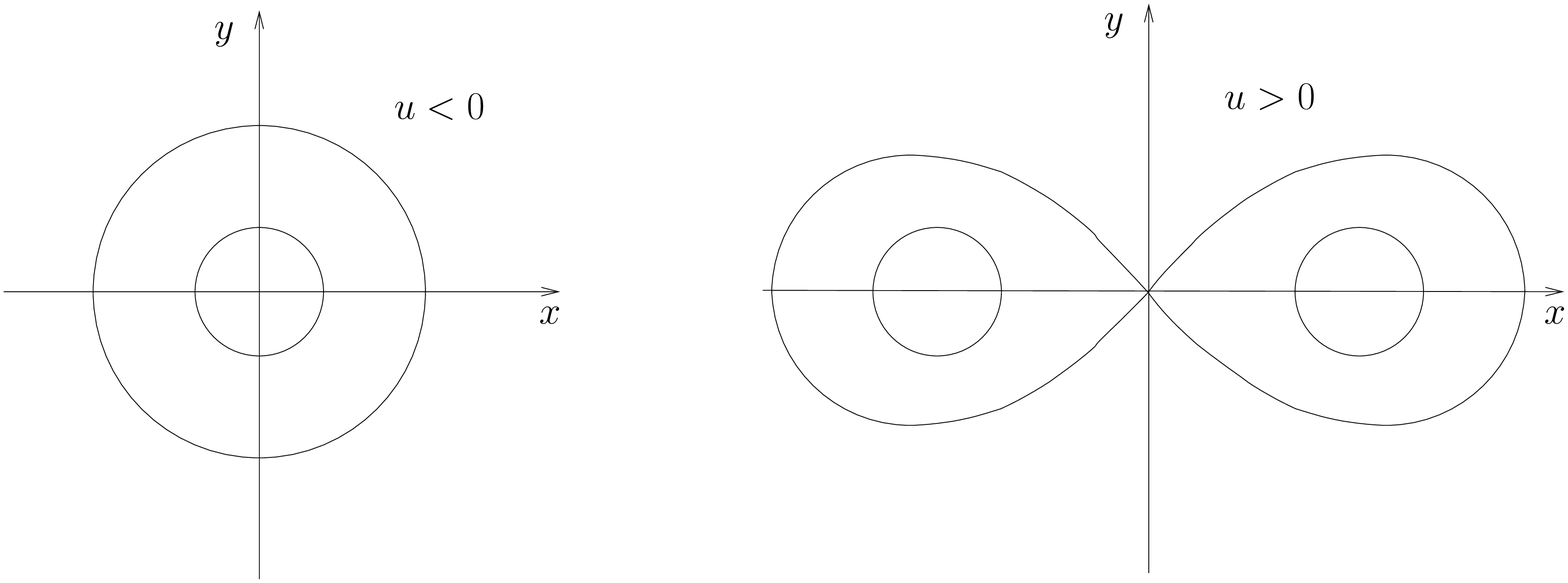}}
\end{center}
\caption{The slow manifold (a) and the dynamics of the frozen system (b). The variables $x$ and $y$ are fast whereas $u$ is slow.}
\figlab{slowman}
\end{figure}

Many problems in physics can be reduced to a 2-d.o.f. Hamiltonian system with one d.o.f. being fast relative to another slow d.o.f., see e.g. \cite{ref1,fif1,nei04,kri2}. Such systems can be further reduced to slow-fast $1\frac12$-d.o.f. systems considered in this paper by reduction of energy. 
Either time $t$ of a slow-fast system is \textit{fast} as in \eqref{toy} so that the velocities of the fast variables are $\mathcal O(1)$ while velocities of the slow ones are $\mathcal O(\epsilon)$. The system is then said to be \textit{fast}. Otherwise time $\epsilon t$ is \textit{slow} such that the velocities of the fast variables are $\mathcal O(\epsilon^{-1})$ while velocities of the slow ones are $\mathcal O(1)$. In this situation the system is said to be \textit{slow}. The limit $\epsilon=0$ of the fast system gives the \textit{layer problem}
while the limit $\epsilon=0$ of the slow system gives the \textit{reduced problem}. For \eqref{toy} the layer problem is 
\begin{align}
 x' &=y,\eqlab{layer1}\\
 y' &=-x(-\sin u+2x^2),\eqlab{layer2}\\
 u'&=0.\nonumber
\end{align}
The equations \eqsref{layer1}{layer2} for the fast variables $x$ and $y$ are called the \textit{fast sub-system}. On the other hand, the reduced problem for \eqref{toy} is
\begin{align}
 0 &=y,\eqlab{reducedProblemToy}\\
 0 &=-x(-\sin u+2x^2),\nonumber\\
 u'&=1.\nonumber
\end{align}
The reduced problem is only defined on critical points of the layer equations. The set of critical points make up the slow manifold for $\epsilon=0$. In particular the set $S$ is a set of elliptic critical points of the fast sub-system. In this paper, the resulting reduced system will be referred to as the \textit{slow manifold approximation}. For \eqref{toy} the slow manifold approximation is the system:
\begin{align}
 u' = 1,\quad (x,y,u)\in S,\,u\in S^1,\eqlab{slowManifoldApproximation}
\end{align}
obtained from \eqref{reducedProblemToy} by restriction $(x,y)$ to the normally elliptic slow manifold $S$. Orbits of this system reach the boundaries of $S$ at $u=0,\,\tau$ where $S$ looses normally ellipticity. But if we extend the system \eqref{slowManifoldApproximation} to the closure of $S$ then we obtain closed orbits
\begin{align}
 y_s=0,\,\vert x_s(u)\vert &=\left\{\begin{array}{cc}
                        0& u\notin (0,\tau),\\
                        \kappa(u)&u\in (0,\tau),
                       \end{array}\right.,\eqlab{xs}
\end{align}
which we will refer to as \textit{singular closed orbits}. The periodic orbits considered in this paper remain close to the singular closed orbits in \eqref{xs}.

To describe the dynamics in 2-d.o.f. slow-fast Hamiltonian systems, one can often apply the theory of adiabatic invariants \cite{arn2}. To explain this theory, first note that the layer problem, where the slow variables are fixed as parameters, is an integrable 1-d.o.f. system. Within a region 
of closed trajectories it is therefore possible to 
introduce action-angle variables, even in the full system for $\epsilon>0$. Then, by averaging the Hamiltonian over the fast angle, one obtains a 1-d.o.f. system for the motion of the slow variables with the action appearing as a parameter. This is called the \textit{adiabatic approximation}. Suppose that the trajectories within the phase plane of slow variables obtained from this approximation are closed. Then for analytic systems the theory says that, in general, the action only perpetually undergoes small oscillations $\mathcal O(\epsilon)$ \cite{arn2,geller1}. The phase space is, up to small gaps, filled with invariant tori \cite{arn2} that are $\mathcal O(\epsilon)$-close to the tori obtained from the adiabatic approximation. 

A scenario, relevant to the problem considered here, where the theory of adiabatic invariants applies, is studied in \cite{geller2}. Here the action-angle variables exist as a result of an elliptic equilibrium of the fast sub-system. Such an equilibrium varies smoothly by the implicit function theorem with respect to the slow variables to form a normally elliptic {slow manifold} \cite{mac1}. 
The reference \cite{geller2} then assumes analyticity and that the reduced problem gives rise to singular closed orbits within the normally elliptic slow manifold and show under these generic conditions that the slow manifold approximation accurately describes dynamics of the true system for $\epsilon$ sufficiently small. The references \cite{Shatah2002572,malchiodi2012a,lu2013a} also consider periodic solutions in a normally elliptic singular perturbation setting. 
The theory from \cite{geller2} does not apply to \eqref{toy} since the singular closed orbits \eqref{xs} reach the boundaries of $S$ where it looses normally ellipticity. 


If there are symmetric separatrices on the phase plane of the fast variables described by the fast sub-system, such as in \figref{slowman} (b), then the theory of adiabatic invariants needs some further modification, see 
\cite{nei87,nei05} for details. In \cite{nei97,nei06,nei09} the authors presented very interesting results for a class of such systems, also including the class of systems considered here in Eq. \eqref{Ham} below. They showed that in such systems there is in an order of $\epsilon^{-1}$-many stable periodic orbits that repeatedly move from rotating within the separatrix lobes to rotating outside these lobes, see \figref{slowman} (b). Moreover, there is an order $1$ measure of regular motion; something that cannot be observed on Poincar\'e-sections: The resonance islands are small, being of order $\epsilon$, and are therefore unlikely to be visible on Poincar\'e sections. A consequence of these results is that in systems with separatrix crossings there exist an order $1$ set of initial conditions for which the action is perpetually invariant and the adiabatic approximation provides an accurate description of the full system. As opposed to the systems in \cite{arn2} without separatrix crossings, there is also an order $1$ set of chaotic dynamics where the adiabatic approximation does not give an accurate description of the dynamics. Crucial to the arguments in \cite{nei97,nei06,nei09}, however, is the condition that the crossings occur away from bifurcation points where the time scales are comparable. This condition was realized by taken the action to be greater than $c^{-1}>0$, $c$ large but independent of $\epsilon$. The main result of this paper partially uncovers what changes when we move close to such bifurcation points and thus 
investigating the adequateness of 
the slow manifold approximation in systems with slow manifold bifurcations. 


The geometric theory of singular perturbation provides another approach to the description of slow-fast systems. Although this theory is primarily applied to dissipative systems, focus being on normally hyperbolic slow manifolds, the view-point taken there is also relevant to mention in the context of this paper. This theory, also referred to as Fenichel's theory \cite{fen1,fen2}, says that normally hyperbolic critical manifolds perturb to invariant slow manifold for $\epsilon$ sufficiently small. The invariant slow manifold is smoothly $\epsilon$-close to the critical one. The flow restricted to the slow manifold converges to the flow of the reduced system, and the dynamics near the invariant slow manifold is, in some sense, inherited from the layer problem. Normally elliptic slow manifolds do not support such a general theory. However, the results of 
 Gelfreich and Lerman in \cite{geller2} show that, for general 2-d.o.f. analytic Hamiltonian slow-fast systems, the normally elliptic 
slow manifolds do in some sense also persist, potentially up to small gaps. The invariant slow manifold with gaps are filled with periodic orbits that are $\mathcal O(\epsilon)$-close to the periodic orbits obtained from the slow manifold approximation. One of the aims of this paper is to land somewhere in-between these two different results and approaches, \cite{nei97,nei06} and \cite{geller2}, addressing periodic orbits in systems with separatrix crossing, as \cite{nei97,nei06}, while on the other hand relating this to normally elliptic slow manifolds and the slow manifold approximation as in \cite{geller2}. 
 Almost invariant normally elliptic slow manifolds have been studied in \cite{Lu20114124,kri3}.
 
 
In \cite{kri1,kri2}, the authors studied different models of tethered satellites. One of these models is a finite-dimensional model that is obtained by replacing the tether connecting the satellite end-points with a spring that goes slack in compression. In \cite{kri2}, we showed for a Galerkin approximation of a more general PDE-model that such ``slack spring'' model, within this approximation, accurately describes the dynamics. In particular, it was shown that the motion remains close to the normally elliptic branches of a bifurcating slow manifold similar to the one shown in \figref{slowman} (a) for a long period of time. The bifurcation of the slow manifold arose as a result of a pitchfork bifurcation within the limiting fast system; i.e. the situation also considering here in this paper. We did not explore a more quantitative description of the dynamics near these objects. The main result in this paper, however, applies to the Galerkin model in \cite{kri2} (upon using the reduction of energy described in \secref{reduction} 
below) and thus provide this example with a 
more detailed description of the 
dynamics in the vicinity of the bifurcating normally elliptic slow manifolds.

The treatment of bifurcating slow manifolds in Hamiltonian systems has also received attention elsewhere. For example, the references \cite{fif1,schsou1,soufif1} considered the case of sub-critical pitchfork bifurcations in 2-d.o.f. Hamiltonian slow-fast systems. This was motivated by interfaces between ordered and disordered crystalline states. In particular, the references showed the persistence of singular heteroclinic
solutions connecting equilibria on the normally hyperbolic critical manifold before and after the bifurcation. It was also
shown that the heteroclinic connections remain close to the union of the normally hyperbolic
branches of the slow manifolds before and after the perturbation. In particular, on the passage through the bifurcation the time scales were comparable.  For this authors used Fenichel's theory and the blowup method of Krupa and Szmolyan \cite{sz1} to extend the normally hyperbolic slow manifolds near the bifurcation. The situation addressed
here is related to these results in the sense that this paper studies the existence of periodic orbits that
remain close to the union of the normally \textit{elliptic} branches of the slow manifold. 

\textbf{Summary of main results}. Before presenting the problem and stating the main results formally, the results of the paper are first described in words:
\begin{itemize}
\item[(a)] Close to the bifurcating normally elliptic slow manifold $S$ there exist many unstable periodic orbits but, if any, fewer stable orbits. The instability of the periodic orbits is in contrast to the stability of $S$ as stable critical points of the Hamiltonian fast sub-system.
 \item[(b)] The many unstable periodic orbits are $\mathcal O(\epsilon^{1/3})$-close to the bifurcating normally elliptic slow manifold. For these orbits the passage through $u=0$ and $u=\tau$ is described by a scaled system in which the time scales are comparable. 
 \item[(c)] Stable solutions $\mathcal O(\epsilon^{1/3})$-close to the bifurcating normally elliptic slow manifold can always be created (and destroyed) by varying the small parameter.
 \item[(d)] There are always stable periodic orbits further away from the bifurcating normally elliptic slow manifold at a distance of $\mathcal O(\epsilon^{1/3}\ln^{1/2} \ln \epsilon^{-1})$. The larger distance from the slow manifold ($\mathcal O(\epsilon^{1/3}\ln^{1/2} \ln \epsilon^{-1})$ vs. $\mathcal O(\epsilon^{1/3})$ in (a), (b) and (c)) manifests itself in the fact that in the scaled system, used to describe the passage through $u=0$ and $u=\tau$, the time scales are, in contrast to the orbits in (a), (b), and (c), not comparable: Close to the bifurcation points the fast variables undergo rapid oscillations. 
\end{itemize}

\textbf{Problem formulation and main result}. I consider the following class of $1\frac12$-d.o.f. slow-fast Hamiltonian systems:
\begin{align}
 H(u,v,x^2,y^2) &= v+\frac{1}{2}y^2(1+M(x^2,y^2,u)) - \frac12 f(u) x^2 \nonumber\\
 &+\frac12x^4(1+V(x^2,u)),\eqlab{Ham}\\
 \omega &=dx\wedge dy+\epsilon^{-1} du\wedge dv,\nonumber
 \end{align}
 with
 \begin{align}
  M(0,0,0)=0,\quad f(0)=0,\quad V(0,0)=0,\quad \text{and}\quad f'(0)=1,\eqlab{condMfV}
 \end{align}
 which gives rise to the following fast system of equations:
\begin{align}
 \dot x &=2y\partial_{y^2} H,\eqlab{eqn0}\\
 \dot y &=-2x\partial_{x^2} H,\nonumber\\
 \dot u&=\epsilon.\nonumber 
\end{align}
I will motivate this choice further in \secref{reduction} below. Here $\epsilon \ll 1$ and $u\in S^1=\mathbb R/(2\pi \mathbb Z)$.
The variable $v$ is the conjugate to the slowly varying phase $u\in S^1$ and is introduced merely for later convenience. The functions $M$, $f$ and $V$ are assumed to be smooth and to satisfy the following conditions:
\begin{itemize}
 \item[(A1)] $f(0)=0=f(\tau)$ for some $\tau <2\pi$;
 \item[(A2)] $f(u)>0$ for $u\in (0,\tau)$ and $f(u)<0$ for $u\in (\tau,2\pi)$;
 \item[(A3)] $f$, $M(x^2,y^2,\cdot)$, $V(x^2,\cdot)$, and hence $H$, are symmetric about $u=\tau/2$:
 \begin{align*}
  f(u)=f(\tau-u),\,M(x^2,y^2,u)=M(x^2,y^2,\tau-u),\,\,V(x^2,u)=V(x^2,\tau-u);
 \end{align*}
 \item[(A4)] The function $M$ also satisfies $1+M>0$. 
\end{itemize}
The assumption (A3) gives rise to the time-reversible symmetry $\mathcal T_{\tau}$ in \eqref{Ttau}: If $(x,y,u)=(x,y,u)(t)$ is a solution then so is \begin{align*}\mathcal T_{\tau}(x,y,u)(t)= (x,-y,\tau-u)(-t).\end{align*} Given the form of $M$ and $V$, an additional symmetry $\mathcal R$, see \eqref{reflection}, has been enforced: If $(x,y,u)=(x,y,u)(t)$ is a solution then so is $$\mathcal R(x,y,u)(t)\equiv (-x,-y,u)(t).$$ I will also think of $\mathcal R=-Id$ as the reflection acting on $(x,y)$ alone.

{The equilibrium $x=0=y$ of the fast sub-system, where $u$ appears as a parameter, undergoes a symmetric pitchfork bifurcation at $u=0$ and $u=\tau$. Cf. \eqref{condMfV} $V=\mathcal O(u+x^2)$ and therefore by the implicit function theorem there exists two additional equilibria solving $\partial_{x^2} H=0$ of the form $(x,y)=(\pm \kappa(u),0)$ for $u>0$ sufficiently small with 
%
\begin{align}
  x^2=\kappa(u)^2 = \frac{f(u)}{2}\left(1+\mathcal O(u)\right),\eqlab{kappa2expr}
 \end{align}
solving
 \begin{align}
 \partial_{x^2} H(u,v,x^2,0) = -\frac12 f(u)+x^2\left(1+V+\frac12 x^2\partial_{x^2} V\right)=0.\eqlab{kappaeqn}
 \end{align}
Note also that $\kappa(\tau-u)=\kappa(u)$ by (A3). We require that the solution $x^2=\kappa(u)^2$ exists and is positive for all $u\in (0,\tau)$, and ensure that $(x,y)=(\pm \kappa(u),0)$ do not undergo addition bifurcations, by adding the following assumption:
\begin{itemize}
 \item[(A5)] The solution $x^2=\kappa(u)^2$ of \eqref{kappaeqn} is smooth, positive and satisfies $$\partial_{x^2}^2 H(u,v,\kappa(u)^2,0) >0,$$ for all $u\in (0,\tau)$. 
\end{itemize}
I shall, without loss of generality, henceforth take $\kappa(u)>0$, $u\in (0,\tau)$. 
Note in (A5) that $$\partial_x^2 H (u,v,\kappa(u)^2,0) = 4\kappa(u)^2 \partial_{x^2}^2 H(u,v,\kappa(u)^2,0),$$ since $$\partial_{x^2} H(u,v,\kappa(u)^2,0)=0,$$ by construction. The condition (A5) therefore implies that $x=\pm \kappa(u),\,y=0$ are elliptic equilibria of the fast sub-system for every $u\in (0,\tau)$. I have illustrated the situation in \figref{slowman} (a) and (b). Together with $x=0,\,y=0$ for $u\notin (0,\tau)$ these equilibria give rise to the slow manifold $S$. Following (A2) and (A5) the set $S\cup \{x=0,\,y=0\}$ is an isolated connected component of critical points for \eqref{eqn0}. Also $S\backslash\{u=0,\,u=\tau\}$ is normally elliptic. 
%
}
Note, however, that $S$ is not uniformly elliptic due to the bifurcations at $u=0$ and $u=\tau$. 

This paper investigates periodic orbits that remain close to $S$ passing near the bifurcation points at $u=0$ and $u=\tau$. To explain this differently: The reduced problem
 has a \textit{singular closed orbit}, see \eqref{xs}.
The periodic orbits of this paper lie close to the singular closed orbit in \eqref{xs} in the sense that $\vert \vert x(u)\vert-\vert x_s(u)\vert\vert+\vert y(u)\vert$ is small with respect to $\epsilon$. There is another singular orbit $x=0=y,\,u\in S^1$ which enters the \textit{normally hyperbolic} part: $x=0,\,y=0,\,u\in (0,\tau)$ of the slow manifold. Due to the $\mathcal R$-symmetry this is in fact a true orbit for all $\epsilon$ and it is cf. e.g. \cite{nei09} typically highly unstable with multipliers of order $e^{\pm \mathcal O(\epsilon^{-1})}$. 



When I later perform some numerical investigations I will base these on the example \eqref{toy} 
where $f(u)=\sin u$, $M=0=V$, and $\tau=\pi$. 


\section{Reduction to \eqref{Ham} from a 2-d.o.f. system}\seclab{reduction}
{One of the ways to obtain \eqref{Ham} from a more general setting, is to start from a natural slow-fast 2-d.o.f. system of the form
\begin{align}
 K&=\frac12 m_f(w,x^2)y^2+\frac12m_s(w) z^2+W(w)+x^2Q(w,x^2),\eqlab{systemKK}\\
 \omega &=dx\wedge dy+\epsilon^{-1} dw\wedge dz,\nonumber
\end{align}
with $m_f,\,m_s>0$, possessing a family of generic periodic orbits within the fix-point set $\{x=0=y\}$ of the symmetry action $(x,y)\mapsto \mathcal R(x,y)= (-x,-y)$, and reduce to a $1\frac12$-d.o.f. system by reduction of energy. Assume the following: 
\begin{itemize}
\item[(C1)] There exists a $w_0$ so that $Q(w_0,0)=0,\, \partial_{w} Q(w_0,0)<0,\,\partial_{x^2}Q(w_0,0)> 0$;
\item[(C2)] The section $\{w=w_0\}$ is transverse to the family of periodic orbits.
\end{itemize}
It follows from (C1) that the fast sub-system:
\begin{align*}
 \dot x &=m_fy,\\
 \dot y&=-2x (Q+x^2\partial_{x^2} Q+\frac12 \partial_{x^2} m_f y^2),
\end{align*}
with $w$ fixed as a parameter, 
undergoes a super-critical pitchfork bifurcation at $w=w_0$ of $x=0,y=0$ and furthermore that there exists a solution $x^2 = \kappa(w)^2>0$ of $\partial_{x^2}Q(w,x^2)=0$ within $w\in (w_0,w_1)$ for some $w_1>w_0$ where $\partial_{x^2}^2 K(w,z,\kappa(w)^2,0)>0$. 
Consider then the reduced problem on $x^2=\kappa(w)^2$:
\begin{align}
K(w,z,\kappa(w)^2,0)&=\frac12m_s(w) z^2+W(w)+\kappa(w)^2Q(\kappa(w)^2,w), \eqlab{systemKReducedProblem}\\
\omega &=dw\wedge dz.\nonumber
\end{align}
We assume the following:
\begin{itemize}
 \item[(C3)] Initial conditions on the half-section $w=w_0,\,z>0$ return to $w=w_0$ with $z<0$ under the forward flow of \eqref{systemKReducedProblem}.
\end{itemize}
This condition will allow us to show that the energy reduced system satisfies assumption (A5) above.

Within the region of closed orbits in $\{x=0=y\}$ we can replace $(w,z)$ by action-angle variables $(\phi,I)\in S^1\times \mathbb R$ and write 
\begin{align}
K(\phi,I,x^2,y^2)&=k(I)+\frac12 m_f(\bar w(\phi,I),x^2)y^2+x^2Q(\bar w(\phi,I),x^2),\eqlab{systemKI}\\
\omega &=dx\wedge dy+\epsilon^{-1}d\phi\wedge dI.\nonumber
\end{align}
Here $k'(I)>0$, $\phi=0$ corresponds to $z=0$ with $w>w_0$, and the orientation of $\phi$ is positive in the clockwise direction. Also $\bar w(\phi,I)=w$. It follows that 
\begin{align}
\partial_\phi \bar w(\phi,I)>0,\eqlab{barwCondition}
\end{align}
for $\phi \in (0,\pi)$. Moreover, from (C3) it can be deduced that $$\dot \phi>0,$$ for the reduced problem \eqref{systemKReducedProblem}. The invariance of $K$ in \eqref{systemKK} with respect to $(w,z,x,y)\mapsto (w,-z,x,-y)$ then becomes an invariance of $K$ in \eqref{systemKI} with respect to 
\begin{align}
(\phi,I,x,y)\mapsto (-\phi,I,x,-y).\eqlab{newSymmetry}
\end{align}
In particular $\bar w(-\phi,I)=\bar w(\phi,I)$. 

\begin{prop}\proplab{reduction}
 Suppose conditions (C1)-(C3). Then by reduction of energy, and appropriate scalings, the system \eqref{systemKI} can be brought into \eqref{Ham} satisfying the conditions (A1)-(A5).
\end{prop}
\begin{proof}
%
To reduce by energy we solve 
\begin{align}
K(\phi,I, x^2,y^2)=e,\eqlab{KeqE}
\end{align}
for 
\begin{align}
I=\bar I(\phi,e,x^2,y^2).\eqlab{barI} 
\end{align}
The function $\bar I$ is invariant with respect to $\phi\mapsto -\phi$ cf. \eqref{newSymmetry}. First, however, we shift the angle $\phi$ so that the origin is based at the bifurcation point $w=w_0,\,x=0,\,y=0$ where $I=k^{-1}(e)$ cf. \eqsref{systemKI}{KeqE}. Let therefore $\tau/2>0$ be so that 
\begin{align}
\bar w(-\tau/2,k^{-1}(e))=w_0,\,x=0,\,y=0,\eqlab{barw}
\end{align}
and set
\begin{align*}
 \phi = -\tau/2+u.
\end{align*}
Here $\tau=\tau(e)<2\pi$. 
For simplicity, we henceforth continue to denote 
\begin{align}
\bar w(-\tau/2+u,I), \,\, K(-\tau/2+u,I,x^2,y^2),\,\,\,\text{and} \,\,\, \bar I(-\tau/2+u,e,x^2,y^2),\eqlab{barw1}
\end{align}
by 
\begin{align}
\bar w(u,I),\quad K(u,I,x^2,y^2),\quad \text{and}\quad \bar I(u,e,x^2,y^2),\eqlab{barw2}
\end{align} 
respectively. The mapping $\phi\mapsto -\phi$ becomes 
\begin{align}
u\mapsto \tau-u,\eqlab{uTau}
\end{align}
and the resulting function $\bar I(u,e,x^2,y^2)$ is therefore invariant with respect to \eqref{uTau}. 

By Taylor expansion of \eqref{KeqE} at $x=0,\,y=0$ we obtain the following local form of $\bar I$:
\begin{align}
 \bar I(u,e,x^2,y^2) &= k^{-1}(e) - (k^{-1})'(e)\bigg( \frac12 y^2(m_f(w,0)+\mathcal O(x^2+y^2))  +Qx^2 \nonumber\\
 &+\left[\partial_{x^2} Q -(k^{-1})'(e) \partial_I  \bar w\partial_{w} Q Q -\frac{(k^{-1})''(e)}{(k^{-1})'(e)} Q\right] x^4+\mathcal O(x^6)\bigg).\eqlab{Iexpr}
 \end{align}
where $Q$ and its partial derivatives are evaluated at $(w,x^2)=(\bar w(u,k^{-1}(e)),0)$. By \eqref{barw} we have that $\bar w(0,k^{-1}(e))=w_0$, using \eqsref{barw1}{barw2}, and by assumption (C1) it therefore follows that $$Q(\bar w(0,k^{-1}(e),0)= 0, \,\partial_w Q(\bar w(0,k^{-1}(e),0)< 0,\, \text{and} \, \partial_{x^2} Q(\bar w(0,k^{-1}(e),0)> 0.$$Since $\dot u = \epsilon \partial_I K$ we obtain:
\begin{align}
 \frac{d x}{dt}& = \partial_y (-\bar I),\eqlab{systemBarI}\\
 \frac{d y}{dt}&= -\partial_x (-\bar I),\nonumber\\
 \frac{du}{dt}&=\epsilon,\nonumber
\end{align}
by implicit differentiation of $K(u,\bar I(u,e,x^2,y^2),x^2,y^2)=e$ and introduction of a new time: $\epsilon^{-1}u$. The system \eqref{systemBarI} is a $1\frac12$-d.o.f. Hamiltonian system with Hamiltonian function 
\begin{align}
H=-\bar I(u,e,x^2,y^2).\eqlab{HEqBarI}
\end{align}
By the invariance of $\bar I$ with respect to \eqref{uTau}, the system \eqref{HEqBarI} possesses the time-reversible symmetry $\mathcal T_\tau$ in \eqref{Ttau}.

We now transform \eqref{HEqBarI} into \eqref{Ham}. To do this we first introduce $\tilde \epsilon$ by $\epsilon = (k^{-1})'(e) \tilde \epsilon$, rescale time $t=(k^{-1})'(e) \tilde t$ and introduce $v$ conjugate to $u$ so that from \eqsref{Iexpr}{HEqBarI}
\begin{align*}
 H(u,v,x^2,y^2) &= v+\frac12 y^2(M_{0}+\mathcal O(u+x^2+y^2)) \\
 &-u\left(f_0+\mathcal O(u)\right)x^2+x^4\left(V_0+\mathcal O(u+ x^2)\right),\\
 \omega &=dx\wedge dy+\tilde \epsilon^{-1} du\wedge dv,
\end{align*}
where
\begin{align*}
 M_{0}=m_f(\bar w(0,k^{-1}(e),0),\quad  f_0  =-\partial_w Q(\bar w(0,k^{-1}(e)),0)\partial_u \bar w(0,k^{-1}(e)),
\end{align*}
and
\begin{align*}
 V_0=\partial_{x^2} Q(\bar w(0,k^{-1}(e)),0).
\end{align*}
These constants are positive by assumptions (C1), (C2) and \eqref{barwCondition}. We henceforth drop the tildes. 
We then scale $x$, $y$, $t$ and $\epsilon$ as follows:
\begin{align*}
 x = \sqrt{\frac{f_0}{2V_0}} \tilde x,\,y = \frac{f_0}{\sqrt{2V_0M_0}} \tilde y,\,t=\frac{1}{\sqrt{f_0M_0}} \tilde t,\,\epsilon = \sqrt{f_0M_0}\tilde \epsilon,
\end{align*}
and obtain the final Hamiltonian
\begin{align*}
 H(u,v,x^2,y^2) &= v+\frac12 y^2(1+\underbrace{\mathcal O(u+x^2+y^2)}_{=M(u,x^2,y^2)}) -\underbrace{u\left(1+\mathcal O(u)\right)}_{=f(u)}x^2\\
 &+\frac12 x^4\left(1+\underbrace{\mathcal O(u+ x^2)}_{=V(u,x^2)}\right),\quad 
 \omega =dx\wedge dy+\epsilon^{-1}du\wedge dv,
 \end{align*}
upon again dropping the tildes, satisfying the conditions (A1)-(A4). Condition (A5) follows from (C3). This completes the proof. 
\qed
\end{proof}
The system in \cite[Eq. (1.1.)]{kri2} with $(u,U,v,V)=(w,z,x,y)$ and $w_0=1$ is of the form \eqref{systemKK} and satisfies the conditions (C1), (C2) and (C3). The system can therefore be brought into \eqref{Ham} and hence the main result of this paper applies to this problem. }
\section{Main result}
Following \cite{geller2} we say that periodic orbits of \eqref{Ham} are \textit{long} if they have periods of order $\mathcal O(\epsilon^{-1})$. 

The main result of the paper is the following one:
\begin{mainres}
There exists an $\epsilon_0>0$ so that the following holds true for $\epsilon\le \epsilon_0$:
\begin{itemize}
 \item[$1^\circ$] There exists an order of $\ln^2\epsilon^{-1}$-many unstable, and \textnormal{long} periodic orbits of \eqref{eqn0} where $(x,y)=(x(u),y(u))$ remain $\mathcal O(\epsilon^{1/3})$-close to the union of the normally elliptic critical manifold. Moreover, an order of $\ln^2\epsilon^{-1}$ of these orbits are symmetric with respect to $\mathcal R$ and/or $\mathcal T_{\tau}$. 
 The characteristic multipliers of the periodic orbits are $\mathcal O(\ln^{\pm 2} \epsilon^{-1})$. 
 \item[$2^\circ$] For every $\epsilon\le \epsilon_0$ there are fewer (\textnormal{typically} with an order $o(\ln \epsilon^{-1})$), if any at all, stable periodic orbits of the type considered in $1^\circ$ than unstable ones. 
\item[$3^\circ$] Take any $\epsilon_1< \epsilon_0$ and consider $\epsilon\in I_1 = [\epsilon_1-c_1\epsilon_1^2,\epsilon_1+c_1\epsilon_1^2]\subset (0,\epsilon_0]$. Then within $I_1$ there will exist $\lfloor c_2^{-1}\ln \epsilon_1^{-1}\rfloor$-many closed intervals, each of length $\ge c_3^{-1} \epsilon_1^2 \ln^{-1}\epsilon_1$, of $\epsilon$-values for which there exists at least one stable solution $(x,y)=(x(u),y(u))$ remaining $\mathcal O(\epsilon^{1/3})$-close to the union of the normally elliptic critical manifold. Here $c_1$, $c_2$ and $c_3$ may be large but they can be taken to be independent of $\epsilon_1$. 
\item[$4^\circ$] There exist stable, and long periodic orbits, symmetric with respect to $\mathcal R$ and/or $\mathcal T_{\tau}$, where $(x,y)=(x(u),y(u))$ remain $\mathcal O(\epsilon^{1/3}\ln^{1/2}\ln \epsilon^{-1})$-close to the union of the normally elliptic critical manifold. 

\end{itemize}
 
\end{mainres}

\begin{remark}\remlab{mainres}
Regarding $1^\circ$: The periodic orbits are \textit{long} since their periods are large being either $2\pi \epsilon^{-1}$ or $4\pi \epsilon^{-1}$. As explained in (b) above, the passage of these orbits through $u=0$ and $u=\tau$ is described in a scaled system in which the time scales are comparable. 

Regarding $2^\circ$: {I have only been able to prove existence of stable orbits upon variation of $\epsilon$ (see $3^\circ$). In fact numerical computations on \eqref{toy} seem to suggest that both existence and non-existence of stable periodic orbits of the type in $1^\circ$ can occur.}

Regarding $3^\circ$: The stable solutions within the intervals of $3^\circ$ are distinct but the intervals could potentially overlap giving rise to several solutions for fixed values of $\epsilon$. Following $2^\circ$, however, the number of potential overlaps are typically $o(\ln \epsilon^{-1})$. 
Further details can be found in \secref{part3}.

Regarding $4^\circ$: The orbits in $2^\circ$ are rare, which we further demonstrate by performing some numerics on example \eqref{toy}, and the stable orbits in $4^\circ$ are, in this sense, \textit{typically} the smallest ones. These orbits are different from those in $1^\circ$ in that they are stable but, moreover, their distance to the normally elliptic critical manifold is also larger: $O(\epsilon^{1/3}\ln^{1/2}\ln \epsilon^{-1})$ vs. $\mathcal O(\epsilon^{1/3})$. In contrast to the orbits in $1^\circ$, this also means that the orbits undergo fast oscillations in the passage through the bifurcation points $u\in \{0,\tau\}$.

The minority of stable periodic orbits is in agreement with the results in \cite{nei97,nei06,nei09} where $I\ge c^{-1}>0$. These references show that there are asymptotically as many stable orbits as there are unstable ones. But, nevertheless, the unstable orbits are more frequent than the stable ones. See \tabref{tbl} in \secref{numerics} for the results of numerical computations of periodic orbits for \eqref{toy} and \cite[Table 1]{nei09} where they authors for an example of motion of charged particles in the Earth's magnetotail find $23865$ unstable orbits and in comparison only $370$ stable orbits. Furthermore, cf. Eq. (41) in \cite{nei09}, decreasing the action has the effect of diminishing the stability region.
\xqed{\lozenge}\end{remark}

\begin{remark}
 Assumption (A3) could be relaxed: It is primarily included by convenience rather than necessity. The problems, such as the one in \cite{kri2}, that arise by the reduction in \secref{reduction} do, however, satisfy this condition and I therefore found it natural to exploit this.
\xqed{\lozenge}\end{remark}

%
\begin{remark}
Each stable periodic orbit will in general give rise to stability islands. This was also the subject of interest in \cite{nei06} showing that in the general case there are $\mathcal O(\epsilon^{-1})$-many stability islands of measuring at least $c^{-1}\epsilon$ for some $c$ large but independent of $\epsilon$. The islands due to the stable periodic orbits in $3^\circ$ and $4^\circ$ are expected to be smaller measuring $\mathcal O(\epsilon \ln^{-3}\epsilon^{-1})$. I will also discuss this further in \secref{resisl}.
\xqed{\lozenge}\end{remark}

\textbf{Outline}. The main result is proved by obtaining fix points of a return map 
\begin{align}
P: \{(x,y,u)\vert u&=-\pi+\tau/2\}\rightarrow \{(x,y,u)\vert u=\pi+\tau/2\},\eqlab{P0}\\
(x,y,u=-\pi+\tau/2)&\mapsto P(x,y,u)=\phi_{2\pi/\epsilon}(x,y,-\pi+\tau/2),\nonumber
\end{align}
$\phi_t$ being the flow of \eqref{eqn0}. The return map is approximated using averaging and asymptotics of the second Painlev\'e equation. The averaging principle is applied to two different blowup systems, one focusing in on $x=0=y$ for $u\in (-\pi+\tau/2,-u_*)\cup (\tau+u_*,\pi+\tau/2)$, with $u_*$ small, and one focusing in on $x=\pm \kappa(u),y=0$ for $u\in (u_*,\tau-u_*)$. \secref{blowup} describes the blowup transformations used. These blowups are appropriate $\epsilon$-dependent scalings. 
To describe the transition from $u = -u_*$ to $u=u_*$ I make use of the fact that \eqref{Ham}, in a certain sense, is close to the second Painlev\'e equation, where there exists known asymptotics \cite{pl2}. I present this asymptotics in \secref{blowup} and show how it can be applied to \eqref{Ham}. The small number $u_*$ is written as $\mu^2 \hat u_*$ with $\mu$ small and connected to $\epsilon$. The number $\hat u_*$ is then fixed as $\mathcal O(1)$ with respect to $\epsilon$. By connecting the $\mu$ with the 
blowups used in 
the fast space in the two separate regimes, I can obtain a lower bound ($\gg \epsilon$) of the normal frequency for $u\notin (-u_*,u_*)\cup (\tau-u_*,\tau+u_*)$. Such lower bound is crucial to successfully apply the averaging principle. The result, however, cannot be obtain by scalings alone; it is important to make use of the fact that the time-scale separation enhances as $u$ moves away from $0$ and $\tau$ to accurately approximate the return map. The averaging part is presented in \secref{returnmap} where I also describe the return map $P$ in further details. 
Finally I solve the fix point equations and prove the main result in \secref{solving}. Here I also perform some numerical investigations on the example \eqref{toy}.



\textbf{Notation}. To prove the result I make use of certain transformations of the fast variables for $u<-u_*$ and $u>u_*$. For this I will follow the following convention: Variables and functions relevant to the regime $u\in (-\pi+\tau/2,-u_*)\cup (\tau+u_*,\pi+\tau/2)$ are denoted by Roman letters, whereas Greek letters are used within $u\in (u_*,\tau-u_*)$. Scaled variables are given a hat: $\hat {()}$. Some variables will be given subscripts starting from $0$ to indicate that they later will be updated as a result of near-identity transformations. As always in work like this there will be introduced an abundance of constants. I will use $c$ with and without subscripts for constants. They will always be independent of $\epsilon$ but I will often need them to be large. When I need a small constant I will therefore write it as $c^{-1}$. To avoid a long enumeration of constants, I will often restart an numeration at the beginning of a new section, a lemma, or even a new paragraph. It should be clear from the context where constants are related.

\section{Blowup}\seclab{blowup}
In reference \cite{nei99} the authors use the following blowup
\begin{align}
 x = \epsilon^{1/3}\breve x,\,y=\epsilon^{2/3}\breve y,\,u=\epsilon^{2/3}\breve u,\eqlab{blowup0}
\end{align}
to reduce \eqref{eqn0} with $M=0=V$ to the Painleve equation of second kind:
\begin{align}
 \frac{d\breve x}{d\breve u} &=\breve y,\eqlab{ple}\\
 \frac{d\breve y}{d\breve u}&=\breve u\breve x-2\breve x^3,\nonumber 
\end{align}
ignoring here higher order terms that come from the expansion of $f=f(u)$ about $u=0$. The reference presents asymptotics from \cite{pl2} of \eqref{ple} for $\pm \breve u$ large that I will also make use of here. The asymptotics show that $\breve x=\mathcal O(\vert \breve u\vert^{-1/4})$, $\breve y=\mathcal O(\vert \breve u\vert^{1/4})$ for $\breve u\rightarrow -\infty$. For $\breve u\rightarrow \infty$, on the other hand, they show that $\breve x\mp \sqrt{\hat u/2}=\mathcal O(\hat u^{-1/4})$ and $\breve y = \mathcal O(\hat u^{1/4})$. This motivates the following blowup 
\begin{align}
\breve x&=\delta^{1/4}  \hat x,\eqlab{brevex}\\
\breve y&= \delta^{-1/4} \hat y,\nonumber 
\end{align}
when taking $\breve u=\delta^{-1}\hat u<0$. Then the asymptotics for $u<0$ can be invoked by letting $\delta \rightarrow 0^+$. For $u>0$ but small I will use an identical blowup of the deviation from $\pm \kappa(u)=\pm \sqrt{u/2}+\mathcal O(u^{3/2})$.

\subsection{Blowup for $u\in [-(\pi-\tau/2),0)$}\seclab{blowupuL0}
Let $\mu=\epsilon^{1/3}\delta^{-1/2}$. Then motivated by the presentation above, in particular by \eqsref{blowup0}{brevex}, I introduce the following scaled variables $\hat x,\,\hat y$ and $\hat u$:
\begin{align}
x&=\epsilon^{1/3}\delta^{1/4} \hat x = \mu \delta^{3/4}\hat x,\eqlab{blowup}\\
y &= \epsilon^{2/3}\delta^{-1/4} \hat y = \mu^2\delta^{3/4} \hat y,\nonumber\\
u&=\epsilon^{2/3}\delta^{-1} \hat u= \mu^{2}\hat u.\nonumber
\end{align}
{It proves useful to write $M$ and $V$ as
\begin{align*}
 M(x^2,y^2,u) &=uM_{00}(u)+x^2M_{10}(u)+y^2M_{01}(u)+\mathcal O((x^2+y^2)^2),\\
 V(x^2,u)&=uV_0(u)+\mathcal O(x^2),
\end{align*}
where the functions $M_{00}$, $M_{10}$, $M_{01}$, and $V_0$ are defined by Taylor's theorem applied to $M$ and $V$ at $(x^2,y^2)=0$ and $x^2=0$, respectively.} Inserting \eqref{blowup} into \eqref{Ham} then gives
\begin{align*}
H &=v+\frac12 \mu^4 \delta^{3/2} \left(\hat y^2(1+uM_{00}(u))+\hat F(\hat u)^2(1+uM_{00}(u))^{-1} \hat x^2\right) \\
&+\frac12 \mu^4 \delta^3 \hat x^4\left(1+uV_0(u)+\mathcal O(\mu^2 \delta^{3/2} \hat x^2)\right)+\frac12 \mu^4 \delta^{3/2} \hat y^2 \bigg(\mu^2 \delta^{3/2} \hat x^2 M_{10}(u)\\
&+\mu^4 \delta^{3/2} \hat y^2M_{01}(u)+\mathcal O(\mu^4\delta^3 \hat x^4 +\mu^8\delta^3 \hat y^4+\mu^6 \delta^3 \hat x^2 \hat y^2)\bigg),\\
 \omega &=\epsilon d\hat x\wedge d\hat y+\epsilon^{-1/3}\delta d\hat u\wedge dv.
 \end{align*}
I have here introduced a scaled frequency $\hat F$ defined by
 \begin{align}
  \hat F(\hat u)^2 = -\mu^{-2} f(\mu^2 \hat u)(1+uM_{00}(u))=-\hat u + \mathcal O(\mu^2).\eqlab{hatOmega}
 \end{align}
 At this stage I think of $\delta$ being small, allowing me to invoke the asymptotics of \eqref{ple}, but it can not be too small as we will need $\mu\ll 1$ and hence
 \begin{align}
 \epsilon\ll \delta.\nonumber
 \end{align}
 We will later need to quantify $\delta$ in relation to $\epsilon$ more accurately. In fact this will be one of the main difficulties. 
 
 Next, I introduce $v=\mu^4 \delta^{3/2}\hat v$ and divide $H$ by $\mu^4\delta^{3/2}$ to obtain a blowup Hamiltonian system $\hat H =\mu^{-4} \delta^{-3/2} H$:
\begin{align}
 \hat H &=h(\hat u,\hat v,\hat x,\hat y)+r(\hat u,\hat x,\hat y)+ \mathcal O\left(\hat y^2 \left(\mu^4 \delta^3\hat x^4+\mu^8\delta^3\hat y^6+\mu^6\delta^3\hat x^2\hat y^2\right)\right),\eqlab{hatH}\\
 \hat \omega &= \mu^{-4} \delta^{-3/2} \omega = \mu^{-1} d\hat x\wedge d\hat y+ \mu^{-1}\delta^{-3/2} d\hat u\wedge d\hat v,\nonumber 
 \end{align}
 introducing a leading order term:
 \begin{align}
  h(\hat u,\hat v,\hat x,\hat y) = \hat v+\frac12 \hat y^2(1+uM_{00}(u))+\frac{\hat F(\hat u)^2}{2}(1+uM_{00}(u))^{-1} \hat x^2,\eqlab{hHam}
 \end{align}
and a remainder:
\begin{align*}
 r (\hat u,\hat x,\hat y) =  \delta^{3/2} \left(\frac12 \hat x^4(1+uV_0(u))+\frac12 \hat y^2 (\mu^2 \hat x^2M_{10}(u)+\mu^4 \hat y^2 M_{01}(u))\right).
\end{align*}
Recall $u=\mu^2 \hat u$. 
The system \eqref{hatH} gives rise to the following equations of motions
\begin{align*}
 \dot{\hat x}&=\mu \left( \hat y(1+uM_{00}(u))+\mathcal O(\delta^{3/2} \mu^2 \hat y( \hat x^2 +\hat y^2 \mu^2))\right),\\
 \dot{\hat y}&=-\mu \left(\hat F(\hat u)^2(1+uM_0(u))^{-1}\hat x+2\delta^{3/2} \hat x^3(1+uV_0(u))+\mathcal O(\delta^3 \mu^4\hat x (\hat x^2+\mu^2 \hat y^2))\right), \\
 \dot{\hat u}&=\mu\delta^{3/2}. 
\end{align*}
The truncation of the following representation of these equations
\begin{align}
 \frac{d\hat x}{d\hat u} &= \delta^{-3/2} \hat y+ \mathcal O(\mu^2),\eqlab{blowupi}\\
 \frac{d\hat y}{d\hat u} &=\delta^{-3/2} \left(\hat u\hat x-2\delta^{3/2}  \hat x^3\right)+\mathcal O(\mu^2\delta^{-3/2}\hat x),\nonumber
\end{align}
using here \eqsref{blowup}{hatOmega},
coincide with the result of applying the scaling $\breve x=\delta^{1/4}\hat x$, $\breve y=\delta^{-1/4}\hat y$, $\breve u=\delta^{-1}\hat u$ to the Painleve equations \eqref{ple} considered in \cite{nei99}.

\textbf{Action-angle variables}. For $u<0$ the function $h$ in \eqref{hHam} is brought into action angle variables 
by introducing $\hat x_0$ and $\hat y_0$ through
\begin{align}
 \hat x = \hat F(\hat u)^{-1/2}(1+uM_{00}(u))^{1/2} \hat x_0,\,\hat y = \hat F(\hat u)^{1/2}(1+uM_{00}(u))^{-1/2} \hat y_0.\eqlab{hatx0y0}
\end{align}
Then 
\begin{align}
 h_0\left(\hat u,\hat v,\frac12 (\hat x_0^2+\hat y_0^2)\right) \equiv h(\hat u,\hat v,\hat x,\hat y)=\hat v + \hat F(\hat u) \frac12(  \hat y_0^2+\hat x_0^2).\eqlab{h0Ham}
\end{align}
This transformation is lifted to a symplectic transformation $(\hat x,\hat y,\hat u,\hat v)\mapsto (\hat x_0,\hat y_0,\hat u_0,\hat v_0) $ on the full space via the generating function
\begin{align*}
 G(\hat x,\hat y_0,\hat u,\hat v_0) = \delta^{-3/2} \hat u \hat v_0 + \hat F(\hat u)^{1/2}(1+uM_{00}(u))^{-1/2}\hat x \hat y_0,
\end{align*}
and the equations
\begin{align*}
 \hat x_0 = \partial_{\hat y_0} G,&\quad \hat y = \partial_{\hat x} G,\\ \hat u_0 = \delta^{3/2} \partial_{\hat v_0} G = \hat u,&\quad \hat v =  \delta^{3/2} \partial_{ \hat u} G=\hat v_0-\frac12 \hat F(\hat u)^{-2}(1+\mathcal O(u)).
\end{align*}
Here I have in the last equality used the result from differentiating \eqref{hatOmega} with respect to $\hat u$. 
Applying this transformation to \eqref{hatH} gives
\begin{align}
 \hat H &= h_0(\hat u,\hat v_0,\hat z_0)+r_0(\hat u,\hat x_0,\hat y_0)+\mathcal O(\hat F^{-5}\delta^3 ),\eqlab{hatH1} 
\end{align}
with $h_0$ the integrable part in \eqref{h0Ham} and a remainder
\begin{align}
r_0(\hat u,\hat x_0,\hat y_0) &= r(\hat u,\hat x,\hat y)+\hat v-\hat v_0\nonumber\\
&=\frac12 \delta^{3/2} \hat F(\hat u)^{-2} \bigg(\hat x_0^4 (1+uV_0(u))\nonumber\\
&+\hat y_0^2 f(u)(1+uM_{00}(u))\left(-M_{10}(u)\hat x_0^2 + f(u)(1+uM_{00}(u))M_{01}(u)\hat y_0^2\right)\nonumber\\
&-\hat x_0\hat y_0(1+\mathcal O(u))\bigg),\eqlab{r0remainder}
\end{align}
and where the action-angle variables $(\hat z_0,w_0)$ are the symplectic polar coordinates of $(\hat x_0,\hat y_0)$:
\begin{align}
 \hat x_0 = \sqrt{2\hat z_0}\cos w_0,\, \hat y_0 = \sqrt{2\hat z_0}\sin w_0.\eqlab{w0hatz0}
\end{align}
\begin{remark}
In \eqsref{hatH1}{r0remainder} I have used \eqref{hatOmega} to eliminate $\mu$ and write these expresions only in terms of $\hat F$ and $\delta$. Note also how I in these expressions mix $u$ and $\hat u=\mu^{-2} u$ together. The reason for introducing $\hat u$ is that when $\hat u\le -\hat u_*=\mathcal O(1)$ then this gives an order $1$ lower bound of the scaled frequency:
\begin{align*}
 \hat F(\hat u)\ge \hat F(\hat u_*),
\end{align*}
provided $\mu$ is sufficiently small. The upper bound $\hat F(\hat u)$ is of order $\mu(\epsilon)^{-1}$ and thus unbounded as $\epsilon\rightarrow 0$. Keeping track of how $\hat F$ enters will be crucial when I am to decide what terms are important when I later wish to approximate the solution of these equations by means of averaging. On the other hand, I keep $M_{00}(u), M_{10}(u),\,M_{01}(u),$ and $V_0(u)$, for example, in terms of $u$ to highlight that its estimate for $u$ small is not particularly important. Instead it is important to highlight these terms are smooth and uniformly order $1$ for all $u\le 0$.

Big-Oh terms $\mathcal O(\hat F^{-q}\delta^p)$, like $\mathcal O(\hat F^{-5}\delta^3)$ in \eqref{hatH1}, are of order $\hat F^{-q}\delta^p$, $q,p>0$ in the sense that such terms can be bounded \textit{point-wise in $\hat u$} from above by $c\hat F(\hat u)^{-q}\delta^p$, $c$ independent of $\delta$ and $\hat u$, for all $\hat u\le -\hat u_*$ and $\delta$ sufficiently small. I will return later to how different $\hat F^{-q}\delta^p$ can be compared.
\end{remark}
\begin{remark}\remlab{remx}
It will later be shown that the action $\hat z_0$ only undergoes small oscillations. It will from this follow that
\begin{align*}
 x &= \mu \delta^{3/4} \hat x = \mu \delta^{3/4}(-\mu^{-2} f(u))^{-1/4} \sqrt{2\hat z_0} \cos w_0\\
 &=\epsilon^{1/2}(-f(u))^{-1/4} \sqrt{2\hat z_0} \cos w_0 = \mathcal O(\epsilon^{1/2})
\end{align*}
using \eqref{hatOmega}, \eqsref{hatx0y0}{w0hatz0} for $u\ll 0$. On the other hand when $u$ is such that $\hat F=\mathcal O(1)$ then $x=\mathcal O(\mu \delta^{3/4}) = \mathcal O(\epsilon^{1/3}\delta^{1/4})$.
 \xqed{\lozenge}\end{remark}

 \subsection{Blowup for $u\in (0,\tau/2]$} 
 To present the asymptotics in \cite{nei99} for the truncation of \eqref{blowupi}:
\begin{align}
 \frac{d\hat x}{d\hat u} &= \delta^{-3/2} \hat y,\eqlab{blowupit}\\
 \frac{d\hat y}{d\hat u} &=\delta^{-3/2} \left(\hat u\hat x-2\delta^{3/2}  \hat x^3\right),\nonumber
\end{align}
it is useful to introduce different blowup variables for $u\in (0,\tau/2)$ which I, as promised, will denote by Greek letters: $(\hat \xi,\hat \sigma)$. They are obtained by performing the blowup \eqref{blowup} to the deviation $(\pm \xi,\pm \sigma)$ from $(x,y)=(\pm \kappa(u),0)$:
\begin{align}
 (x,y) &= (\pm \kappa(u),0)+(\pm \xi,\pm \sigma),\eqlab{barxy}\\
 (\xi,\sigma) &=(\mu \delta^{3/4}\hat \xi,\mu^2 \delta^{3/4}\hat \sigma).\nonumber
\end{align}
The particular form is based on the equivariance of the equations with respect to the action of $\mathcal R$.
I also introduce symplectic polar coordinates $(\hat \varrho_0,\phi_0)$ by setting
\begin{align*}
\hat \xi_0& \equiv\sqrt{2\hat \varrho_0}\cos \phi_0=\hat \Omega^{1/2}(1+\kappa(u)^2M_0(u))^{-1/2}\hat \xi,\\
\hat \sigma_0& \equiv \sqrt{2\hat \varrho_0}\sin \phi_0= \hat \Omega^{-1/2}(1+\kappa(u)^2M_0(u))^{1/2}\hat \sigma,
\end{align*}
much as above. Here 
%
 \begin{align}
 \hat \Omega(\hat u)^2 &= 2\mu^{-2} \vartheta(u)=2\hat u(1+\mathcal O(u)),\eqlab{tildeOmega}
\end{align}
is a scaled frequency with
\begin{align}
\vartheta(u) &=2 \kappa(u)^2\partial_{x^2}^2 H(\kappa(u)^2,0,u)(1+uM^{00}(u))^{-1}=u+\mathcal O(u^2).\eqlab{vOmega}
\end{align}
Also $$uM^{00}(u)\equiv M(\kappa(u)^2,0,u)=u\int_0^1 \frac{d}{du} M(\kappa(su)^2,0,su)ds.$$ 
Here $\vartheta(u)>0$ for $u\in (0,\tau)$ by (A5).
The blowup Hamiltonian, now denoted by $\hat{\Lambda}$, then becomes
\begin{align}
 \hat{\Lambda} &=  \zeta_0(\hat u,\hat \nu_0,\hat \varrho_0)+\rho_0(\hat u,\hat \varrho_0,\phi_0) +\mathcal O(\hat \Omega^{-7/2}\delta^{9/4}+\hat \Omega^{-5}\delta^3),\eqlab{Hinner}\\
  \hat \omega &= \mu^{-1} d\hat \xi_0\wedge d\hat \sigma_0+ \mu^{-1}\delta^{-3/2} d\hat u\wedge d\hat \nu_0.
 \end{align}
 splitting $\hat{\Lambda}$ it into an integrable part:
\begin{align*}
 \zeta_0(\hat u,\hat \nu_0,\hat \varrho_0) &=\hat \nu_0+\hat \Omega_0(\hat u) \hat \varrho_0,\nonumber
 \end{align*}
 and a remainder
 \begin{align*}
 \rho_0(\hat u,\hat \varrho_0,\phi_0)&= \delta^{3/4}\hat \Omega(\hat u)^{-1/2}\rho_{01}(\hat u,\hat \varrho_0,\phi_0)+\delta^{3/2} \hat \Omega(\hat u)^{-2}\rho_{02}(\hat u,\hat \varrho_0,\phi_0),\nonumber
 \end{align*}
 where
 \begin{align}
 \rho_{01} (\hat u,\hat \xi_0,\hat \sigma_0) &= -\frac12 (1+\mathcal O(u)) \hat \sigma_0 + (1+\mathcal O(u))\hat \xi_0^3+uM^{10}(u)(1+\mathcal O(u))\hat \sigma_0^2 \hat \xi_0,\eqlab{rho01}\\
 \rho_{02} (\hat u,\hat \xi_0,\hat \sigma_0) &= \frac12 (1+\mathcal O(u))\hat \sigma_0\hat \xi_0+\frac12 (1+\mathcal O(u))\hat \xi_0^4\nonumber\\
 &+uM^{10}(u)(1+\mathcal O(u))\hat \sigma_0^2 \hat \xi_0^2+2u^2 M^{01}(u)(1+\mathcal O(u))\hat \sigma_0^4,\eqlab{rho02}
%
\end{align}
with
\begin{align*}
M^{10}(u) &\equiv \partial_{x^2} M(\kappa(u)^2,0,u),\\
M^{01}(u) &\equiv \partial_{y^2} M(\kappa(u)^2,0,u).
\end{align*}
From \eqref{tildeOmega} follows
\begin{align}
 \hat \Omega'(u) &= \frac{1+\mathcal O(u)}{\hat \Omega},\eqlab{dOmega}\\
\mu^2 \hat \Omega^2 &= 2u(1+\mathcal O(u)),\eqlab{mu2}
\end{align}
which was used in \eqref{Hinner} and will be used later on. 

The details of the $\mathcal O(u)$-terms in \eqref{Hinner} are not important. It is again just important to highlight that they are smooth and uniformly bounded up until $u=0$. 
\begin{remark}
It will later be shown that the action $\hat \varrho_0$ only undergoes small oscillations. It will from this follow that
$x\mp \kappa(u)=\mathcal O(\epsilon^{1/2})$ when $u\gg 0$. On the other hand when $u$ is such that $\hat \Omega =\mathcal O(1)$ then $x\mp \kappa(u)=\mathcal O(\epsilon^{1/3}\delta^{1/4})$.
 \xqed{\lozenge}\end{remark}

Having now introduced both $(\hat z_0,w_0)$ and $(\hat \varrho_0,\phi_0)$, I am ready to present the asymptotics from \cite{nei99} that I need.
%
\begin{lemma}\lemmalab{asymplem}
 Consider \eqref{blowupit} and fix $c>0$ large. Assume moreover that $\hat z_0>0$, $\hat z_0 \ne \ln 2/(2\pi)$ and that $\lambda=\lambda(\hat z_0,w_0)$ given by \eqref{pphase} belongs to the interval $[c^{-1},\pi-c^{-1}]\cup [\pi +c^{-1},2\pi-c^{-1}]$. Then for $\hat u\ge \delta \breve u_*$ with $\breve u_*$ sufficiently large the following asymptotics hold:
 \begin{align}
  \hat x_0(-\hat u) &= \sqrt{2\hat z_0}\cos w_0,\eqlab{tildexyasymp1}\\
  \hat y_0(-\hat u) &= \sqrt{2\hat z_0}\sin w_0,\nonumber
 \end{align}
and
 \begin{align}
  \hat \xi_0(\hat u) &= \sqrt{2\hat \varrho_0}\cos \phi_0,\eqlab{tildexyasymp}\\
  \hat \sigma_0(\hat u) &= \sqrt{2\hat \varrho_0}\sin \phi_0,\nonumber
 \end{align}
 with the action-angle variables $(\hat z_0,w_0)$ and $(\hat \varrho_0,\phi_0)$ related by the following expressions
 \begin{align}
 w_0 &=\frac{2}{3}\delta^{-3/2} \hat u^{3/2} +\frac{3}{2} \hat z_0 \ln (\delta^{-1} \hat u)-\frac{\pi}{2}+l,\eqlab{phihatexpr}\\ 
  \hat \varrho_0 &= \frac{1}{2\pi} \ln \frac{1+\vert p\vert^2}{2\vert \textnormal{Im}\,p\vert},\eqlab{Jthetaexpr}\\
  \phi_0 &=-\frac{2\sqrt{2}}{3}\delta^{-3/2}\hat u^{3/2}+{3}\hat \varrho_0\ln (\delta^{-1}\hat u)-\theta,\eqlab{phitildeexpr}
  \end{align}
  where
  \begin{align}
  \theta(\hat \varrho_0,p) &=Q(\hat \varrho_0)-\textnormal{arg}\,(1+p^2),\eqlab{thetaeqn}\\
  Q(\hat \varrho_0)&=-\frac{\pi}{4}+7\hat \varrho_0 \ln 2-\textnormal{arg}\,\Gamma (2i\hat \varrho_0),\eqlab{Qeqn}\\
 p(\hat z_0,\lambda) &=(e^{2\pi \hat z_0}-1)^{1/2} e^{i\lambda},\eqlab{peqn}\\
\lambda&=3\hat z_0\ln 2-\frac{\pi}{4}-\textnormal{arg}\,\Gamma (i\hat z_0)-l.\eqlab{pphase}
 \end{align}
If $\lambda$ in \eqref{pphase} belongs to $[\pi+c^{-1},2\pi -c^{-1}]$ then the $+$ sign should be taken in \eqref{barxy}. If $\lambda\in [c^{-1},\pi-c^{-1}]$ then the $-$ sign should be taken \eqref{barxy}.

Moreover, if $\hat u\ge \hat u_*$, with $\hat u_*$ fixed, then the errors in \eqsref{tildexyasymp1}{tildexyasymp} are of order $\delta^{3/2}$ and $\delta^{3/4}$, respectively.
\end{lemma}
\begin{remark}
  The requirement $\hat z_0\ne \ln 2/(2\pi)$ comes from the fact that $\hat z_0=\ln 2/(2\pi )$ and $\lambda = \pi/2$ $\text{mod}\,\pi$ gives $1+p^2= 0$ where $\text{arg}\,(1+p^2)$, appearing in \eqref{thetaeqn}, is undefined.\xqed{\lozenge}
\end{remark}

%


\begin{proof}

 I use Eqs. (3)-(7) in \cite{nei99} and write them in the blowup variables $(\hat x_0,\hat y_0)$ respectively $(\hat \xi_0,\hat \sigma_0)$. Their $s$ is my $\breve u=\delta^{-1} \hat u$. Moreover, their $\alpha$ and $\rho$ are related to my $\hat z_0$ and $\hat \varrho_0$ by $\alpha^2 = 2\hat z_0$ and $\rho^2=2\hat \varrho_0$. To relate their $\phi$ and $\theta$ to my $w_0$ and $\phi_0$ I also end up having to solve the equations
 \begin{align*}
  \cos w_0 &= \sin\left(\frac{2}{3}\delta^{-3/2} \hat u^{3/2} + \frac{3}{2} \hat z_0 \ln (\delta^{-1} \hat u) +l\right),\\
  \sin w_0 &= -\cos\left(\frac{2}{3}\delta^{-3/2} \hat u^{3/2} + \frac{3}{2} \hat z_0 \ln (\delta^{-1} \hat u) +l\right),
 \end{align*}
 and 
\begin{align*}
  \cos \phi_0 &= \cos\left(\frac{2\sqrt{2}}{3}\delta^{-3/2} (2\hat u)^{3/2} - 3 \hat \varrho_0 \ln (\delta^{-1} \hat u) +\theta\right),\\
 \sin \phi_0 &=-\sin\left(\frac{2\sqrt{2}}{3}\delta^{-3/2} (2\hat u)^{3/2} - 3 \hat \varrho_0 \ln (\delta^{-1} \hat u) +\theta\right) ,
 \end{align*}
with respect to $w_0$ and $\phi_0$. The solutions are
\begin{align}
 w_0& =\frac{2}{3}\delta^{-3/2} \hat u^{3/2} +\frac{3}{2} \hat z \ln (\delta^{-1} \hat u)-\frac{\pi}{2}+l,\eqlab{hatphi0phi}\\
 \phi_0&=-\frac{2\sqrt{2}}{3}\delta^{-3/2}\hat u+{3}\hat \varrho_0\ln (\delta^{-1}\hat u)-\theta.
\end{align}
\qed\end{proof}

\begin{remark}
 During the passing from $\hat u=-\hat u_*$ to $\hat u=\hat u_*$, as described in the previous lemma, the original variable $x=x(u)$ remains $\mathcal O(\epsilon^{1/3})$-close to 
 \begin{align*}
  \left\{ \begin{array}{cc} x = 0,& u<0,\\
  x=\pm \kappa(u),& u>0.
  \end{array}\right.\xqedhere{128.5pt}{\xqed{\lozenge}}
\end{align*}
 \end{remark}

The assignment $(\hat x,\hat y)(-\hat u)\mapsto (\hat \xi,\hat \sigma)(\hat u)$, $\hat u\ge \delta \breve u_*$, in \lemmaref{asymplem} is two-to-one due to its invariance with respect to the symmetry $(\hat x,\hat y)\mapsto \mathcal R (\hat x,\hat y)=(-\hat x,-\hat y)$; in the action-angle variables $(\hat z_0,w_0)$ the symmetry corresponds to a translation:
\begin{align}
\mathcal R:\quad (\hat z_0,w_0)\mapsto (\hat z_0,w_0+\pi),\eqlab{translation}
\end{align}
which we continue to denote by $\mathcal R$: The pair $(\hat \varrho_0,\phi_0)$ in \eqsref{Jthetaexpr}{phitildeexpr} is invariant with respect to $\mathcal R$. However, if I further assign the sign in \eqref{barxy} to the image of this assignment then I obtain a one-to-one mapping. 
Therefore consider $U = [c_1,c_2]\times S^1 \subset \{(\hat z_0,w_0)\}$ with $\ln 2/(2\pi)\notin [c_1,c_2]$ and remove a small, closed neighborhood of the singular set where $\lambda(\hat z_0,w_0) = 0,\,\pi$:
\begin{align}
 V = U\backslash\{(\hat z_0,w_0)\in U\vert \lambda(\hat z_0,w_0)\in [-c^{-1},c^{-1}]\cup [-c^{-1}+\pi,\pi+c^{-1}]\}.\eqlab{Vset}
\end{align}
Note that $V$ is open and large in measure, the complement having a measure of order $c^{-1}$ with $c$ large. Indeed, the set $V$ is strip-like. The strips that are subtracted from $V$ are closed sets having widths of order $c^{-1} \ln^{-1} \delta^{-1}$ cf. \eqref{phihatexpr}. The strips that are included in $V$, on the other hand, are open with non-empty interior having a width of order $\ln^{-1} \delta^{-1}$. There are $\ln \delta^{-1}$ of such strips. Then
\begin{definition}\defnlab{Pcext}
for any $(\hat z_0,w_0)\in V$ I introduce $P_{\textnormal{cr}}$ and $P_{\textnormal{cr}}^{\textnormal{ext}}$ by $P_{\textnormal{cr}}^{\textnormal{ext}}(\hat z_0,w_0)=(\hat \varrho_0,\phi_0,\eta)$ where the pair $(\hat \varrho_0,\phi_0)=P_{\textnormal{cr}}(\hat z_0,w_0)$ is the image of $P_{\textnormal{cr}}$ and given by \eqsref{Jthetaexpr}{phitildeexpr}. The argument $\eta\in\{\pm 1\}$ is $\eta = +1$ when $\lambda\in [\pi+c^{-1},2\pi -c^{-1}]$ and $\eta=-1$ when $\lambda\in [c^{-1},\pi -c^{-1}]$. 
$\square$ \end{definition}
The subscript {cr} in $P_{\textnormal{cr}}$ and $P_{\textnormal{cr}}^{\textnormal{ext}}$ is for ``crossing.'' The superscript ${\textnormal{ext}}$ is for ``extended.''
The assignment $P_{\textnormal{cr}}^{\textnormal{ext}}:V\rightarrow P_{\textnormal{cr}}^{\textnormal{ext}}(V)$ is a diffeomorphism as a flow map and for the truncated equations in \eqref{blowupit} it is described by the asymptotics in \lemmaref{asymplem}. Moreover, it commutes with the translation $(\hat z_0,w_0)\mapsto \mathcal R(\hat z_0,w_0)=(\hat z_0,w_0+\pi)$ in the following way:
\begin{align}
P_{\textnormal{cr}}(\mathcal R(\hat z_0,w_0)) &= P_{\textnormal{cr}}(\hat z_0,w_0),\quad 
 \eta(\mathcal R (\hat z_0,w_0))=-\eta(\hat z_0,w_0).\eqlab{Pcext}
\end{align}
The fact that $P_{\textnormal{cr}}^{\textnormal{ext}}$ is still accurately described by the asymptotics in \lemmaref{asymplem}, when the remainder in \eqref{blowupi} is included, is the subject of the following section. 
\subsection{Applying the asymptotics of the second Painleve Eq. to Eq. \eqref{blowupi}}
So far the asymptotics above is only valid for the truncation \eqref{blowupit} of our blowup \eqref{blowupi}. The remainder in \eqref{blowupi} that is ignored by the truncation is of order $\mu^2 \delta^{-3/2}\vert \hat x\vert$ to lowest order. 
To control the remainder I first need to estimate the truncation's growth with respect to $\delta$.
\begin{lemma}\lemmalab{separationlem}
    Suppose that the assumption of \lemmaref{asymplem} holds true and consider $\hat u\in [-\hat u_*,\hat u_*]$. Then for $\delta$ sufficiently small there exists a constant $c>0$ so that the assignment $\Psi_{\hat u}:(\hat x,\hat y)(-\hat u_*)\mapsto (\hat x,\hat y)(\hat u)$ has the following growth properties with respect to $\delta$:
  \begin{align*}
  \vert \hat x(\hat u) \vert \le c\delta^{-1/4},&\,\vert \hat y(\hat u)\vert \le c\delta^{1/4},\quad 0\le \hat u<\breve u_*\delta,   
  \end{align*}
  and
  \begin{align*}
  \vert \hat x(\hat u) \vert \le c\delta^{-3/4}\hat u^{1/2},&\,\vert \hat y(\hat u)\vert \le c\hat u^{1/4},\quad \hat u\ge \breve u_* \delta.
  \end{align*}
  Here $\breve u_*$ is from \lemmaref{asymplem}.

  Moreover, for $\hat u=\hat u_*$ the Jacobian of this assignment satisfies
\begin{align}
\vert \partial\Psi_{\hat u_*} \vert,\,\vert \partial\Psi_{\hat u_*}^{-1} \vert &\le c\ln^2 \delta^{-1}.\eqlab{jacestu0}
    \end{align}
For any $\hat u\in (-\hat u_*,\hat u_*)$ the more pessimistic estimate applies:
\begin{align}
\vert \partial\Psi_{\hat u} \vert,\,\vert \partial \Psi_{\hat u}^{-1} \vert &\le c\delta^{-1/4}\ln \delta^{-1}.\eqlab{jacestu}
\end{align}
 \end{lemma}
%

 \begin{proof}
 The necessary estimation is delayed to \appref{prooflem}.
 \qed\end{proof}
%
Using this lemma I can then describe how small $\epsilon$ should be relative to $\delta$ to be able to apply \lemmaref{asymplem} to our blowup \eqref{blowupi}.
 \begin{lemma}\lemmalab{deltacond}
 \lemmaref{asymplem} applies to \eqref{blowupi} provided 
 \begin{align}
\epsilon^{2/3}\delta^{-7/2}\ln^3 \delta^{-1}\ll 1.\eqlab{deltacond}
 \end{align}
 \end{lemma}
 \begin{proof}
%
 Denote by $z=(\hat x,\hat y)$ the solution of \eqref{blowupi}. Then I set
 \begin{align}
  z(\hat u) =\Psi_{\hat u}(\Delta(\hat u)),\eqlab{zPhiw}
 \end{align}
 with $\Delta(\hat u)$ unknown and $\Psi_{\hat u}$ the flow-map from \lemmaref{separationlem} associated with \eqref{blowupit} with $\Psi_{-\hat u_*}=Id$. This last equality also implies that $z(-\hat u_*)=\Delta(-\hat u_*)$. By differentiating \eqref{zPhiw} with respect to $\hat u$ I obtain an equation for $\Delta=\Delta({\hat u})$:
\begin{align*}
 \partial \Psi_{\hat u} \Delta'(\hat u) = \mathcal O(\mu^2 \delta^{-9/4}).
\end{align*}
 I have here used \lemmaref{separationlem} to conclude that for every $\hat u$ the remainder of order $\mu^2 \delta^{-3/2}\vert \hat x\vert$ is bounded by a term of order $\mu^2 \delta^{-9/4}$. I then use \eqref{jacestu0} to invert $\partial \Psi_{\hat u}$ and integrate the resulting equation from $-\hat u_*$ to $\hat u_*$ to obtain
\begin{align*}
 \Delta(\hat u_*)&=\Delta(-\hat u_*) + \mathcal O(\mu^2 \delta^{-5/2}\ln \delta^{-1}) =z(-\hat u_*) + \mathcal O(\mu^2 \delta^{-5/2} \ln \delta^{-1}).
\end{align*}
Inserting this into \eqref{zPhiw} finally gives:
\begin{align*}
 z(\hat u_*) = \Psi_{\hat u_*}(z(-\hat u_*)) + \mathcal O(\mu^2 \delta^{-5/2} \ln^3 \delta^{-1}),
\end{align*}
where we have also used \eqref{jacestu} to expand $\Psi_{\hat u_*}(\Delta(\hat u_*))$. Inserting $\mu^2 =\epsilon^{2/3}\delta^{-1}$ completes the proof.
\qed\end{proof}

\section{The return map}\seclab{returnmap}
It is natural to view the return map $P$ in \eqref{P0} as a stroboscopic mapping that assigns initial conditions $(\hat z_0,w_0)$ at $u=-\pi+\tau/2$ to final conditions $2\pi$-later at $u=\pi+\tau/2$. That is
\begin{align*}
 (\hat z_0,w_0) \mapsto P(\hat z_0,w_0)=\phi_{2\pi}(\hat z_0,w_0;-\pi+\tau/2),
\end{align*}
with $\phi_{u}(\hat z_0,w_0;u_0)$, satisfying $\phi_{u_0}(\hat z_0,w_0;u_0)=(\hat z_0,w_0)$, being the flow of the non-autonomous system \eqref{blowupi}. For simplicity I will write this as
\begin{align}
 (\hat z_0,w_0)(u=-\pi+\tau/2) \mapsto P(\hat z_0,w_0) = (\hat z_0,w_0)(u=\pi+\tau/2).\eqlab{notation}
\end{align}
I will decompose the mapping $P$ into the following two parts $$P_1:(\hat z_0,w_0)(u=-\pi +\tau/2)\mapsto (\hat \varrho_0(u=\tau/2),\phi_0(u=\tau/2),\eta),$$ and $$P_2:(\hat \varrho_0(u=\tau/2),\phi_0(u=\tau/2),\eta) \mapsto (\hat z_0,w_0)(u=\pi +\tau/2),$$ so that $P=P_2\circ P_1$. Here I have adopted a similar notation to the one used in \eqref{notation}. I further decompose $P_1$ into three parts setting $$P_1=P_i\circ P_{\text{cr}}^{\text{ext}}\circ P_o,$$ where $P_{\text{cr}}^{\text{ext}}$ is as in \defnref{Pcext}, which following \lemmaref{deltacond} is accurately described by the asymptotics in \lemmaref{asymplem}, and where
\begin{itemize}
\item $P_o$ is the ``outer'' map 
\begin{align}
P_o:(\hat z_0,w_0)(u=-\pi +\tau/2) \mapsto (\hat z_0,w_0)(\hat u=-\hat u_*);\eqlab{Poo}
\end{align}
 \item $P_i$ is the ``inner'' map: 
 \begin{align}
P_{i}:(\hat \varrho_0,\phi_0)(\hat u=\hat u_*) \mapsto (\hat \varrho_0,\phi_0)(u= \tau/2).\eqlab{Pii}  
 \end{align}
\end{itemize}
In principle, to make sense of $P_i\circ P_{\text{cr}}^{\text{ext}}$, I should include $\eta$ in the argument of $P_i$, but by the $\mathcal R$-symmetry $P_i$ applies as two identical copies on the components $\eta=\pm 1$. For ease of notation I just think of $P_i$ acting on the $P_{\text{cr}}$-part only (recall the definition of $P_{\text{cr}}$ in \defnref{Pcext}). 
The maps are illustrated in \figref{Pmaps}. 
\begin{figure}[h!]
\begin{center}
{\includegraphics[width=.95\textwidth]{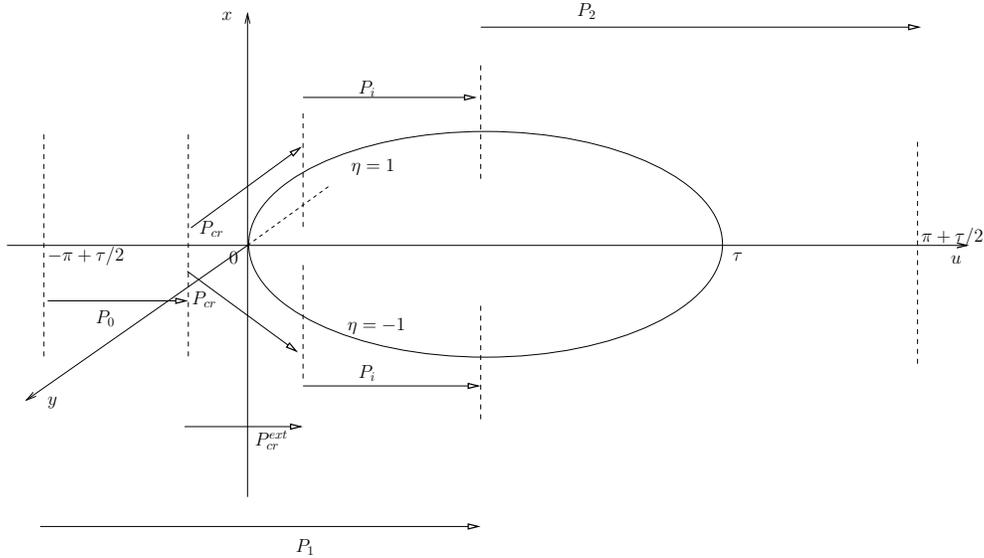}}
\end{center}
\caption{The return map $P$ and its factors. }
\figlab{Pmaps}
\end{figure}

By assumption (A3) the equations of motion for $(x(u),y(u))$ are invariant with respect to the following transformation
\begin{align*}
x&\mapsto x,\\
y&\mapsto -y,\\
u&\mapsto \tau-u.
\end{align*}
This is just the time-reversible $\mathcal T_{\tau}$-symmetry \eqref{Ttau} viewed as an action on $(x,y)=(x,y)(u)$. From this follows: 
\begin{lemma}\lemmalab{P2P1}
 The mapping $P_2^{-1}$ is related to $P_1$ through the following expression:
 \begin{align*}
  P_2^{-1} &= E \circ P_1 \circ E,
  \end{align*}
  where $E=\textnormal{diag}\,(1,-1)$.
\end{lemma}
In the following I will obtain approximations to the inner and outer maps $P_i$ and $P_o$, respectively. I start by considering the one that requires most effort, $P_i$.

\subsection{The inner map $P_i$}\seclab{Pinner}
To approximate the inner map we first apply averaging to \eqref{Hinner}. It will be important to keep track of how the frequency $\hat \Omega=\hat\Omega(\hat u)$ enters the remainder, and to highlight the most important terms I will make use of the following two lemmata.  
\begin{lemma}\lemmalab{Omegacontrol0}
 Let $q$ and $p$ be positive real numbers satisfying $0<q<p$. Then there exists a constant $c=c(\hat u_*)$ so that
 \begin{align}
  \hat \Omega(\hat u)^{-2p}\le c^{p-q} \hat \Omega(\hat u)^{-2q}, \eqlab{tildeOmegadominate}
 \end{align}
for all $\hat u\in [\hat u_*,\mu^{-2} \tau/2]$.
\end{lemma}
\begin{proof}
By choice of $\hat u_*$, the function $\hat \Omega(\hat u)^{-2}$ is bounded from above by some constant $c=c(\hat u_*)$. Therefore
\begin{align*}
  \frac{\hat \Omega (\hat u)^{-2p}}{\hat \Omega (\hat u)^{-2q}} = \hat \Omega(\hat u)^{-2(p-q)} \le c^{p-q},
\end{align*}
given that $p>q$.
\qed\end{proof}

\begin{lemma}\lemmalab{Omegacontrol}
 Let $q\in \overline{\R}_+$. Given an integrable function $r=r(\hat u),\, \hat u\in [\hat u_*,\mu^{-2} \tau/2]$ satisfying the following estimate
 \begin{align*}
 \vert r(\hat u)\vert \le \hat \Omega(\hat u)^{-2q},\quad \hat u\in [\hat u_*,\mu^{-2} \tau/2].
 \end{align*}
If $q< 1$ then there exists a $c_1=c_1(q)$ so that
\begin{align*}
 \vert \int_{\hat u_*}^{\epsilon^{-2/3}\delta \tau/2} r(\hat u)d\hat u\vert \le c_1\mu^{-2(1-q)}.
\end{align*}
If $q=1$ then there exists a $c_2$ so that
\begin{align}
 \vert \int_{\hat u_*}^{\epsilon^{-2/3}\delta \tau/2} r(\hat u)d\hat u\vert &\le c_2\ln (\mu^{-2} \hat u_*^{-1}).\eqlab{intq1}
\end{align}
Finally if $q>1$ then the corresponding integral is uniformly bounded with respect to $\epsilon$: There exists a $c_3$ so that 
\begin{align*}
 \vert \int_{\hat u_*}^{\epsilon^{-2/3}\delta \tau/2} r(\hat u)d\hat u\vert \le c_3(q-1)^{-1} \hat u_*^{1-q}.
\end{align*}
\end{lemma}
\begin{proof}
 For $q\ne 1$ I have that
 \begin{align*}
  \vert \int_{\hat u_*}^{\mu^{-2} \tau/2} r(\hat u)d\hat u\vert &\le \int_{\hat u_*}^{\mu^{-2}\tau/2} \vert r(\hat u)\vert d\hat u\\
  &\le \int_{\hat u_*}^{\mu^{-2}\tau/2} \hat \Omega(\hat u)^{-2q}d\hat u\\
 &=\text{(Use \eqref{tildeOmega})}\\
 &\le c_1\left(\mu^{2}\right)^{q-1} \int_{u_*}^{\tau/2} \vartheta(u)^{-q}du,
 \end{align*}
 setting $u_*=\mu^{2}\hat u_*$. Here $c_1$ is some constant depending only $q$.
I then use that $u\vartheta(u)^{-1}=1+\mathcal O(u)$ and therefore $\vartheta(u)^{-1}=u^{-1} + \mathcal O(1)$ cf. \eqref{vOmega} for $u>0$ to conclude
 \begin{align}
  \vert \int_{\hat u_*}^{\mu^{-2}\tau/2} r(\hat u)d\hat u\vert &\le c_2 c_1\mu^{2(q-1)} (1-q)^{-1} \left((\tau/2)^{1-q} - u_*^{1-q}\right),\eqlab{term}
  \end{align}
  for some constant $c_2$ again depending only on $q$. If $q<1$ then $(\tau/2)^{1-q}$ dominates the last factor and I obtain the first result of the lemma:
  \begin{align*}
   \vert \int_{\hat u_*}^{\mu^{-2} \tau/2} r(\hat u)d\hat u\vert &\le  2c_2c_1 \mu^{-2(q-1)}(1-q)^{-1} (\tau/2)^{1-q},
  \end{align*}
   for $\hat u_*$ sufficiently small. For $q>1$ the last term in \eqref{term} dominates for $\epsilon$ small and so 
  upon inserting $u_*=\mu^{2}\hat u_*$ I obtain
  \begin{align*}
   \vert \int_{\hat u_*}^{\mu^{-2} \tau/2} r(\hat u)d\hat u\vert &\le  2c_2c_1 (q-1)^{-1} \hat u_*^{1-q},
  \end{align*}
completing the third part for $\hat u_*$ sufficiently small. The case $q=1$ is also the consequence of a simple calculation.
\qed\end{proof}
Recall the form of $\hat\Lambda$ in \eqref{Hinner}:
\begin{align}
  \hat\Lambda&=\zeta_0(\hat u,\hat \nu_0,\hat \varrho_0)+\rho_0(\hat u,\hat \varrho_0,\phi_0)+\mathcal O(\hat \Omega^{-7/2}\delta^{9/4}),\nonumber\\
  \rho_0 &=\delta^{3/4}\hat \Omega(\hat u)^{-1/2}\rho_{01}+\delta^{3/2}\hat \Omega(\hat u)^{-2}\rho_{02},\eqlab{r0eqn}\\
  \hat \omega &= d\hat \xi_0\wedge d\hat \sigma_0+ \delta^{-3/2} d\hat u\wedge d\hat \nu_0,\nonumber
\end{align}
where I have used \lemmaref{Omegacontrol0} to say that $\hat \Omega^{-5}\le c^{3/4}\hat \Omega^{-7/2}$. In comparison with \eqref{Hinner} I have also multiplied the symplectic form by $\mu$ (which corresponds to scaling time by $\mu$). 
The averages of $\rho_{01}$ and $\rho_{02}$ in \eqsref{rho01}{rho02} are easily computed: 
\begin{align*}
\overline{\rho}_{01} (u,\hat \varrho_0)&=\frac{1}{2\pi}\int_0^{2\pi} \rho_{01}( u,\hat \varrho_0,s)ds = 0,\\
\overline{\rho}_{02} (u,\hat \varrho_0)&=\frac{1}{2\pi}\int_0^{2\pi} \rho_{02}( u,\hat \varrho_0,s)ds=\frac34  (1+\mathcal O(u)) \hat \varrho_0^2.
\end{align*}
Therefore
\begin{align*}
 \overline{\rho}_0(\hat u,\hat \varrho_0 )& =
 \frac34 \delta^{3/2}\hat \Omega(\hat u)^{-2} (1+\mathcal O(u)) \hat \varrho_0^2.
\end{align*}
I also set $\tilde \rho_0=\rho_0-\overline{\rho}_0$ and $\tilde \rho_{02}=\rho_{02}-\overline{\rho}_{02}$ so that $\tilde \rho_0$ has zero average. I first realise the following:
\begin{lemma} \lemmalab{rtilde0}
 The order of $\tilde \rho_0$ is $\hat \Omega^{-1/2}\delta^{3/4}$. Similarly the order of $\partial_{\hat u}\tilde \rho_0$ is $\hat \Omega^{-5/2}\delta^{3/4}$.
\end{lemma}
\begin{proof}
 The term with $\rho_{01}$ dominates the expression for $\rho_0$ cf. $\delta^{3/2}\ll \delta^{3/4}$ and using \eqref{tildeOmegadominate} to say that $\hat \Omega^{-2}\le c^{3/4}\hat \Omega^{-1/2}$. This gives the first claim. The next claim follows from similar arguments upon differentiation with respect to $\hat u$:
 \begin{align*}
  \partial_{\hat u} \tilde \rho_0 &= -\frac12 \delta^{3/4} \hat \Omega(\hat u)^{-5/2} (1+\mathcal O(u)) \rho_{01}(u,\hat \varrho_0,\phi_0)+\delta^{3/4} \mu^2 \hat \Omega(\hat u)^{-1/2}\partial_{u} \rho_{01}(u,\hat \varrho_0,\phi_0)\\
  &-2\delta^{3/2}\hat \Omega(\hat u)^{-4}(1+\mathcal O(u))\tilde  \rho_{02}(u,\hat \varrho_0,\phi_0)+\delta^{3/2}\mu^2 \hat \Omega(\hat u)^{-2}\partial_{u} \tilde \rho_{02}(u,\hat \varrho_0,\phi_0).
 \end{align*}
 Here I have used \eqref{dOmega}. Replacing $\mu^2$ by \eqref{mu2} gives
 \begin{align*}
  \partial_{\hat u} \tilde \rho &=\mathcal O(\hat \Omega^{-5/2}\delta^{3/4}+ \hat \Omega^{-4}\delta^{3/2})=\mathcal O(\hat \Omega^{-5/2}\delta^{3/4}),
 \end{align*}
 which completes the result.
 \qed
%
\end{proof}
To push the phase-dependency to higher order, I then use the following generating function 
 \begin{align*}
  G(\hat u,\hat \nu_1,\hat \varrho_{0},\phi_1)&=\delta^{-3/2} \hat u \hat \nu_1+\hat \varrho_0\phi_1 +\hat \Omega(\hat u)^{-1}\int_0^{\phi_1}\tilde \rho_0(\hat u,\hat \varrho_0,s)ds,  
   \end{align*}
   to generate a transformation given as the solution to the equations:
\begin{align}
 \hat u_1 = \delta^{3/2}\partial_{\hat v_1} G &= \hat u,\nonumber\\ 
 \hat \nu_0 = \delta^{3/2}\partial_{\hat u} G &= \hat \nu_{1} +\delta^{3/2}\hat \Omega(\hat u)^{-1}\int_0^{\phi_1}\partial_{\hat u}\tilde \rho_0(\hat u,\hat \varrho_0,s)ds\nonumber\\
 &=\hat \nu_1 + \mathcal O(\hat \Omega(\hat u)^{-7/2}\delta^{9/4}),\eqlab{tildev1}
\end{align}
using \lemmaref{rtilde0} in the last equality, 
and
\begin{align*}
\hat \varrho_1  = \partial_{\phi_1} G &= \hat \varrho_0+\hat \Omega(\hat u)^{-1}\tilde \rho_0(\hat u,\hat \varrho_0,\phi_1),\\
 \phi_{0} = \partial_{\hat \varrho_0} G &= \phi_1 +\hat \Omega(\hat u)^{-1}\int_0^{\phi_1}\partial_{\hat \varrho_0}\tilde \rho_0(\hat u,\hat \varrho_0,s)ds.
\end{align*}
I then obtain the following system
\begin{align*}
 \hat\Lambda&= \zeta_1(\hat u,\hat \nu_1,\hat \varrho_1)+\rho_1(\hat u,\hat \nu_1,\hat \varrho_1)+\mathcal O(\hat \Omega^{-7/2}\delta^{9/4}),
\end{align*}
where
\begin{align*}
\zeta_1(\hat u,\hat \nu_1,\hat \varrho_1) &=\hat \nu_1+\hat \Omega(\hat u)\hat \varrho_1+\overline{\rho}_0(\hat u,\hat \varrho_1)=\hat \nu_1+\hat \Omega(\hat u)\hat \varrho_1+\frac34 \delta^{3/2}\hat \Omega(\hat u)^{-2} (1+\mathcal O(u)) \hat \varrho_0^2,\\
 \rho_1(\hat u,\hat \varrho_1,\phi_1) &= \overline{\rho}_0(\hat u,\hat \varrho_0)-\overline{\rho}_0(\hat u,\hat \varrho_1)+\tilde \rho_0(\hat u,\hat \varrho_0,\phi_0) -\tilde \rho_0(\hat u,\hat \varrho_0,\phi_1)\\
 &+\delta^{3/2}\hat \Omega(\hat u)^{-1} \int_0^{\phi_1}\partial_{\hat u}\tilde \rho_0(\hat u,\hat \varrho_0,s)ds.
 \end{align*}
In the following lemma I estimate $\rho_1$.
\begin{lemma}
 The remainder $\rho_1=\rho_1(\hat u,\hat \varrho_1,\phi_1)$ takes the following form
 \begin{align*}
  \rho_1(\hat u,\hat \varrho_1,\phi_1) = \delta^{3/2} \hat \Omega(\hat u)^{-2} \partial_{\phi_0}\rho_{01}( u,\hat \varrho_1,\phi_1)\int_0^{\phi_1} \partial_{\hat \varrho_0} \rho_{01}( u,\hat \varrho_1,s)ds+\mathcal O(\hat \Omega^{-7/2}\delta^{9/4}).
 \end{align*}

\end{lemma}
\begin{proof}
Firstly,
\begin{align*}
 \overline{\rho}_0(\hat u,\hat \varrho_0)-\overline{\rho}_0(\hat u,\hat \varrho_1) &= \frac34 \delta^{3/2}\hat \Omega(\hat u)^{-2} (1+\mathcal O(u)) (\hat \varrho_0+\hat \varrho_1)(\hat \varrho_0-\hat \varrho_1)\\
 &=-\frac34 \delta^{3/2}\hat \Omega(\hat u)^{-3} (1+\mathcal O(u)) (\hat \varrho_0+\hat \varrho_1)\tilde \rho(\hat u,\hat \varrho_0,\phi_1)\\
 &=\mathcal O(\hat \Omega^{-7/2}\delta^{9/4}),
\end{align*}
using \lemmaref{rtilde0}. Next using Taylor's theorem:
\begin{align*}
 \tilde{\rho}_0(\hat u,\hat \varrho_0,\phi_0)&-\tilde{\rho}_0(\hat u,\hat \varrho_0,\phi_1)=\partial_{\phi_0} \tilde \rho_0(\hat u,\hat \varrho_0,\phi_1)(\phi_0-\phi_1)+\int_0^1(1-s) \\
 &\times \partial_{\phi_0}^2 \tilde \rho_0(\hat u,\hat \varrho_0,\phi_1+s(\phi_0-\phi_1))(\phi_0-\phi_1)^2 ds\\
 &=\hat \Omega(\hat u)^{-1} \partial_{\phi_0} \tilde \rho_0(\hat u,\hat \varrho_0,\phi_1)\int_0^{\phi_1}\partial_{\hat \varrho_0}\tilde \rho_0(\hat u,\hat \varrho_0,s)ds+\mathcal O(\hat \Omega(\hat u)^{-7/2} \delta^{9/4})\\
 &=\hat \Omega(\hat u)^{-1} \partial_{\phi_0} \tilde \rho_0(\hat u,\hat \varrho_1,\phi_1)\int_0^{\phi_1}\partial_{\hat \varrho_0}\tilde \rho_0(\hat u,\hat \varrho_1,s)ds+\mathcal O(\hat \Omega(\hat u)^{-7/2} \delta^{9/4}),
\end{align*}
having here also used that $\hat \varrho_0-\hat \varrho_1,\,\phi_0-\phi_1=\mathcal O(\hat \Omega (\hat u)^{-3/2} \delta^{3/4})$ cf. \lemmaref{rtilde0}.
%
I complete the result by using \eqref{tildev1} and the following:
\begin{align*}
 \hat \Omega(\hat u)^{-1} \partial_{\phi_0}\tilde \rho_0(\hat u,\hat \varrho_1,\phi_1)&\int_0^{\phi_1} \partial_{\hat \varrho_0}\tilde \rho_0(\hat u,\hat \varrho_1,s)ds = \delta^{3/2} \hat \Omega(\hat u)^{-2} \partial_{\phi_0}\rho_{01}( u,\hat \varrho_1,\tau)\\
 &\times \int_0^{\phi_1} \partial_{\hat \varrho_0}\rho_{01}( u,\hat \varrho_1,s)ds+\mathcal O(\hat \Omega^{-7/2}\delta^{9/4}),
\end{align*}
which follows from \eqref{r0eqn}.
\qed\end{proof}
One more averaging step is needed to push the order of the error below $\hat \Omega^{-2}\delta^{3/2}$: Note that its contribution matters cf. \eqref{intq1} on the time scale $u\in [\mu^2 \hat u,\tau/2]$ relevant for $P_i$, see \eqref{Pii}. I therefore define 
\begin{align*}
 \rho_{12}( u,\hat \varrho_1,\phi_1)&=\partial_{\phi_0}\rho_{01}( u,\hat \varrho_1,\phi_1)\int_0^{\phi_1} \partial_{\hat \varrho_0} \rho_{01}( u,\hat \varrho_1,s)ds,
\end{align*}
so that
\begin{align*}
 \rho_1(\hat u,\hat \varrho_1,\phi_1) = \delta^{3/2} \hat \Omega(\hat u)^{-2} \rho_{12} (u,\hat \varrho_1,\phi_1)+\mathcal O(\hat \Omega^{-7/2}\delta^{9/4}). 
\end{align*}
The average of $\rho_{12}$ is easily computed given \eqref{rho01}:
\begin{align*}
 \overline{\rho}_{12}(u,\hat \varrho_1)& =
 -\frac{15}{4} (1+\mathcal O(u))\hat \varrho_1^2-\frac18 (1+\mathcal O(u)).
\end{align*}
I set $\tilde \rho_{1}=\rho_1-\overline{\rho}_1$ and use the following generating function 
\begin{align*}
   G(\hat u_1,\hat \nu,\hat \varrho_{1},\phi)&=\delta^{-3/2} \hat u_1 \hat \nu+\hat \varrho_1\phi +\hat \Omega(\hat u_1)^{-1}\int_0^{\phi}\tilde \rho_1(\hat u_1,\hat \varrho_1,s)ds,  
\end{align*}
to generate a final transformation $$(\hat u_1,\hat \nu_1,\hat \varrho_1,\phi_1)\mapsto (\hat u,\hat \nu,\hat \varrho,\phi).$$ This gives the following form of the Hamiltonian in the new variables:
\begin{align*}
  \hat\Lambda&= \zeta(\hat u,\hat \nu,\hat \varrho)+ \mathcal O(\hat \Omega^{-7/2}\delta^{9/4}),
  \end{align*}
  where 
  \begin{align*}
  \zeta(\hat u,\hat \nu,\hat \varrho)&=h_1(\hat u,\hat \nu,\hat \varrho)+\overline{\rho}_1(\hat u,\hat \varrho)=\hat \nu+\hat \Omega(\hat u) \hat \varrho -3\delta^{3/2} \hat \Omega(\hat u)^{-2} \left(1+\mathcal O(u)\right)\hat \varrho^2.
  \end{align*}
  I have here ignored the term $-\frac18 \delta^{3/2} \hat \Omega(\hat u)^{-2} (1+\mathcal O(u))$ that only depends on $u$; it can be removed by a further translation of $\hat \nu$.
The equations of motion are
\begin{align}
 \frac{d\hat \varrho}{d\hat u} &=\mathcal O(\hat \Omega^{-7/2} \delta^{3/4}),\eqlab{tildeJeqn}\\
 \frac{d\phi}{d\hat u}&=- \delta^{-3/2}\hat \Omega(\hat u)+6 \hat \Omega(\hat u)^{-2} (1+\mathcal O(u))\hat \varrho+\mathcal O(\hat \Omega^{-7/2} \delta^{3/4}).\eqlab{tildephieqn}
\end{align}
The mapping $P_i$ describes the assignment $(\hat \varrho_0,\phi_0)(\hat u=\hat u_*)\mapsto (\hat \varrho_0,\phi_0)(\hat u=\mu^{-2} \tau/2)$. To approximate this I will solve the truncation of the transformed differential equations \eqsref{tildeJeqn}{tildephieqn}:
\begin{align*}
 \frac{d\hat \varrho}{d\hat u} &=0,\\
 \frac{d\phi}{d\hat u}&=- \delta^{-3/2}\hat \Omega(\hat u)+6 \hat \Omega(\hat u)^{-2} (1+\mathcal O(u))\hat \varrho,
\end{align*}
from $\hat u=\hat u_*$ to $\hat u=\mu^{-2} \tau/2$. This is adequate because of the following:
\begin{lemma}
The action $\hat \varrho_0$ is conserved on the interval from $\hat u=\hat u_*$ to $\hat u=\mu^{-2} \tau/2$ with an accuracy of $\delta^{3/4}$.
\end{lemma}
\begin{proof}
 I use \lemmaref{Omegacontrol} with $\vert r(\hat u)\vert \le \hat \Omega^{-7/2}$ ($q=7/4>1$) together with \eqref{tildeJeqn} to conclude that the variation $\Delta \hat \varrho$ of $\hat \varrho$ on the given interval can be estimated from above by
 \begin{align*}
  \vert \Delta \hat \varrho \vert &\le c\delta^{3/4}.
 \end{align*}
Since $\hat \varrho-\hat \varrho_0=\mathcal O(\delta^{3/4})$ this completes the result.
\qed\end{proof}

By a similar argument, I estimate the effect of the remainder in \eqref{tildephieqn} by $\mathcal O(\delta^{3/4})$ and I compute the variation in $\phi$ by
\begin{align}
 \phi(\epsilon^{-2/3}\delta \tau/2 )&=\phi(\hat u_*) -\int_{\hat u_*}^{\mu^{-2} \tau/2 } \bigg(\delta^{-3/2}\hat \Omega(\hat u) -6\hat \Omega(\hat u)^{-2}(1+\mathcal O(u)))\hat \varrho_0\bigg)\nonumber\\
 &+\int_{\hat u_*}^{\mu^{-2} \tau/2 }\mathcal O(\hat \Omega(\hat u)^{-2}  \delta^{3/4} )d\hat u +\mathcal O(\delta^{3/4}).\eqlab{tildephiPi}
 \end{align}
 The remainder $\mathcal O( \hat \Omega(\hat u)^{-2} \delta^{3/4})$ in the integral in \eqref{tildephiPi} comes from $\hat \varrho(\hat u)=\hat \varrho_0+\mathcal O(\delta^{3/4})$ with $\hat \varrho_0=\text{const}.$ on this interval. This term can be estimated from above by a term of order $\delta^{3/4}\ln \epsilon^{-1}$ cf. \lemmaref{Omegacontrol} with $q=1$. The following lemma gives asymptotics of the two other integrals appearing in \eqref{tildephiPi}.
 \begin{lemma}\lemmalab{intint}
 \begin{align*}
  \delta^{-3/2} \int_{\hat u_*}^{\mu^{-2} \tau/2 }\hat \Omega(\hat u) d\hat u &= \sqrt{2}\epsilon^{-1} e_1 -\frac{2\sqrt{2}}{3}\delta^{-3/2}\hat u_*^{3/2}+\mathcal O(\epsilon^{2/3}\delta^{-5/2}).
   \end{align*}
   with
   \begin{align*}
    e_1=\int_0^{\tau/2} \vartheta(u)^{1/2}du,
   \end{align*}
with $\vartheta=\vartheta(u)=u+\mathcal O(u^2)$ \eqref{vOmega}. Moreover, there exists a positive constant $e_2$ independent of $\epsilon$ such that
\begin{align*}
\int_{\hat u_*}^{\mu^{-2} \tau/2} \hat \Omega(\hat u)^{-2}(1+\mathcal O(u)) d\hat u = \frac{1}{3}\ln(e_2 \epsilon^{-1})-\frac{1}{2}\ln (\delta^{-1} \hat u_*)+\mathcal O(\mu^2).
\end{align*}
 \end{lemma}
\begin{proof}
 I use \eqref{vOmega} to write $\hat \Omega(\hat u)$ as $\sqrt{2} \mu^{-1} \vartheta(u)^{1/2}$ with $\vartheta(0)=0$, $\vartheta'(0)=1$. Therefore
 \begin{align}
    \delta^{-3/2} \int_{\hat u_*}^{\mu^{-2} \tau/2 }\hat \Omega(\hat u) d\hat u  &= \sqrt{2} \epsilon^{-1} \int_{u_*}^{\tau/2} \vartheta(u)^{1/2} du \nonumber\\
    &= \sqrt{2} \epsilon^{-1} e_1 -\sqrt{2} \epsilon^{-1} \int_0^1 \vartheta(u_*s)^{1/2}ds u_*,\eqlab{int1}
 \end{align}
 here $u_*=\mu^{2}\hat u_*$. For the last integral I use the following
\begin{align}
 \int_0^1 u_*^{-1/2} \vartheta(u_*s)^{1/2}ds u_*^{3/2}& = \int_0^1 (1+\mathcal O(u_*^{3/2} s^{3/2})) ds u_*^{3/2} \nonumber\\
 &= \frac23 u_*^{3/2}+\mathcal O(u_*^{5/2}).\eqlab{intvsqrt}
\end{align}
Inserting this back into \eqref{int1} completes the first part of the proof.

For the second part, the integral is written as
\begin{align*}
\int_{\hat u_*}^{\mu^{-2} \tau/2} \hat \Omega(\hat u)^{-2}(1+\mathcal O(u)) d\hat u =  \frac12 \int_{u_*}^{\tau/2} \vartheta(u)^{-1} (1+\mathcal O(u)) du.
\end{align*}
I write $\vartheta(u)^{-1}=u^{-1}+\mathcal O(1)$ for $u$ small so that
\begin{align*}
 \int_{\hat u_*}^{\epsilon^{-2/3}\delta \tau/2} \hat \Omega(\hat u)^{-2}\left( 1+\mathcal O(u)\right) d\hat u &= \frac12 \int_{u_*}^{\tau/2} u^{-1} du + \mathcal O(1)\\
 &=\frac12 \ln (\mu^{-2} \hat u_*^{-1}) + \mathcal O(1)\\
 &=\frac13 \ln \epsilon^{-1} -\frac{1}{2}\ln (\delta^{-1} \hat u_*)+\mathcal O(1),
\end{align*}
using that $\mu^2=\epsilon^{2/3}\delta^{-1}$. The $\mathcal O(1)$-term is smooth as a function of $u_*$ and can therefore be written as $\text{const}.+\mathcal O(u_*)$ for $u_*$ small by Taylor's theorem. Writing the constant as $\frac13 \ln e_2$ completes the proof.
\qed\end{proof}
Following this lemma I can write \eqref{tildephiPi} as
\begin{align*}
 \phi_0(\mu^{-2} \tau/2 )&=\phi_0(\hat u_*) -\sqrt{2} \epsilon^{-1} e_1+2\ln(e_2 \epsilon^{-1})\hat \varrho_0 +\frac{2\sqrt{2}}{3}\delta^{-3/2} \hat u_*^{3/2} -3\ln (\delta^{-1} \hat u_*) \hat \varrho_0\\
 &+\mathcal O(\delta^{3/4}\ln \epsilon^{-1}),
\end{align*}
using $\phi_0=\phi+\mathcal O(\delta^{3/4})$.
I collect the result about $P_i$ in the following proposition:
\begin{prop}\proplab{pinner}
Suppose that 
\begin{align}
 \delta^{3/4}\ln \epsilon^{-1} \ll 1,\eqlab{deltacond1}
\end{align}
then the mapping $P_i$, see \eqref{Pii}, satisfies
 \begin{align}
 P_i(\hat \varrho_0,\phi_0) &=\begin{pmatrix} \hat \varrho_0\\
                                \phi_0 -\sqrt{2}\epsilon^{-1}e_1 + (2\ln(e_2 \epsilon^{-1})-3\ln (\delta^{-1}\hat u_*))\hat \varrho_0 +\frac{2\sqrt{2}}{3}\delta^{-3/2} \hat u_*^{3/2}
                                \end{pmatrix}\nonumber\\
                                &+\mathcal O(\delta^{3/4}\ln \epsilon^{-1}).\eqlab{Piapp}
\end{align}
\end{prop}

\subsection{The outer map $P_o$}\seclab{avule0}
The derivation of an approximation of the outer map $P_o$ in \eqref{Poo} is similar to the derivation presented above in \secref{Pinner} for the inner map $P_i$.
The details are therefore delayed to \appref{pouterapp} and I simply state the result:
\begin{prop}\proplab{pouter}
Suppose that 
\begin{align*}
 \delta^{3/2}\ln \epsilon^{-1} \ll 1,
\end{align*}
then
 \begin{align}
 P_o(\hat z_0,w_0) &=\begin{pmatrix} \hat z_0\\
                                w_0 -\epsilon^{-1} e_3- (\ln(e_4 \epsilon^{-1})-\frac32 \ln (\delta^{-1}\hat u_*)) \hat z_0 +\frac{2}{3}\delta^{-3/2} \hat u_*^{3/2} 
                                \end{pmatrix}\nonumber\\
                                &+\mathcal O(\delta^{3/2}\ln \epsilon^{-1})                                .\eqlab{Poapp}
\end{align}
Here 
 \begin{align*}
    e_3&=\int_0^{\pi-\tau/2} (-f(-u))^{1/2}du, 
\end{align*}
while $e_4$ is another positive constant defined in \appref{pouterapp}.
\end{prop}
\begin{remark}
 The condition $\delta^{3/2}\ln \epsilon^{-1} \ll 1$ in \propref{pouter} is weaker than the condition $\delta^{3/4}\ln \epsilon^{-1} \ll 1$ in \propref{pinner}. Note that it is possible to realize $\delta^{3/4}\ln \epsilon^{-1} \ll 1$ without violating condition \eqref{deltacond}.\xqed{\lozenge}
\end{remark}

\subsection{Fix point Eq.} \seclab{fpeqn}
To obtain periodic orbits I solve the following fix point equation
\begin{align}
\chi (\hat z_0,w_0)=
P(\hat z_0,w_0),\quad P = P_2\circ P_1,\eqlab{fp}
\end{align}
up to symmetry $\chi\in \{Id,\mathcal R\}$. Here $\{Id,\,\mathcal R\}$ is the group generated by the translation $\mathcal R$ from \eqref{translation}. A simple calculation gives the following result:
\begin{lemma}
 If $\chi=\mathcal R$ in \eqref{fp} then $(\hat z_0,w_0)$ is a periodic-$2$ point.
\end{lemma}
\begin{proof}
It follows from the fact that $\mathcal R^2 = Id$ and that $P$ is equivariant with respect to the action of $\mathcal R$.
\qed\end{proof}
To avoid having to invert the crossing map $P_{\text{cr}}^{\text{ext}}$ I re-write \eqref{fp} as
\begin{align}
 P_1\begin{pmatrix}
  \hat z_0\\
  w_0
 \end{pmatrix} =P_2^{-1}\circ \chi \begin{pmatrix}
  \hat z_0\\
  w_0
 \end{pmatrix}= E \circ P_1 \circ E \circ \chi \begin{pmatrix}
  \hat z_0\\
  w_0
 \end{pmatrix},\,\,\, E=\text{diag}(1,-1),\eqlab{P1Fixpoint}
\end{align}
by inverting $P_2$ using \lemmaref{P2P1}. 
The mapping $P_1$ is then replaced by $$P_i\circ P_{\text{cr}}^{\text{ext}}\circ P_o,$$ and then by using the equivariance of $P_o$ and $E$ with respect to the $\mathcal R$-action, one easily verifies, that \eqref{P1Fixpoint} implies that
\begin{align}
 P_i \circ P_{\text{cr}} \circ P_o\begin{pmatrix}
  \hat z_0\\
  w_0
 \end{pmatrix} =  E \circ P_i \circ P_{\text{cr}} \circ P_o \circ E\begin{pmatrix}
  \hat z_0\\
  w_0
 \end{pmatrix},\eqlab{fpeqn}
 \end{align}
 taking $\chi=Id$ if
 \begin{align*}
 \eta\left(P_o\begin{pmatrix}
  \hat z_0\\
  w_0
 \end{pmatrix}\right) = \eta\left(P_o\circ E\begin{pmatrix}
  \hat z_0\\
  w_0
 \end{pmatrix}\right),
\end{align*}
and taking, cf. \eqref{Pcext}, $\chi=\mathcal R$, if this last equality, on the other hand, does not hold. 

\begin{remark}\remlab{p2eta}
A solution $(\hat z_0,w_0)$ to \eqref{fpeqn} always defines a periodic orbit. Solutions with $\chi=Id$ correspond to fix points of $P$, that is periodic orbits of \eqref{eqn0} with periods $T=2\pi\epsilon^{-1}$, where $(\vert x(u)\vert,y(u))$ remains close to the singular solution $(\vert x_s(u)\vert,y_s)$ \eqref{xs}. The periodic-2 points, that appear when one has to take $\chi$ to be $\mathcal R$, are fix points of $P^2$ and correspond to periodic orbits of twice the period: $T=4\pi\epsilon^{-1}$. Here $(\vert x(u)\vert,y(u))$ is still close to the singular solution $(\vert x_s(u)\vert,y_s)$ \eqref{xs}, but in this case the motion alternates between being close to $\kappa(u)$ to being close to $-\kappa(u)$: It is $\mathcal R$-symmetric. The latter property is a consequence of the fact that if $w_0$ is shifted by $\pi$ then $\lambda$ is also shifted by $\pi$, cf. \eqsref{phihatexpr}{pphase}. By definition this changes the sign of $\eta$ and what route is followed on $S$. The symmetry properties of the periodic orbits are also the subject of \propref{symmetry} 
below.
 \xqed{\lozenge}\end{remark}

\section{Solving the fix point Eq. \eqref{fpeqn} - Proof of the main results}\seclab{solving}
In the following two sections, the two sides of Eq. \eqref{fpeqn} are computed using the approximations of $P_o$, $P_{cr}$ and $P_i$ established above.
\subsubsection*{The left hand side of Eq. \eqref{fpeqn}}
Setting $w_0(-\hat u_*)$ from \eqref{Poapp} equal to the $w_0(-\hat u_*)$ in \eqref{phihatexpr} it is realized that the phase $l=l_l$ in \eqref{phihatexpr}, using here the subscript $l$ to indicate ``left'', is given as
\begin{align}
 l_l = w_0 -\epsilon^{-1} e_3- \ln(e_4 \epsilon^{-1})\hat z_0+\frac{\pi}{2},\eqlab{phil}
\end{align}
ignoring for simplicity the $\mathcal O(\delta^{3/2}\ln \epsilon^{-1})$-remainder. 
The image $P_{\text{cr}} \circ P_o(\hat z_0,
                               w_0
                              )$ therefore becomes
                              \begin{align}
                               \hat \varrho_0 &= \frac{1}{2\pi}\ln \frac{1+\vert p_l\vert^2}{2\vert \text{Im}\,p_l\vert},\eqlab{leqn1}\\
                               \phi_0&=-\frac{2\sqrt{2}}{3}\delta^{-3/2} \hat u_*^{3/2} + 3 \ln (\delta^{-1}\hat u_*)\hat \varrho_0-\theta_l,\nonumber
                              \end{align}
where 
\begin{align}
\theta_l=\theta(\hat \varrho_0,p_l)\,\, \text{with $\theta$ as in \eqref{thetaeqn}},\eqlab{theta_l}
\end{align}
$p_l=p(\hat z_0,\lambda_l)$ from \eqref{peqn} and
\begin{align}
 \lambda_l 
 &=3\hat z_0 \ln 2 -\frac{\pi}{4}-\text{arg}\,\Gamma(i\hat z_0)-l_l.\nonumber
\end{align}
Given \eqref{phil} it follows that $\lambda_l(\hat z_0,w_0)$ can be written as
\begin{align}
 \lambda_l &=-w_0+ \ln(e_4 \epsilon^{-1})\hat z_0 + G(\hat z_0),\eqlab{lambda_l}\\
 G(\hat z_0)&=3\hat z_0 \ln 2 -\frac{3\pi}{4}-\text{arg}\,\Gamma(i\hat z_0)+\epsilon^{-1} e_3.\eqlab{Geqn}
\end{align}

Finally applying $P_i$ in \eqref{Piapp} to the Eqs. \eqref{leqn1} gives the left hand side of \eqref{fpeqn} which is denoted by $(\hat \varrho_l,\phi_l)$:
                              \begin{align}
                               \hat \varrho_l &= \frac{1}{2\pi}\ln \frac{1+\vert p_l\vert^2}{2\vert \text{Im}\,p_l\vert},\eqlab{varrhol}\\
                               \phi_l&=-\sqrt{2}\epsilon^{-1}e_1 +2\ln (e_2\epsilon^{-1})\hat \varrho_l -\theta_l.\eqlab{tildephil}
                              \end{align}
                              ignoring the $\mathcal O(\delta^{3/2}\ln \epsilon^{-1})$-remainder.
                              
\subsubsection*{The right hand side of Eq. \eqref{fpeqn}}
For the right hand side, the image of $P_o\circ E$, $(\hat z_0,w_0)(\mu^{-2}\tau+\hat u_*) = P_o\circ E(\hat z_0,w_0)$, is initially computed using \eqref{Poapp}:
\begin{align}
 \hat z_0(\mu^{-2}\tau+\hat u_*) &= \hat z_0,\eqlab{rhs1}\\
 w_0(\mu^{-2}\tau+\hat u_*)&=-w_0 -\epsilon^{-1} e_3- \ln(e_4 \epsilon^{-1})\hat z_0 +\frac{2}{3}\delta^{-3/2} \hat u_*^{3/2} +\frac32\ln (\delta^{-1} \hat u_*).\nonumber
\end{align}
Here the $\mathcal O(\delta^{3/2}\ln \epsilon^{-1})$-remainder has again been left out for simplicity. The phase $w_0(\mu^{-2}\tau+\hat u_*)$ is then, much as above, set equal to the phase $w_0$ in \eqref{phihatexpr}. This gives the following expression for the phase $l=l_r$, now using the subscript $r$ to indicate ``right'':
\begin{align}
l_r = -w_0 -\epsilon^{-1} e_3- \ln(e_4 \epsilon^{-1})\hat z_0+\frac{\pi}{2}.\eqlab{phir}
\end{align}
Then upon applying $P_{\text{cr}}$ to \eqref{rhs1} I obtain
\begin{align*}
                               \hat \varrho_0 &= \frac{1}{2\pi}\ln \frac{1+\vert p_r\vert^2}{2\vert \text{Im}\,p_r\vert},\\
                               \phi_0&=-\frac{2\sqrt{2}}{3}\delta^{-3/2} \hat u_*^{3/2} + 3 \ln (\delta^{-1}\hat u_*)\hat \varrho_0-\theta_r,
                              \end{align*}
                               with $\theta_r=\theta(\hat \varrho_0,p_r)$ from \eqref{thetaeqn}, $p_r=p(\hat z_0,\lambda_r)$ as in \eqref{peqn} and
\begin{align}
 \lambda_r 
 &=3\hat z_0 \ln 2 -\frac{\pi}{4}-\text{arg}\,\Gamma(i\hat z_0)-l_r.\nonumber
\end{align}
Given \eqref{phir} it follows that $\lambda_r=\lambda_r(\hat z_0,w_0)$ takes the following form:
\begin{align}
 \lambda_r = w_0+\ln (e_4\epsilon^{-1})\hat z_0+G(\hat z_0),\eqlab{lambda_r}
\end{align}
with $G$ as in \eqref{Geqn}. 
Finally, $E\circ P_i$ is applied to this, using the approximation \eqref{Piapp} for $P_i$. This finally gives the right hand side of \eqref{fpeqn} which is denoted by $(\hat \varrho_r,\phi_r)$
\begin{align}
 \hat \varrho_r &= \frac{1}{2\pi}\ln \frac{1+\vert p_r\vert^2}{2\vert \text{Im}\,p_r\vert},\eqlab{varrhor}\\
 \phi_r&=\sqrt{2} \epsilon^{-1} e_1-2\ln(e_2 \epsilon^{-1})\hat \varrho_r +\theta_r.\eqlab{tildephir}
\end{align}
ignoring the $\mathcal O(\delta^{3/2}\ln \epsilon^{-1})$-remainder.
%
The equation \eqref{fpeqn} therefore becomes $\hat \varrho_l(\hat z_0,w_0)=\hat \varrho_r(\hat z_0,w_0)$, $\phi_l(\hat z_0,w_0)=\phi_r(\hat z_0,w_0)$, see \eqref{varrhol}, \eqref{tildephil}, \eqref{varrhor}, and \eqref{tildephir}.
\subsubsection*{Solving $\hat \varrho_l =\hat \varrho_r$ using \eqsref{varrhol}{varrhor}}
The absolute value of $p$ depends, cf. \eqref{peqn}, only on $\hat z_0$. Hence $\vert p_l\vert =\vert p_r\vert$. Setting 
$$ \hat \varrho_0\equiv \hat \varrho_l(\hat z_0,w_0)=\hat \varrho_r(\hat z_0,w_0),$$ I therefore conclude that
\begin{align}
\sin \lambda_l = \pm \sin \lambda_r,
 \end{align}
and hence
 \begin{align}
  \lambda_r = \left\{\begin{array}{c}
                        \pm \lambda_l\\
                       \pi \pm \lambda_l
                      \end{array}\right.,\eqlab{phiprl}
 \end{align}
 with $\lambda_l(\hat z_0,w_0)$ and $\lambda_r(\hat z_0,w_0)$ given by Eqs. \eqsref{lambda_l}{lambda_r}, respectively.
 Before continuing with solving the equation:
 \begin{align*}
  \phi_l = \phi_r\quad \text{mod}\,2\pi,
 \end{align*}
 I first compute the trace of the Jacobian matrix $\partial P = \partial_{(\hat z_0,w_0)}P$ of the truncation of the mapping. For this Eq. \eqref{phiprl} will be used. I compute this trace for two reasons: (i) To decide when the contraction mapping theorem can be applied to conclude that the solutions can be continued into true solutions of the non-truncated equations. (ii): To investigate the stability of the periodic orbits. 
 \subsubsection*{The Jacobian of $P$}
The function $\eta$ is locally constant so it can be ignored completely in the calculations. 
\begin{lemma}\lemmalab{jac}
Suppose $1+p^2\ne 0$. If (i) $\lambda_r=\lambda_l$ or $\pi+\lambda_l$ then
\begin{align}
  \textnormal{tr}\,\left(\partial P\right) &=2+4B\bigg((2\ln (e_2\epsilon^{-1})-q)(\ln (e_4 \epsilon^{-1})+g)B\nonumber\\
  &+A(2\ln (e_2\epsilon^{-1})-q)+D(\ln(e_4\epsilon^{-1})+g)+C\bigg).\eqlab{tr1}
 \end{align}
 Here
  \begin{align}
g&=g(\hat z_0)=\partial_{\hat z_0} G=3\ln 2-\partial_{\hat z_0}\textnormal{arg}\,\Gamma(i\hat z_0),\eqlab{weqn}
\end{align}
and
\begin{align}
q=q(\hat \varrho_0)=\partial_{\hat \varrho_0} Q = 7\ln 2-\partial_{\hat \varrho_0}\textnormal{arg}\,\Gamma (2i\hat \varrho_0),\eqlab{qeqn}
\end{align}
are the derivatives of $G=G(\hat z_0)$ \eqref{Geqn} and $Q=Q(\hat \varrho_0)$ \eqref{Qeqn} with respect to $\hat z_0$ and $\hat \varrho_0$, respectively. Moreover
\begin{align}
 A &=\partial_{\hat z_0} \varrho_0(\hat z_0,\lambda)= \frac{\vert p\vert^2+1}{2\vert p\vert^2},\eqlab{aeqn}\\
 B&=\partial_{\lambda} \varrho_0(\hat z_0,\lambda)=-\frac{1}{2\pi}\cot \lambda,\eqlab{beqn}
 \end{align}
 with $\varrho_0$ from \eqref{Jthetaexpr}, and
 \begin{align}
 C &= \partial_{\hat z_0} (\textnormal{arg}\,(1+p^2))=\frac{2\pi(1+\vert p\vert^2)\sin (2\lambda)}{\vert 1+ p^2\vert^2},\eqlab{ceqn}\\
 D&=\partial_{\lambda} (\textnormal{arg}\,(1+p^2))=\frac{2\vert p\vert^2(\cos (2\lambda)+\vert p\vert^2)}{\vert 1+p^2\vert^2}.\eqlab{deqn}
\end{align}
These functions are all evaluated for $p=p_l(\hat z_0,\lambda)$ and $\lambda=\lambda_l(\hat z_0,w_0)$ in \eqref{tr1}.
  On the other hand, if (ii) $\lambda_r=-\lambda_l$ or $\pi-\lambda_l$ then 
 \begin{align}
  \textnormal{tr}\,\left(\partial P\right) = 2-4(2\ln (e_2 \epsilon^{-1})-g)(\ln (e_4\epsilon^{-1})+q)B^2,\eqlab{tr2}
 \end{align}
 with $B$ \eqref{beqn} evaluated at $\lambda=\lambda_l(\hat z_0,w_0)$.
\end{lemma}
\begin{remark}
 Note that the lemma, cf. \eqref{phiprl}, considers all of the four different scenarios. Note also that $A=A(\hat z_0)$ in \eqref{aeqn} only depends upon $\hat z_0$ since $\vert p\vert$, cf. \eqref{peqn}, is independent of $\lambda$
\xqed{\lozenge}\end{remark}
\begin{proof}
I use the expressions for $(\hat \varrho_l,\phi_l)$ in \eqsref{varrhol}{tildephil} and $(\hat \varrho_r,\phi_r)$ in \eqsref{varrhor}{tildephir}. For case (ii) I also use that the functions $A$ and $D$ are even with respect to $\lambda$. The functions $B$ and $C$ are, on the other hand, odd. 
\qed\end{proof}
Eqs. \eqsref{tr1}{tr2} are asymptotically of the form:
\begin{align}
  \textnormal{tr}\,\left(\partial P\right)- 2 =\mp 8 \ln^2(\epsilon^{-1}) B^2+\mathcal O(\ln \epsilon^{-1}),\eqlab{tr12asymp}
\end{align}
respectively.
Therefore
\begin{lemma}\lemmalab{stability}
 Consider all solutions of \eqref{fp} with $\lambda_l$ belonging to $[c^{-1},\pi/2-c^{-1}]\cup [\pi/2+c^{-1},\pi-c^{-1}]$ or a $\pi$-translation of this set. These periodic orbits are all unstable for $\epsilon$ sufficiently small.
\end{lemma}
\begin{proof}
 A sufficient condition for instability is that 
 \begin{align}
 \vert \text{tr}\,\partial P \vert >2,\eqlab{trg2}
 \end{align}
 using also here that $\text{tr}\,(\partial P)^2  =\text{tr}^2\,\partial P-2$. For the considered $\lambda_l$-values $\vert B\vert \ge c_1^{-1}>0$, $c_1$ independent of $\epsilon$, and so \eqref{trg2} can always be archived by taking $\epsilon$ sufficiently small cf. \eqref{tr12asymp}.  
 \qed\end{proof}
Stable orbits can therefore only occur if $\lambda_l$ is near $\pi/2$ or $3\pi/2$ where $B$ \eqref{beqn} is small. I consider this in \secref{stable}. 

\subsection{Unstable solutions - Part $1^\circ$ of the main result}
In this section I will find unstable periodic orbits. I continue from \eqref{phiprl} and start by dividing the presentation into two separate cases: 
\begin{itemize}
\item case (i) where $\lambda_r=\lambda_l$ or $\lambda_r=\pi+\lambda_l$;
\item case (ii) where $\lambda_r=-\lambda_l$ or $\lambda_r=\pi-\lambda_l$.
\end{itemize}
These two cases cover all the possible scenarios. They also correspond to solutions with different symmetry properties (see \secref{symmetry} below).

Consider first case (i) and $\lambda_r=\lambda_l$. Then by using \eqsref{lambda_l}{lambda_r}
\begin{align*}
 w_0 = 0\quad \text{or}\quad \pi.
\end{align*}
If $\lambda_r=\pi+\lambda_l$ then
\begin{align*}
 w_0 = \pi/2\quad \text{or}\quad 3\pi/2.
\end{align*}
It follows that if
\begin{align*}
 w_0=0\quad \text{or}\quad \pi/2\quad \text{mod}\,\pi,
\end{align*}
then $\lambda_r=\lambda_l\,\text{mod}\,\pi$. Also by \eqref{peqn}
\begin{align}
 p_r = \pm p_l.\eqlab{prpli}
\end{align}

On the other hand, for case (ii) with $\lambda_r=-\lambda_l$ then
\begin{align*}
2\ln (e_4\epsilon^{-1}) \hat z_0 + 2G(\hat z_0) =0\quad \text{mod}\,2\pi,
\end{align*}
using \eqsref{lambda_l}{lambda_r}. 
For $\lambda_r=\pi-\lambda_l$ then I similarly get
\begin{align*}
2\ln (e_4\epsilon^{-1}) \hat z_0 + 2G(\hat z_0) =\pi \quad \text{mod}\,2\pi.
\end{align*}
It therefore follows that solutions to the equation
\begin{align*}
2\ln (e_4\epsilon^{-1}) \hat z_0 + 2G(\hat z_0) = 0\quad \text{mod}\,\pi,
\end{align*}
solve $\lambda_r=-\lambda_l\,\text{mod}\,\pi$. Also
\begin{align}
 p_r=\pm \overline{p}_l,\eqlab{prplii}
\end{align}
cf. \eqref{peqn}.

 Next, I consider the equation $\phi_l=\phi_r$ and use \eqsref{tildephil}{tildephir}. In case (i) I use \eqref{prpli} so that $\theta_r=\theta_l$, $\theta=\theta(\hat \varrho_0,p)$ \eqref{thetaeqn} depending only on $p^2$, and conclude that
 \begin{align*}
 -\sqrt{2}\epsilon^{-1} e_1+2\ln (e_2\epsilon^{-1})\hat \varrho_0-\theta_l=0\quad \text{mod}\,\pi,
 \end{align*}
with $\theta_l=\theta(\hat \varrho_0,p_l)$, $p_l=(e^{2\pi\hat z_0}-1)^{1/2}e^{i\lambda_l}$. 

For (ii) I similarly use \eqref{prplii} so that $\theta_r=Q-\text{arg}\,(1+p_r^2)=Q+\text{arg}\,(1+p_l^2)$ 
and therefore
 \begin{align*}
  -\sqrt{2}\epsilon^{-1} e_1+2\ln (e_2\epsilon^{-1})\hat \varrho_0-Q=0\quad \text{mod}\,\pi,
 \end{align*}
 with $Q=Q(\hat \varrho_0)$ \eqref{Qeqn}. 
I collect the results in the following proposition:
\begin{prop}\proplab{fixPointProp}
 Periodic orbits can be obtained as
 \begin{itemize}
  \item[(i)] 
  \begin{align}
  w_0&=0,\,\pi/2\,\text{mod}\,\pi,\eqlab{hatphi0i}\\
   F^{(2)}_i(\hat z_0) &= 0\,\quad \text{mod}\,\pi,\eqlab{tildephilreqni}
  \end{align}
  where 
  \begin{align}
   F^{(2)}_i(\hat z_0)\equiv -\sqrt{2}\epsilon^{-1} e_1&+2\ln (e_2\epsilon^{-1})\hat \varrho_0(\hat z_0,\lambda_l(\hat z_0))-\theta_l(\hat z_0),\eqlab{Fi2eqn}
  \end{align}
  with $\theta_l$ from \eqref{theta_l}, $\lambda_l$ as in \eqref{lambda_l} here both viewed as a function of $\hat z_0$ only, given here that the value of $w_0\, \text{mod}\,\pi/2$ is fixed by \eqref{hatphi0i}. 
\item[(ii)]  
\begin{align}
F_{ii}^{(1)} (\hat z_0) = 0 \quad \text{mod}\,\pi,\eqlab{tildeJeqnii}\\
F_{ii}^{(2)} (\hat z_0,\lambda_l) = 0\quad  \text{mod}\,\pi,\eqlab{tildephilreqnii}
\end{align}
where
\begin{align*}
 F_{ii}^{(1)}(\hat z_0) &\equiv 2\ln(e_4\epsilon^{-1}) \hat z_0+ 2G(\hat z_0),\\
 F_{ii}^{(2)}(\hat z_0,\lambda) &\equiv-\sqrt{2}\epsilon^{-1} e_1+2\ln (e_2\epsilon^{-1})\hat \varrho_0-Q(\hat z_0).
\end{align*}
 \end{itemize}
The action $\hat \varrho_0$ is given by \eqref{varrhol}:
\begin{align}
 \hat \varrho_0 &=\frac{1}{2\pi} \ln \frac{1+\vert p_l\vert^2}{2\vert \textnormal{Im}\,p_l\vert} =\frac{\hat z_0}{2}+\frac{1}{4\pi}\ln \left(1-e^{-2\pi\hat z_0}\right)^{-1} - \frac{1}{2\pi}\ln (2\vert \sin \lambda_l\vert).\eqlab{varrho0}
\end{align}
Here $\lambda_l = \lambda_l(\hat z_0,w_0)$ from \eqref{lambda_l}.
\end{prop}
I will address symmetry properties in the following section.
 \subsection{Symmetry properties}\seclab{symmetry}
To any solution $(x,y)(u)$ there exists two solutions obtained by applying the symmetry $\mathcal R$ and the time-reversible symmetry (A3) $\mathcal T_{\tau}$:
\begin{align*}
 \mathcal R(x,y,u)(t) = (-x,-y,u)(t),\quad \mathcal T_{\tau}(x,y,u)(t) = (x,-y,\tau-u)(-t).
\end{align*}
\begin{prop}\proplab{symmetry}
Generic periodic orbits obtained from case (i) corresponds orbits that are symmetric with respect to $\mathcal R$ and $\mathcal T_{\tau}$ ($4\pi \epsilon^{-1}$-periodic) or just $\mathcal T_{\tau}$ ($2\pi \epsilon^{-1}$-periodic). Case (ii) consists of non-symmetric $2\pi\epsilon^{-1}$-periodic orbits or $\mathcal R$-symmetric $4\pi \epsilon^{-1}$-periodic orbits.
\end{prop}
\begin{proof}
Take $t=0$ to be the time when $u=-\pi+\tau/2$. In case (i): $w_0=0,\,\pi/2\,\text{mod}\,\pi$ up to an error $\mathcal O(\delta^{3/4}\ln \epsilon^{-1})$. This means that either $x(0)=0$ or $y(0)=0$ at $t=0$. Consider a periodic orbit with $y(0)=0$ first. This is the case when $\lambda_r=\lambda_l$ so the periodic orbit is a fix point of $P$. The initial condition is near $(x(0),0,-\pi+\tau/2)$. 
The initial condition for the solution $$\mathcal T_{\tau}(x,y,u)=(x,-y,\tau-u)(-t),$$ is near
\begin{align*}
 (x(0),0,\tau-(-\pi+\tau/2)) = (x(0),0,\pi+\tau/2)=(x(0),0,-\pi+\tau/2).
\end{align*}
By local uniqueness of generic periodic orbits it follows that the periodic orbit coincides with its symmetry-related $\mathcal T_{\tau}$-periodic orbit.

Consider next $x(0)=0$. This is the case when $\lambda_r=\pi+\lambda_l$ so the periodic orbit is a fix point of $P^2$. The initial condition is near $(0,y(0),-\pi+\tau/2)$. The image of $P$ is near $(0,-y(0),-\pi+\tau/2)$. The initial condition for the solution $\mathcal R(x(t),y(t),u(t))=(-x(t),-y(t),u(t))$ is near
\begin{align*}
 (0,-y(0),-\pi+\tau/2).
\end{align*}
Similarly the initial condition for the solution $$\mathcal T_{\tau}(x(t),y(t),u(t))=(x(-t),-y(-t),\tau-u(-t)),$$ is near
\begin{align*}
 (0,-y(0),\tau-(-\pi+\tau/2))=(0,-y(0),\pi+\tau/2) = (0,-y(0),-\pi+\tau/2).
\end{align*}
By local uniqueness of generic periodic orbits it follows that the periodic orbit is $\mathcal R$ and $\mathcal T_{\tau}$-symmetric. The argument for the $\mathcal R$-symmetry can be modified so that it also applies to the $4\pi \epsilon^{-1}$-periodic orbits in case (ii).
\qed\end{proof}

\begin{cor}
 Generic symmetric periodic orbits from case (i) satisfy the $\mathcal T_{\tau}$-symmetry condition
 \begin{align*}
  (x,y,u)(t)=(x,-y,\tau-u)(-t),
 \end{align*}
 when they are $2\pi\epsilon^{-1}$-periodic, 
or a translated version
\begin{align*}
 (x,y,u)(t+2\pi\epsilon^{-1})=(x,-y,\tau-u)(-t),
\end{align*}
when they are $4\pi\epsilon^{-1}$-periodic. 
In the latter case, the following $\mathcal R$-symmetry condition also applies
\begin{align*}
 (x,y,u)(t+2\pi\epsilon^{-1})=(-x,-y,u)(t),
\end{align*}
\end{cor}
Note that these conditions imply that $y=0$ at $u=\tau/2$ in agreement with the fact that \eqref{tildephilreqni} is just the condition that $\phi_l=0\,\text{mod}\,\pi$ cf. \eqref{tildephil}. 

I then show the following:
\begin{theorem}
 Consider $\hat z_0\in [c_1^{-1},c_2]$ and $\lambda_l(\hat z_0,w_0)\in [c^{-1},\pi-c^{-1}] \cup [\pi+c^{-1},2\pi-c^{-1}]$. Then for both cases (i) ($\mathcal R$ and $\mathcal T_{\tau}$-symmetric if they have period $4\pi\epsilon^{-1}$ or just $\mathcal T_{\tau}$ if they have period $2\pi\epsilon^{-1}$) and (ii) ($\mathcal R$-symmetric or non-symmetric) the following statement holds true: For $\epsilon$ sufficiently small there exist $\mathcal O(\ln^2 \epsilon^{-1})$-many unstable periodic orbits. The characteristic multipliers are $\mathcal O(\ln^{\pm 2} \epsilon^{-1})$. 
\end{theorem}
\begin{remark}
 These orbits are unstable but cf. the estimate for the characteristic multipliers, the separation is much more modest in comparison with the typical exponential separation from the trivial periodic orbit $x=0=y$ (see also the appendix in \cite{nei09}). This result gives $1^\circ$ in our main result.
\xqed{\lozenge}\end{remark}
\begin{proof}
\subsubsection*{Case (i)}
I first consider case (i). Given that $w_0$ is already determined by \eqref{hatphi0i},
I am to solve the equation \eqref{tildephilreqni}
for $\hat z_0$. The dependency on $\hat z_0$ enters e.g. through $\hat \varrho_0$ in \eqref{varrho0} with $\lambda_l=\lambda_l(\hat z_0)$ as in \eqref{lambda_l}.  
I need to exclude those $\lambda_l\notin (-c^{-1},c^{-1})\,\text{mod}\,\pi$ where $P_{cr}$ is not defined, and those $\lambda_l \notin (\pi/2-c^{-1},\pi/2+c^{-1})$ where $\text{tr}\,\partial P$ is near $\pm 2$. Following
\begin{align}
 \partial_{\hat z_0}\lambda_l = \ln \epsilon^{-1}+\mathcal O(1),\eqlab{dlambdadz0}
\end{align}
I realise that this gives rise to the exclusion of intervals within the interval $[c_1,c_2]$ of $\hat z_0$-values which have widths of order $c^{-1}\ln^{-1} \epsilon^{-1}$. The complement contains closed intervals with widths of order $\ln^{-1}\epsilon^{-1}$. There is an order of $\ln \epsilon^{-1}$ of such strips given an order $1$ measure set of $\hat z_0$-values. Now within each such strip for which $\lambda_l\notin (\pi/2-c^{-1},\pi/2+c^{-1})$ I have
\begin{align}
 (F_i^{(2)})'(\hat z_0) = -\frac{1}{\pi}\cot (\lambda_l) \ln^2 (\epsilon^{-1})+\mathcal O(\ln \epsilon^{-1}).\eqlab{Fi2p}
\end{align}
Given that $\cot \lambda_l \ne 0$ it therefore follows that solutions of $F_i^{(2)}(\hat z_0)=0$ $\text{mod}\,\pi$, see \eqref{tildephilreqni}, exist within each strip. These solutions are separated by a distance of order $\ln^{-2}\epsilon^{-1}$. Furthermore, given that the order of the width of each strip is $\ln^{-1} \epsilon^{-1}$ there is an order of $\ln \epsilon^{-1}$-many solutions within each strip. In total there are $\mathcal O(\ln^2\epsilon^{-1})$-many solutions as claimed.

These solutions can all be continued into true solutions by applying the contraction mapping theorem taking $\epsilon$ sufficiently small.
\subsubsection*{Case (ii)}
In this case I have to solve two equations: \eqsref{tildeJeqnii}{tildephilreqnii} for $\hat z_0$ and $w_0$. Cf. \eqref{pphase} I can, however, solve for $\hat z_0$ and $\lambda=\lambda_l$ instead.
Firstly
\begin{align*}
 (F_{ii}^{(1)})'(\hat z_0) = 2\ln (e_4\epsilon^{-1}) + \mathcal O(1),
\end{align*}
and so there exists an order $\ln \epsilon^{-1}$ many solutions $\hat z_0$ to \eqref{tildeJeqnii}. These are separated by a distance of order $\ln^{-1}\epsilon^{-1}$ within the order $1$ set of $\hat z_0$-values. For each solution to \eqref{tildeJeqnii} I will solve \eqref{tildephilreqnii} with respect to $\lambda$. 
Again I compute the derivative
\begin{align*}
 \partial_{\lambda} F_{ii}^{(2)} = -\frac{1}{\pi} \cot (\lambda)\ln \epsilon^{-1}  +\mathcal O(1),
\end{align*}
using \eqref{varrho0}. 
Upon excluding a region around $\lambda=\pi/2$ and $3\pi/2$ where $\cot \lambda=0$, the same reasoning, as used above, can be applied to conclude the existence of $\mathcal O(\ln \epsilon^{-1})$-many solutions for each  solution $\hat z_0$ of \eqref{tildeJeqnii}. In total there is therefore an order of $\ln^{2}\epsilon^{-1}$-many solutions. 

Again, the solutions can be continued into true solutions by applying the contraction mapping theorem for $\epsilon$ sufficiently small.

The solutions are unstable cf. \lemmaref{stability}. From this also follows the estimates of the characteristic multipliers. 
\qed\end{proof}

From now I will focus on stable solutions.
\subsection{Stable solutions}\seclab{stable}
I focus on $\lambda=\lambda_l(\hat z_0)$ near $\pi/2$ and case (i) where $\lambda_r = \lambda_l \,\text{mod}\,\pi$ and $w_0=0,\,\pi/2\,\mod \pi$. Here $\lambda_l$ is again given by \eqref{lambda_l} with $w_0=0,\,\pi/2\,\text{mod}\,\pi$. Case (ii) can be handled in a similar way. Following \eqref{tr12asymp} I wish to consider $\lambda$ of the form 
 \begin{align}
  \lambda = \frac{\pi}{2} + \hat \lambda \ln^{-1} (\epsilon^{-1}). \eqlab{xieqn}
 \end{align}
\begin{lemma}\lemmalab{xicond}
Let $c$ be large and consider $\hat z_0>c^{-1}$ and $\hat z_0 \notin (\ln 2/(2\pi)-c^{-1},\ln 2/(2\pi) +c^{-1})$. Then for $\epsilon$ sufficiently small, all period orbits obtained from case (i) with $\lambda=\lambda_l(\hat z_0)$, with $\lambda$ as in \eqref{xieqn} and where $\hat \lambda$ satisfies
\begin{align}
 \hat \lambda  &\in 
\left\{\begin{array}{cc}
\left[-2\pi (2A+D_1)+c^{-1},-c^{-1}\right]& \text{for}\, \, 2A+D_1>0\\
\left[c^{-1},-2\pi (2A+D_1)-c^{-1}\right]& \text{for}\, \, 2A+D_1<0
                           \end{array}\right.,\eqlab{xiint}
                           \end{align}
                           are stable. Here 
                           $$D_1(\hat z_0) \equiv D(\hat z_0,\pi/2),$$
                           and 
                            \begin{align}
                            (2A+D_1) &= \frac{3e^{4\pi \hat z_0}-6e^{2\pi \hat z_0}+2}{(e^{2\pi \hat z_0}-1)(e^{2\pi \hat z_0}-2)}\nonumber\\
                            &=3+3e^{-2\pi \hat z_0}+\mathcal O(e^{-4\pi \hat z_0}),\eqlab{2ADasymp}\\
\end{align}
for $\hat z_0$ large, where $A$ and $D$ are given in \eqsref{aeqn}{deqn}, respectively.
\end{lemma}
\begin{remark}
It is important to note, when comparing with the equations in \cite{nei06,nei09}, the difference that our equations depend on the pseudo-angle $\lambda$ (denoted by $\pi \eta$ in \cite{nei06,nei09}) strongly due to the factor $\ln \epsilon^{-1}$ of $\ln \vert \sin \lambda \vert$. Moreover, their $\gamma_1$ and $\gamma_2$ are both $\mathcal O(\ln \epsilon^{-1})$ in our case. Increasing $\gamma_1$ and $\gamma_2$ has the consequence of diminishing the stability region cf. Eq. (41) in \cite{nei09} and in agreement with \lemmaref{xicond}.
\xqed{\lozenge}\end{remark}
\begin{proof}
 Eq. \eqref{xieqn} is inserted into \eqref{beqn} which is then used in \eqref{tr1}. This gives
 \begin{align*}
 \text{tr}\,\partial P = 2+\frac{2\hat \lambda}{\pi}\left(\frac{\hat \lambda}{\pi} + 2A+D_1\right)+\mathcal O(\ln^{-1}\epsilon^{-1}).
 \end{align*}
 A sufficient condition for stability is that $\vert \text{tr}\,\partial P\vert<2$. Solving this inequality for $\hat \lambda$ one can then verify the statement about the stability for $c$ large. 
 \qed\end{proof}
I will in the following prove $2^\circ$ and $3^\circ$ of the main theorem and therefore focus, as in $1^\circ$, on $\hat z_0\in [c_1^{-1},c_2]$ with $\ln 2/(2\pi) \notin [c_1^{-1},c_2].$ 

To solve for stable orbits I proceed as above but this time I start by solving $\lambda_l(\hat z_0)=\lambda$ for $\hat z_0$ where the right hand side $\lambda$ is as in \eqref{xieqn}. Following \eqref{dlambdadz0} this gives $\mathcal O(\ln \epsilon^{-1})$-many intervals of lengths at least $\pi \vert 2A+D_1\vert \ln^{-2}\epsilon^{-1}$. These intervals are separated by intervals of lengths $\mathcal O(\ln^{-1}\epsilon^{-1})$ where $\vert\text{tr}\,\partial P\vert >2$. Also since $$(F_{i}^{(2)})'(\hat z_0)=-\ln (\epsilon^{-1}) (2A+D_1) +\mathcal O(1),$$ when $\lambda_l(\hat z_0)$ is as in \eqref{xieqn}, these intervals will under $F_{i}^{(2)}$ be mapped into $\mathcal O(\ln \epsilon^{-1})$-many intervals in $\mathbb R/(\pi\mathbb Z)$ of lengths $\pi \vert 2A+D_1\vert^2\ln^{-1}\epsilon^{-1}$. These intervals are separated by lengths 
\begin{align}
\pi (2A+D_1)+\mathcal O(\ln^{-1}\epsilon^{-1}).\eqlab{sep}
\end{align}
Stable solutions are then the consequence 
of 
the intersection of these mapped intervals
with $0$ cf. \eqref{tildephilreqni}. 
The stable solutions are therefore rare compared to the unstable ones: There can be at most $\ln \epsilon^{-1}$, in contrast to $\ln^2 \epsilon^{-1}$, but this is clearly very optimistic. Before supporting this claim by further analysis I first present some numerics for \eqref{toy}.
\subsection{Numerics for Eq. \eqref{toy}}\seclab{numerics}
\figref{stablesolii} shows the number of stable periodic solutions for \eqref{toy} 
within $\hat z_0\in [0.12,2]$ and $1000$ different values of $\ln^{-1} \epsilon^{-1}$. The distribution of solutions are shown in \figref{stablesolii2}. The solutions are obtained from the truncations of \eqsref{tildephilreqni}{xiint}. To continue them into true solutions one needs to invoke the contraction mapping theorem. For some of the extremely small values of $\epsilon$ considered it is difficult if \textit{not} impossible to integrate the equations directly. Out of the $1000$ different values of $\ln^{-1} \epsilon^{-1}$ the case with no stable solutions occurred $368$ times. In $632$ of the cases I found at least one stable solution. 

Based on these observations, one could be led to the following conjecture:
\begin{con}\conlab{con1}
There exists an $\epsilon_0$ and a number $N$ so that the following holds true: For almost all $\epsilon\le \epsilon_0$ the number of stable solutions is less then $N$.
\end{con}

\newpage 

\begin{figure}\begin{center}\includegraphics[width=.9\textwidth]{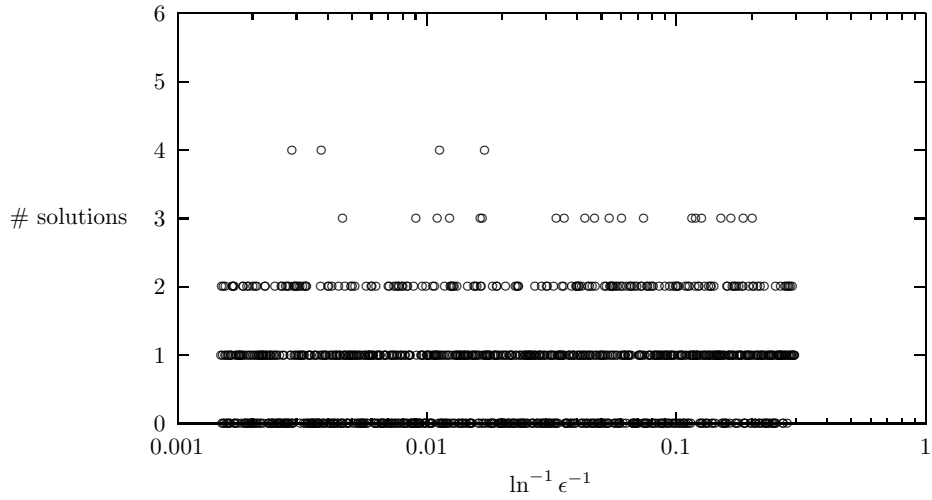}
\caption{The number of stable solutions for $\hat z_0\in [0.1,2]$ for $1000$ different values of $\ln^{-1} \epsilon^{-1}$ are shown with points.}
\figlab{stablesolii}
                		\end{center}
              \end{figure}

\begin{figure}[h!]
\begin{center}
{\includegraphics[width=.9\textwidth]{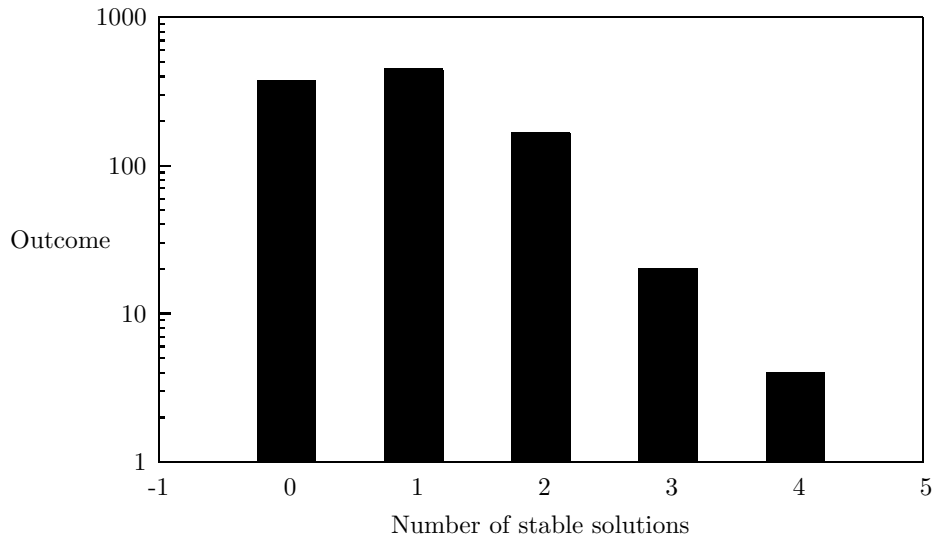}}
\end{center}
\caption{The number of outcomes for the different numbers of stable solutions. A total of $1000$ different values of $\ln^{-1} \epsilon^{-1}$ were considered.}
\figlab{stablesolii2}
\end{figure}

\newpage

In \tabref{tbl} I have documented the result from computing $\mathcal T_{\pi}$-symmetric, $2\pi\epsilon^{-1}$-periodic solutions of \eqref{toy}. As in \cite{nei09} I have used step-size control in the initial conditions $(x,0)$ to carefully scan for periodic orbits in the interval $x\in (0,0.5]$. I have used a $2$-stage fully implicit Gauss-Legendre symplectic method for the time integration. POS in the second column of \tabref{tbl} is the number of periodic orbits. SPOS in the third column is for the number of periodic orbits. The fourth column gives the number of stable periodic orbits for $x\in (0,2\epsilon^{1/2}]$. The final column gives the number of unstable periodic orbits for $x\in (0,2\epsilon^{1/2}]$. The upper value $x=2\epsilon^{1/2}$ corresponds cf. \remref{remx} to the upper bound $\hat z_0=2$ considered above.

{\begin{table}
\begin{tabular}{|c||c|c|c|c|}
 \hline
$\epsilon$ & POS & SPOS& SPOS $x\in (0,2\epsilon^{1/2}]$ & UPOS $x\in (0,2\epsilon^{1/2}]$ \\
\hline
\hline
$0.08$ &  $33$ &  $0$ & $0$ &  $33$\\
\hline
$0.04$ &  $69$ &  $2$ & $1$ & $46$\\
\hline
$0.02$ &  $154$ &  $5$ & $0$ & $64$\\
\hline
$0.01$ &  $298$ &  $8$ & $0$ & $72$\\
\hline
$0.005$ &  $583$ &  $18$ & $0$ & $84$\\
\hline
$0.0025$ &  $1118$ &  $35$ & $0$ & $108$\\
\hline
\end{tabular}
\caption{Distribution of $\mathcal T_{\pi}$-symmetric, $2\pi\epsilon^{-1}$-periodic orbits for \eqref{toy}.}\tablab{tbl}
\end{table}}
In agreement with the analysis and the computations above the stable periodic orbits are truly rare within this interval. Also, in agreement with the results from \cite{nei97,nei06} the total number of periodic orbits, the number of stable period orbits, and the unstable ones all behave like $\epsilon^{-1}$: The second and third column almost doubles when $\epsilon$ is halved. Finally, in agreement with the results of this paper the number of unstable periodic orbits within $x\in (0,2\epsilon^{1/2}]$, the last column in \tabref{tbl}, behaves like $\ln^{2}\epsilon^{-1}$. A comparison is shown in \tabref{tbl2}.

{\begin{table}
 \begin{center}
\begin{tabular}{|c||c|c|c|c|}
 \hline
$\epsilon$ & UPOS $x\in (0,2\epsilon^{1/2}]$ & UPOS-fit: $\lfloor 2.39 \ln^2 \epsilon^{-1}\rfloor+21$& Relative Error\\
\hline
\hline
$0.08$ &  $33$ & $36$ &$9.1\%$\\
\hline
$0.04$ &  $46$ & $45$ &$2.2\%$\\
\hline
$0.02$ &  $64$ &$57$ &$11\%$\\
\hline
$0.01$ &  $72$ & $71$& $1.4\%$ \\
\hline
$0.005$ &  $84$ & $88$ &$4.8\%$\\
\hline
$0.0025$ & $108$ & $106$ & $1.9\%$\\
\hline
\end{tabular}
\caption{Comparison of the number of unstable periodic orbits in $x\in (0,2\epsilon^{1/2}]$ with the number predicted by the theory (third column via linear fit).}\tablab{tbl2}
\end{center}
\end{table}}
%
%

\subsection{Part $2^\circ$ of the main result}\seclab{part2}
To prove part $2^\circ$ of the main result, suppose that an interval in $\mathbb R/(\pi \mathbb Z)$, arising as the image under $F_{i}^{(2)}$ of a $\hat z_0$-interval with $\lambda_l(\hat z_0)$ as in \eqref{xieqn}, intersects with $0$. Then a stable solution exists cf. \propref{fixPointProp} and \lemmaref{xicond}. Let $\hat z_0^0$ denote the left end-point of the $\hat z_0$-interval and denote by $f_0$ the image of $\hat z_0$ so that $f_0=F_{i}^{(2)}(\hat z_0^0)$. The $f_0$ is an end-point of the mapped interval $\mathbb R/(\pi \mathbb Z)$ since $(F_{i}^{(2)})'(\hat z_0)\ne 0$. The end-points of the following $\hat z_0$-intervals are denoted by $\hat z_0^n$ and similarly $f_n$ will denote the image of $\hat z_0^n$ under $F_{i}^{(2)}$. Here $n=1,\ldots, N$ with $N=\mathcal O(\ln \epsilon^{-1})$. This induces a mapping of the following form
\begin{align}
f_{n+1} &= f_{n}+\varpi(\hat z_0^{n}) \quad \text{mod}\,\pi,\eqlab{gng0}\\
\hat z_0^{n+1} &=\hat z_0^{n}+\pi \ln^{-1} \epsilon^{-1} +\mathcal O(\ln^{-2} \epsilon^{-1}).\nonumber
\end{align}
In the following I will obtain an upper bound for the number steps required to obtain another solution. To obtain another solution, say $\hat z_0^n$, within the following $N$ steps it is necessary for $f_n$ to be within a distance of order $\ln^{-1}\epsilon^{-1}$ of $f_0$. The answer depends on the arithmetic properties of the number $\varpi(\hat z_0^0)$. Consider therefore $c_1>0$ large and $d>1$ and the following set of \textit{Diophantine} numbers
\begin{align*}
 D_{c_1,d} = \{z\in [0,1)\vert \vert jz-i\vert \ge \frac{c_1^{-1}}{j^{d}},\,i,\,j\in \N\}.
\end{align*}
They have \textit{almost} full measure: $1-c_2(d)c_1^{-1}$ given that $c_1$ is large. Here the $d$ dependency in $c_2(d)$ enters through a factor of $\sum_{j=1}^\infty j^{-d}$. This is why $d>1$. 
\begin{prop}\proplab{distsol}
Suppose, with little loss of generality, that $$D_{c_1,d} \ni\pi^{-1}\varpi(\hat z_0^0)\,\textnormal{mod}\,1,$$ with $c_1$ large and $d>1$. Then
\begin{align*}
 \vert f_n - f_0 \vert \gg \ln^{-1}\epsilon^{-1},
\end{align*}
for all $n\le \lfloor c_3^{-1} \ln^{(2+d)^{-1}} \epsilon^{-1}\rfloor$ with $c_3$ large. 
\end{prop}
\begin{proof}
 I linearize \eqref{gng0} about $\hat z_0^0$ and use the Diophantine property to obtain the following
\begin{align*}
 \vert f_n - f_0 \vert = \vert n\varpi (\hat z_0^0) + \mathcal O(n^2 \ln^{-1} \epsilon^{-1})\vert \ge \frac{c_1^{-1}}{\pi m^{d}}+\mathcal O(m^2 \ln^{-1}\epsilon^{-1}).
\end{align*}
for all $n\le m\ll \mathcal O(\ln\epsilon^{-1})$. Setting $m=\lfloor c_3^{-1} \ln^{(2+d)^{-1}} \epsilon^{-1}\rfloor$ then implies that $m^{-d}\gg m^2 \ln^{-1}\epsilon^{-1}$ provided $c_3$ is large enough. From here also follows that
\begin{align*}
  \vert f_n - f_0 \vert \ge \frac{c_1^{-1}}{2\pi m^{d}}\gg \ln^{-1}\epsilon^{-1},
\end{align*}
for all such $n\le m=\lfloor c_3^{-1} \ln^{(2+d)^{-1}} \epsilon^{-1}\rfloor$. This completes the proof.
\qed\end{proof}
The values $\hat z_0^n$ are separated by the distance $\pi \ln^{-1} \epsilon^{-1}+\mathcal O(\ln^{-2}\epsilon^{-1})$. Given that one \textit{typically} (in the sense that the Diophantine numbers have almost full measure) have to wait longer than $m=\lfloor c_3^{-1} \ln^{(2+d)^{-1}} \epsilon^{-1}\rfloor$ steps between solutions (cf. \propref{distsol}), I can therefore, with little loss of generality, conclude that there can be at most
\begin{align}
 m^{-1} \ln \epsilon^{-1} =\mathcal O( \ln^{\frac{1+d}{2+d}} \epsilon^{-1}),\eqlab{upperest}
\end{align}
solutions within an order $1$ interval of $\hat z_0$-values. This gives $2^\circ$ of the main result.
\begin{remark}
 This result could perhaps be improved to something like \conref{con1}. It would require a theory for the asymptotic distribution of points on the circle $\mathbb R/(\pi \mathbb Z)$ of ``forced'' circle maps of the form \eqref{gng0}. 
\xqed{\lozenge}\end{remark}

\subsection{Distribution of stable solutions - Part $3^\circ$ of the main result}\seclab{part3}
Thus far I have kept $\epsilon$ fixed but small without being able to say anything about the existence or non-existence of stable solutions close to the bifurcating normally elliptic slow manifold. In fact, the numerics from above seem to indicate that both situations with existence and non-existence can occur. In this section, I will, however, study how stable solutions can be created when varying $\epsilon$. This will cover part $3^\circ$ of the main result. I still focus on case (i) although the result is also true for case (ii). Consider $\hat z_0^0=\hat z_0^0(\epsilon)$ from section \secref{part2} above. It is a solution of $\lambda_l(\hat z_0^0)=\frac{\pi}{2}+\hat \lambda\ln^{-1}\epsilon^{-1}$ with $\hat \lambda$ held fixed at a value in the interior of the interval in \eqref{xiint}. This gives the following expression for $(\hat z_0^0)'(\epsilon)=\frac{d}{d\epsilon}\hat z_0^0(\epsilon)$:
\begin{align*}
 (\hat z_0^0)'(\epsilon) = -(\partial_{\hat z_0} \lambda_l)^{-1} ({\partial_\epsilon \lambda_l}+\mathcal O(\epsilon^{-1} \ln^{-2}\epsilon^{-1})) = \ln^{-1} (\epsilon^{-1})\epsilon^{-2}e_3+\mathcal O(\ln^{-2}(\epsilon^{-1})\epsilon^{-2}).
\end{align*}
The rate of change of $f_0=f_0(\epsilon)$, which is the image of $\hat z_0^0(\epsilon)$ under $F_{i}^{(2)}$, is then also easily obtained:
\begin{align}
f_0'(\epsilon) = ((2A+D_1)e_3-e_1)\epsilon^{-2}+\mathcal O(\ln^{-1}(\epsilon^{-1})\epsilon^{-2}).\eqlab{f0prime}
\end{align}

%
\begin{theorem}
Take any $\epsilon=\epsilon_1$ sufficiently small and set $I_1 = [\epsilon_1-c_1\epsilon_1^2,\epsilon_1+c_1\epsilon_1^2]$. Then within $I_1$ there will exist $\lfloor c_2^{-1}\ln \epsilon_1^{-1}\rfloor$-many closed intervals of lengths $\ge c_3^{-1} \epsilon_1^2 \ln^{-1}\epsilon_1$ for which there exists stable solutions. Here $c_1$, $c_2$ and $c_3$ may be large but they can be taken to be independent of $\epsilon_1$. 
\end{theorem}
\begin{proof}
Consider $f_0=f_0(\epsilon)$ with derivative $f_0'$ as in \eqref{f0prime}. Recall that $f_0$ is an end-point of a mapped interval that has length of order $\ln^{-1} \epsilon^{-1}$. If this interval does not intersect with $0$ in $\mathbb R/(\pi\mathbb Z)$ then one can simply, cf. \eqref{f0prime}, alter $\epsilon$ from $\epsilon_1$ by an amount of $c_1\epsilon_1^{2}$ to ``push'' this towards $0$ (provided that $(2A+D_1)e_3-e_1\ne 0$)  and ensure the existence of a stable solution $(\hat z_0^0,\hat \lambda)$. Since the mapped interval has a length of order $\ln^{-1} \epsilon_1$ this stable solution persists within a closed interval of $\epsilon$-values with size of order $c_3^{-1}\epsilon_1^2 \ln^{-1} \epsilon_1^{-1}$. The $\hat z_0$-interval based at $\hat z_0^0$ was arbitrary and the argument applies to all of the $\lfloor c_2^{-1}\ln \epsilon_1^{-1}\rfloor$-many intervals with end-points at $\hat z_0^n$. The solutions can be continued into true solutions by applying the contraction mapping theorem. 
\qed\end{proof}
By a similar argument to the one used in \propref{distsol} it will also follow that if one takes $\lfloor c_4^{-1} \ln^{(2+d)^{-1}} \epsilon_1^{-1}\rfloor$-following $f_n$'s, then they will typically be distant from each-other by a length of $\gg \ln^{-1}\epsilon_1^{-1}$. Here $d>1$. The relative measure within $I_1$ of the union of the $\lfloor c_2^{-1}\ln \epsilon_1^{-1}\rfloor$-many closed intervals of stable orbits will therefore typically be larger than $c_5^{-1} \ln^{-\frac{1+d}{2+d}} \epsilon_1^{-1}\gg \ln^{-1}\epsilon^{-1}$.
\subsection{Part $4^\circ$ of the main result}
The last part $4^\circ$ of the main result is the consequence of the following simple observation: The separation \eqref{sep} between two consecutive mapped intervals having end-points e.g. at $f_n$ and $f_{n+1}$ are
\begin{align*}
3\pi+ 3\pi e^{-2\pi \hat z_0}+\mathcal O(e^{-4\pi \hat z_0})  = 3\pi e^{-2\pi \hat z_0}+\mathcal O(e^{-4\pi \hat z_0})\quad \text{mod}\,\pi,
\end{align*}
cf. \eqref{2ADasymp}, dropping for simplicity the super-script on $\hat z_0^{n}$. By taking $\hat z_0$ ``large'' of size $\mathcal O(\ln \ln \epsilon^{-1})$ I can therefore guarantee that the separation is small being of order $\ln^{-1} \epsilon^{-1}$. Indeed, replace $\hat z_0 $ by $\frac{1}{2\pi} \ln (\hat Z_0^{-1} \ln \epsilon^{-1})$. Then
\begin{align*}
 3\pi e^{-2\pi \hat z_0}+\mathcal O(e^{-4\pi \hat z_0}) &=3\pi e^{\ln (\hat Z_0 \ln^{-1} \epsilon^{-1})}+\mathcal O(e^{\ln (\hat Z_0^{2} \ln^{-2} \epsilon^{-1})})\\
 &=3\pi \hat Z_0\ln^{-1} \epsilon^{-1} + \mathcal O(\ln^{-2} \epsilon^{-1}).
\end{align*}
Hence for $\hat Z_0$ large two consecutive intervals, with left end-points at $f_n$ and $f_{n+1}$, are guaranteed to overlap. 
\begin{remark}
Taking $\hat z_0$ of this form implies, cf. e.g. \eqref{ple}, that the motion of $(x,y)$ undergo fast oscillations near the passage through $u=0$. Also cf. e.g. \remref{remx} the distance to the slow manifold is $\mathcal O(\epsilon^{1/3}\ln^{1/2}\ln \epsilon^{-1})$.
 \xqed{\lozenge}\end{remark}

All of the estimates above can be modified to account for $\hat z_0$ of this form. For $P_i$ and $P_o$ this follow immediately from the analysis above. For $P_{cr}$ one needs to analyze the correction term. This leads into similar calculations as the ones performed for $P_i$ and $P_o$. The changes are therefore minor and I therefore leave the details out of the manuscript.

Now, remember that within an $\mathcal O(1)$-interval of initial $\hat z_0$-values there are $\mathcal O(\ln \epsilon^{-1})$-many mapped intervals of lengths $\approx 12\pi \ln^{-1}\epsilon^{-1}$ cf. \eqsref{xiint}{2ADasymp}. Therefore by taking $\hat Z_0$ sufficiently small, but independent of $\epsilon$, and a sufficiently large interval of initial $\hat z_0$-values, the union of these many intervals can be guaranteed to cover the whole circle $\mathbb R/(\pi \mathbb Z)$. In particular, there is at least one interval which intersects $0$. This provides the existence of a stable solution and proves $4^\circ$. 



I have collected the conclusions of $2^\circ$ and $4^\circ$ on case (i) in an illustration shown in \figref{intervalsF}.

\begin{figure}[h!]
\begin{center}
{\includegraphics[width=.8\textwidth]{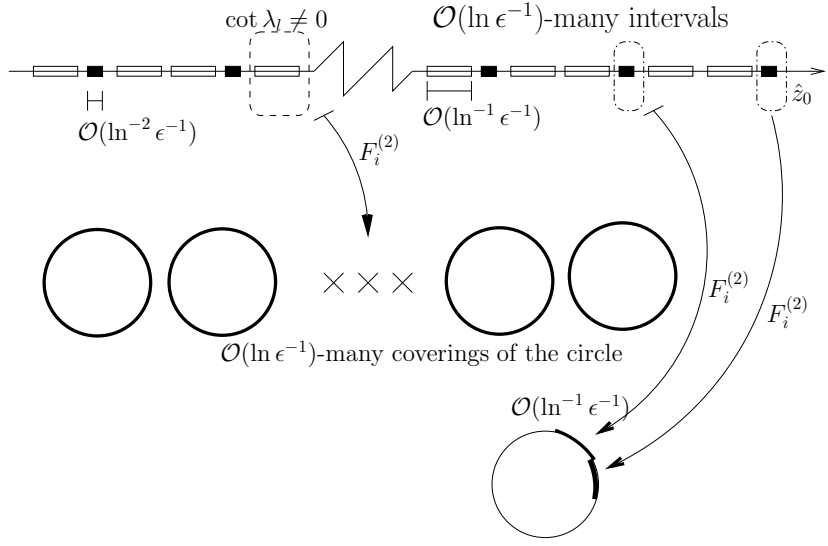}}
\end{center}
\caption{Unstable and periodic orbits obtained from case (i). On the $\hat z_0$-axis on the top the non-filled rectangles indicate the $\hat z_0$-values where $\cot^2 \lambda_l\ge c^{-1}$, giving rise to unstable periodic orbits, while the black, filled rectangles correspond to regions where stability could be attained according to \eqref{xiint}. The image under $F_{i}^{(2)}$ of one of the non-filled rectangles covers the circle $\mathbb R/(\pi \mathbb Z)$ an order of $\mathcal O(\ln \epsilon^{-1})$-many times cf. \eqref{Fi2p}. This gives rise to a total of $\mathcal O(\ln^2 \epsilon^{-1})$-many unstable solutions. The image under $F_i^{(2)}$ of one the smaller filled rectangles, however, only covers a small $\mathcal O(\ln^{-1}\epsilon^{-1})$-portion of the circle. According to $4^\circ$, however, taking large $\hat z_0$-values of order $\ln \ln \epsilon^{-1}$, two upon each-other following images under $F_i^{(2)}$ can be guaranteed to overlap in the way shown by the thick tubes of $F_i^{(2)}$-images of the two rectangles highlighted by the dash-dotted lines in the top right corner. This way by taking $\mathcal O(\ln \epsilon^{-1})$-many mapped intervals the circle can be covered completely by $F_i^{(2)}$-images of black rectangles and the existence of at least one stable solution follows.}
\figlab{intervalsF}
\end{figure}
\section{Stability islands}\seclab{resisl}
Generically stability islands surround stable fix points of $P$ or $P^2$. How large are these stability islands? To address this I first need to introduce a further blowup or scaling: 
\begin{align}
\hat z_0 = \ln^{-2}(\epsilon^{-1})\hat Z, \quad w_0 =\ln^{-1}(\epsilon^{-1})\hat W.\eqlab{finalBlowUp}
\end{align}
The purpose of this is to obtain an order $1$ blowup Poincar\'e mapping denoted by $\hat P$ when $\lambda_l$ is given as in \eqref{xieqn}. Indeed, from \lemmaref{jac} it follows that the Jacobian of $\hat P=\hat P(\hat Z,\hat W)$ with respect to these variables satisfy:
\begin{align*}
  \partial_{(\hat Z,\hat W)} \hat P &= \left( \begin {array}{cc} AD_1+{\frac { \left( 2A+D_1 \right) 
\hat \lambda}{\pi }}+{\frac {{\hat \lambda}^{2}}{{\pi }^{2}}}&-{\frac {D_1\hat \lambda}{\pi }}-{
\frac {{\hat \lambda}^{2}}{{\pi }^{2}}}\\ \noalign{\medskip}-2A \left( 2A+D_1
 \right) -{\frac { \left( 4A+D_1 \right) \hat \lambda}{\pi }}-{\frac {{\hat \lambda}^{2}
}{{\pi }^{2}}}&2-AD_1+{\frac { \left( 2A+D_1 \right) \hat \lambda}{\pi }}+{
\frac {{\hat \lambda}^{2}}{{\pi }^{2}}}\end {array}
 \right)\\
 &+\mathcal O(\ln^{-1}\epsilon^{-1}),
\end{align*}
in case (i) and
\begin{align*}
 \partial_{(\hat Z,\hat W)} \hat P = \left( \begin {array}{cc} AD_1-2{\frac {A\hat \lambda}{\pi }}-{\frac {{\hat \lambda}^{2
}}{{\pi }^{2}}}&{\frac {{\hat \lambda}^{2}}{{\pi }^{2}}}\\ \noalign{\medskip}-2
a \left( 2A+D_1 \right) +{\frac {{\hat \lambda}^{2}}{{\pi }^{2}}}&2-AD_1+2{
\frac {A\hat \lambda}{\pi }}-{\frac {{\hat \lambda}^{2}}{{\pi }^{2}}}\end {array}
 \right)+\mathcal O(\ln^{-1}\epsilon^{-1}),
\end{align*}
in case (ii). Also
 \begin{align*}
   \hat \lambda = \pi/2 \ln \epsilon^{-1} - \hat W + (1+\ln^{-1}(\epsilon^{-1})\ln e_4)\hat Z+\ln (\epsilon^{-1})G(\ln^{-2}(\epsilon^{-1})\hat Z).
 \end{align*}
 Note in particular how the trace in case (i) here agrees with the one presented in \lemmaref{xicond}. Consider $(\hat Z,\hat W)=(\hat Z_e,\hat W_e)$ a stable fix points of $\hat P$ or $\hat P^2$. Then generically KAM-theory can be applied to $\hat P$ or $\hat P^2$ to conclude the existence of invariant curves surrounding $(\hat Z,\hat W)=(\hat Z_e,\hat W_e)$. The last one of such invariant curves creates a resonance island that measures $\mathcal O(1)$ in the $(\hat Z,\hat W)$-space; $\mathcal O(\ln^{-3}\epsilon^{-1})$ in $(\hat z_0,w_0)$-space cf. \eqref{finalBlowUp}; $\mathcal O(\epsilon \ln^{-3}\epsilon^{-1})$ in the original variables $(x,y)$ cf. \eqref{blowup}.
\section{Future work}
The existence of the stable orbits in $4^\circ$ provides a beginning of a connection with the work in \cite{nei97,nei06,nei09}. Future work should seek to describe a more detailed connection to this work, by providing a description of the distribution of periodic orbits further away from the slow manifold. It is reasonable to believe that this requires a combination of the techniques used here with those used in \cite{nei97,nei06,nei09}.
\section{Acknowledgement}
I would like to thank Prof. A. I. Neishtadt for pointing me in the direction of \cite{nei99} and for suggestions leading to an improved manuscript.
\newpage
\appendix
\section{Proof of \lemmaref{separationlem}}\applab{prooflem}
I focus on the estimates for the Jacobian. The first statement about the growth of $\hat x$ and $\hat y$ will follow from similar estimates. First I take $\hat u\le -\breve u_*\delta$ with $\breve u_*$ large as in \lemmaref{asymplem}. Then from the results presented in that lemma I obtain the following asymptotics
 \begin{align*}
 \partial_{(\hat{x}(-\hat u_*),\hat{y}(-\hat u_*))}\begin{pmatrix}
                                                  \hat{x}(\hat u)\\
                                                  \hat{y}(\hat u)
                                                 \end{pmatrix}& = \partial_{(\hat{z}_0,w_0)}\begin{pmatrix}
                                                  \hat{x}(\hat u)\\
                                                  \hat{y}(\hat u)
                                                 \end{pmatrix} \partial_{(\hat z_0,w_0)}\begin{pmatrix}
 \hat x(-\hat u_*) \\
              \hat y(-\hat u_*)
                                                 \end{pmatrix}^{-1}\\
                                                 &= \mathcal O\left(\ln \delta^{-1} \begin{pmatrix}
                                                 \delta^{-1/4} & \delta^{-1/4} \\
                                                  \delta^{1/4}& \delta^{1/4}
                                                 \end{pmatrix}\right),
 \end{align*}
For $\hat u\in (-\breve u_*\delta,\breve u_*\delta)$ I use the coordinates: $(\breve x,\breve y)=(\delta^{1/4} \hat x, \delta^{-1/4} \hat y)$ and $\hat u$ replaced by $\breve u=\delta^{-1} \hat u$, also used in \cite{nei99}, to control the assignment $(\breve x,\breve y)(-\breve u_* )\mapsto (\breve x,\breve y)({\hat u\delta^{-1}})$, with a bound that is independent of $\delta$. Returning to my coordinates I then obtain the following 
\begin{align*}
\partial_{(\hat x(-\breve u_* \delta ),\hat y(-\breve u_* \delta ))}\begin{pmatrix}
                                                  \hat x(\hat u)\\
                                                  \hat y(\hat u))
                                                 \end{pmatrix} = \mathcal O \begin{pmatrix}
                                                 1&\delta^{-1/2}\\
                                                 \delta^{1/2} &1
                                                                                                   \end{pmatrix}.                                                 
\end{align*}
 Therefore
 \begin{align*}
 \partial_{(\hat x(-\hat u_*),\hat y(-\hat u_*))}\begin{pmatrix}
                                                  \hat x(\hat u)\\
                                                  \hat y(\hat u))
                                                 \end{pmatrix} &= \partial_{(\hat x(-\breve u_* \delta),\hat y(-\breve u_* \delta))}\begin{pmatrix}
                                                  \hat x(\hat u)\\
                                                  \hat y(\hat u))
                                                 \end{pmatrix} \partial_{(\hat x(-\hat u_*),\hat y(-\hat u_* ))}\begin{pmatrix}
                                                  \hat x(-\breve u_* \delta)\\
                                                  \hat y(-\breve u_* \delta )
                                                 \end{pmatrix}\\
                                                 &= \mathcal O\left(\ln \delta^{-1}  \begin{pmatrix}
                                                 1&\delta^{-1/2}\\
                                                 \delta^{1/2} &1
                                                                                                   \end{pmatrix}\begin{pmatrix}
                                                 \delta^{-1/4} & \delta^{-1/4} \\
                                                  \delta^{1/4}& \delta^{1/4}
                                                 \end{pmatrix}\right)\\
                                                 &= \mathcal O\left(\ln \delta^{-1} \begin{pmatrix}
                                                 \delta^{-1/4 } & \delta^{-1/4 }\\
                                                 \delta^{1/4} & \delta^{1/4}
                                                                                                   \end{pmatrix}\right).
 \end{align*}
  Finally, I consider $\hat u\ge \breve u_*\delta$ and use the following
\begin{align*} 
  \partial_{(\hat x(-\hat u_*),\hat y(-\hat u_*))}\begin{pmatrix}
                                                  \hat x(\hat u)\\
                                                  \hat y(\hat u))
                                                 \end{pmatrix} &= \partial_{(\hat \varrho_0,\phi_0)}\begin{pmatrix}
                                                  \hat x(\hat u)\\
                                                  \hat y(\hat u))
                                                 \end{pmatrix}\partial_{(\hat z_0,w_0)}\begin{pmatrix}
 \hat \varrho_0 \\
              \phi_0
                                                 \end{pmatrix} \partial_{(\hat z_0,w_0)}\begin{pmatrix}
 \hat x(-\hat u_*) \\
              \hat y(-\hat u_*)
                                                 \end{pmatrix}^{-1}\\
                                                 &= \mathcal O\left(\ln \delta^{-1} \begin{pmatrix}
                                                 \hat u^{-1/4} \ln (\delta^{-1} \hat u) & \hat u^{-1/4} \ln (\delta^{-1} \hat u) \\
                                                  \hat u^{1/4} \ln (\delta^{-1} \hat u) & \hat u^{1/4} \ln (\delta^{-1} \hat u)
                                                 \end{pmatrix}\right).
\end{align*}
Setting $\hat u=\hat u_*$ here gives \eqref{jacestu0}.  For \eqref{jacestu} I combine the asymptotics to obtain the following bound
\begin{align*}
   \vert \partial_{(\hat x(-\hat u_*),\hat y(-\hat u_*))}\begin{pmatrix}
                                                  \hat x(\hat u)\\
                                                  \hat y(\hat u))
                                                 \end{pmatrix} \vert \le c \delta^{-1/4}\ln \delta^{-1},
\end{align*}
uniformly in $\hat u$. The estimate of the inverses can be derived in a similar manner using the fact that the Jacobian has determinant equal to $1$. 
\section{Approximation of $P_{o}$}\applab{pouterapp}
I start by presenting two lemmata similar to \lemmaref{Omegacontrol0} and \lemmaref{Omegacontrol}:
\begin{lemma}\lemmalab{Omegacontrol01}
 Let $q$ and $p$ be positive real numbers satisfying $0<q<p$. Then there exists a constant $c=c(\hat u_*)$ so that
 \begin{align*}
  \hat F(\hat u)^{-2p}\le c^{p-q} \hat F(\hat u)^{-2q}, 
 \end{align*}
for all $\hat u\in [-\mu^{-2}(\pi-\tau/2),-\hat u_*]$.
\end{lemma}
\begin{lemma}\lemmalab{Omegacontrol1}
 Let $q\in \overline{\R}_+$. Given an integrable function $r=r(\hat u)$ satisfying the following estimate
 \begin{align*}
 \vert r(\hat u)\vert \le \hat F(\hat u)^{-2q},\quad \hat u\in [\hat u_*,\mu^{-2} \tau/2].
 \end{align*}
 If $q< 1$ then there exists a $c_1=c_1(q)$ so that
\begin{align*}
 \vert \int_{-\mu^{-2}(\pi-\tau/2)}^{-\hat u_*} r(\hat u)d\hat u\vert \le c_1\mu^{-2(1-q)}.
\end{align*}
If $q=1$ then there exists a $c_2$ so that
\begin{align*}
 \vert \int_{-\mu^{-2}(\pi-\tau/2)}^{-\hat u_*} r(\hat u)d\hat u\vert &\le c_2\ln (\mu^{-2} \hat u_*^{-1}).
\end{align*}
Finally if $q>1$ then the corresponding integral is uniformly bounded with respect to $\epsilon$: There exists a $c_3$ so that 
\begin{align*}
 \vert \int_{-\mu^{-2}(\pi-\tau/2)}^{-\hat u_*} r(\hat u)d\hat u\vert \le c_3(q-1)^{-1} \hat u_*^{1-q}.
\end{align*}
\end{lemma}
The proofs of these lemmata are almost identical to the proofs of  \lemmaref{Omegacontrol0} and \lemmaref{Omegacontrol} and therefore left out.

I then recall the form of $\hat H$ in \eqref{hatH1}:
\begin{align}
  \hat H = h_0(\hat u,\hat v,\hat z_0) +r_0(\hat u,\hat v,\hat z_0,w_0)+\mathcal O(\hat F(\hat u)^{-5}\delta^3),\eqlab{Houter2}
\end{align}
and note that
\begin{align*}
 \overline{r}_0 = \frac{1}{2\pi}\int_0^{2\pi} r_0(\hat u,\hat z_0,\tau)d\tau &= \frac{3}{4}\delta^{3/2} \hat F(\hat u)^{-2} \hat z_0^2-\frac14 \delta^{3/2} \hat F(\hat u)^{-2} f(u)M_0(u) \hat z_0^2\\
 &=\frac{3}{4}\delta^{3/2} \hat F(\hat u)^{-2}(1+\mathcal O(u)) \hat z_0^2.
\end{align*}
I then introduce the following generating function
\begin{align*}
  G(\hat u,\hat v,\hat z_{0},w_1)&=\delta^{-3/2} \hat u \hat v+\hat z_0w_1 +\hat F(\hat u)^{-1}\int_0^{\tilde \phi_1}\tilde r_{0} (\hat u,\hat z_0,\tau) d\tau,\\
  \tilde r_{0} &=r_{0}-\overline{r}_{0}.
   \end{align*}
   This generates a symplectic transformation $(\hat u,\hat v_0,\hat z_0,w_0)\mapsto (\hat u,\hat v,\hat z,w)$ with $\hat z_0=\hat z+\mathcal O(\delta^{3/2})$ transforming $H$ \eqref{Houter2} into
\begin{align*}
 \hat H = \hat v+\hat F(\hat u) \hat z+\frac{3}{4}\delta^{3/2} \hat F(\hat u)^{-2}(1+\mathcal O(u))\hat z^2 +\mathcal O(\hat F(\hat u)^{-5}\delta^3).
\end{align*}
The equations of motion are
\begin{align}
 \frac{d\hat z}{d\hat u} &=\mathcal O( \hat F(\hat u)^{-5}\delta^{3/2}),\nonumber\\
 \frac{dw}{d\hat u}&=- \delta^{-3/2}\hat F(\hat u)-\frac{3}{2} \hat F(\hat u)^{-2}(1+\mathcal O(u))\hat z+\mathcal O(\hat F(\hat u)^{-5}\delta^{3/2}).\eqlab{hatphi}
\end{align}
In accordance with the definition of $P_o$ I consider $\hat u$ here from $\hat u=-\mu^{-2} (\pi-\tau/2)$ to $\hat u=-\hat u_*$. 
\begin{lemma}
$\hat z_0$ is conserved on the interval from $\hat u=-\mu^{-2} (\pi-\tau/2)$ to $\hat u=-\hat u_*$ up to an error of order $\mathcal O(\delta^{3/2})$. 
\end{lemma}
\begin{proof}
I can take $p=5/2>1$ in \lemmaref{Omegacontrol1} to control the variation of $\hat z$ by an error of order $\delta^{3/2}$. Since $\hat z_0=\hat z+\mathcal O(\delta^{3/2})$ the result follows.
\qed\end{proof}
%

By a similar argument, I estimate the effect of the remainder in \eqref{hatphi} by $\mathcal O(\delta^{3/2})$ and I compute the variation in the angle $w$ by
\begin{align}
 w(-\hat u_*)&=w(-\mu^{-2} (\pi-\tau/2) ) -\int_{-\mu^{-2} (\pi-\tau/2)}^{-\hat u_*} \bigg(\delta^{-3/2}\hat F(\hat u) \nonumber\\
 &+\frac32 \hat F(\hat u)^{-2}(1+\mathcal O(u))\hat z_0+\mathcal O( \hat F(\hat u)^{-2}\delta^{3/2})\bigg)d\hat u+\mathcal O(\delta^{3/2}).\eqlab{hatphiPi}
 \end{align}
 As above in \secref{Pinner}, the remainder $\mathcal O(\hat F(\hat u)^{-2} \delta^{3/2} )$ in the integral comes from $\hat z(\hat u)=\hat z_0+\mathcal O(\delta^{3/2})$, with $\hat z_0=\text{const}.$ on this interval. This can be estimated from above by a term of order $\delta^{3/2}\ln \mu^{-1}$ using $q=1/2$ in \lemmaref{Omegacontrol1}. The following lemma gives asymptotics of the two other integrals appearing in \eqref{hatphiPi}.
 \begin{lemma}
 \begin{align*}
  \delta^{-3/2} \int_{-\mu^{-2} (\pi-\tau/2)}^{-\hat u_*} \hat F(\hat u) d\hat u &= \mu^{-3}\delta^{-3/2} e_3 -\frac{2}{3}\delta^{-3/2}\hat u_*^{3/2}+\mathcal O(\epsilon^{2/3}\delta^{-5/2}).
   \end{align*}
   with
   \begin{align*}
    e_3=\int_0^{\pi-\tau/2} (-f(-u))^{1/2}du.
   \end{align*}
Moreover, there exists some positive constant $e_4$ such that
\begin{align*}
\int_{-\mu^{-2} (\pi-\tau/2)}^{-\hat u_*} \hat F(\hat u)^{-2}d\hat u = \ln(e_4 \mu^{-2}\delta^{-1})-\ln (\delta^{-1} \hat u_*)+\mathcal O(\mu^2).
\end{align*}
 \end{lemma}
\begin{proof}
 I use \eqref{hatOmega}:
 \begin{align}
    \delta^{-3/2} \int_{-\mu^{-2} (\pi-\tau/2)}^{-\hat u_*}\hat F(\hat u) d\hat u  &= \mu^{-3}\delta^{-3/2} \int_{u_*}^{\pi-\tau/2} (-f(-u))^{1/2} du \nonumber\\
    &= \mu^{-3}\delta^{-3/2} e_3 -\mu^{-3}\delta^{-3/2} \int_0^1 (-f(-u_*s))^{1/2}ds u_*,\eqlab{intint}
 \end{align}
 here $u_*=\mu^2 \hat u_*$. Since $f(u)=u+\mathcal O(u^2)$ cf. \eqref{condMfV} for small $u$, I can for the last integral in \eqref{intint} use the same argument used for \eqref{intvsqrt} to complete the first part of the proof.

For the second part:
\begin{align*}
\int_{-\mu^{-2} (\pi-\tau/2)}^{-\hat u_*} \hat F(\hat u)^{-2}(1+\mathcal O(u))d\hat u =  \int_{u_*}^{\pi-\tau/2} (-f(-u))^{-1}(1+\mathcal O(u)) du.
\end{align*}
I write $(-f(-u))^{-1}=u^{-1}+\mathcal O(1)$ for $u$ small and complete the result as in the proof of \lemmaref{intint}.
\qed\end{proof}
Following this lemma I can therefore write \eqref{hatphiPi} as
\begin{align}
 w_0(-\hat u_* )&=w_0(-\mu^{-2}(\pi-\tau/2) ) -\epsilon^{-1} e_3- \frac32 \ln(e_4 \mu^{-2}\delta^{-1})\hat z_0 +\frac{2}{3}\delta^{-3/2} \hat u_*^{3/2}  \nonumber\\+\frac32\ln (\delta^{-1} \hat u_*)\hat z_0
 &+\mathcal O(\delta^{3/2}\ln \mu^{-1}).\eqlab{outerphi}
\end{align}
using here that $w_0=w+\mathcal O(\delta^{3/2})$. Colleting the results, I obtain \propref{pouter}. 
\bibliography{refs}

\begin{thebibliography}{10}

\bibitem{arn2}
V.~I. Arnol'd, V.~V. Kozlov, and A.~I. Neishtadt.
\newblock {\em Mathematical Aspects of Classical and Celestial Mechanics}.
\newblock 2nd Edition, Springer-Verlag Berlin Heidelberg, 1991.

\bibitem{ref1}
J.~B\"uckner and L.~M. Zelenyi.
\newblock Regular and chaotic charged particle motion in magnetotaillike field
  reversals. 1: Basic theory and trapped motion.
\newblock {\em Journal of Geophysical Research A}, 94:11821--11842, 1989.

\bibitem{fen1}
N.~Fenichel.
\newblock Persistence and smoothness of invariant manifolds for flows.
\newblock {\em Indiana University Mathematics Journal}, 21:193--226, 1971.

\bibitem{fen2}
N.~Fenichel.
\newblock Asymptotic stability with rate conditions.
\newblock {\em Indiana University Mathematics Journal}, 23:1109--1137, 1974.

\bibitem{fif1}
P.~C. Fife.
\newblock A phase plane analysis of a corner layer problem arising in the study
  of crystalline grain boundaries.
\newblock {\em Unpublished but available here:
  \verb#http://www.math.utah.edu/~fife/hardcase.pdf#}, 2009.

\bibitem{geller1}
V.~Gelfreich and L.~Lerman.
\newblock {Almost invariant elliptic manifold in a singularly perturbed
  Hamiltonian system}.
\newblock {\em Nonlinearity}, 15:447--557, 2002.

\bibitem{geller2}
V.~Gelfreich and L.~Lerman.
\newblock {Long-periodic orbits and invariant tori in a singularly perturbed
  Hamiltonian system}.
\newblock {\em Physica D}, 176(3-4), 2003.

\bibitem{pl2}
A.~R. Its and A.~A. Kapaev.
\newblock The method of isomonodromy deformations and connection formulas for
  the second painlev\'e transcendent.
\newblock {\em Mathematics of the USSR-Izvestiya}, 31:878--892, 1988.

\bibitem{sz1}
M.~Krupa and P.~Szmolyan.
\newblock Extending geometric singular perturbation theory to non-hyperbolic
  points - fold and canard points in two dimensions.
\newblock {\em SIAM Journal of Mathematical Analysis}, 33:286--314, 2001.

\bibitem{lu2013a}
Nan Lu.
\newblock Local dynamics at focal points coupled with elliptic directions.
\newblock {\em {Journal of Dynamics and Differential Equations}},
  25(3):715--732, 2013.

\bibitem{Lu20114124}
Nan Lu and Chongchun Zeng.
\newblock Normally elliptic singular perturbations and persistence of
  homoclinic orbits.
\newblock {\em Journal of Differential Equations}, 250(11):4124 -- 4176, 2011.

\bibitem{mac1}
R.~S. MacKay.
\newblock Slow manifolds.
\newblock {\em In: ``Energy Localisation and Transfer'', eds T Dauxois, A
  Litvak-Hinenzon, RS MacKay, A Spanoudaki, World Scientific}, pages 149--192,
  2004.

\bibitem{malchiodi2012a}
Andrea Malchiodi, Scuola Internazionale, and Superiore~Studi Avanzati.
\newblock Adiabatic limits of closed orbits for some newtonian systems in r n.
\newblock 2012.

\bibitem{nei87}
A.~I. Neishtadt.
\newblock On the change in the adiabatic invariant on crossing a separatrix in
  systems with two degrees of freedom.
\newblock {\em Journal of Applied Mathematics and Mechanics}, 51:586--592,
  1987.

\bibitem{nei04}
A.~I. Neishtadt and V.~V. Sidorenko.
\newblock Wisdom system: Dynamics in the adiabatic approximation.
\newblock {\em Celestial Mechanics and Dynamical Astronomy}, 90:307--330, 2004.

\bibitem{nei97}
A.~I. Neishtadt, V.~V. Sidorenko, and D.~V. Treschev.
\newblock Stable periodic motion in the problem on passage through a
  separatrix.
\newblock {\em Chaos}, 7:2--11, 1997.

\bibitem{nei05}
A.~I. Neishtadt and V.~V. Vasiliev.
\newblock Phase change between separatrix crossings in systems with two degrees
  of freedom.
\newblock {\em Nonlinearity}, 18:1393--1406, 2005.

\bibitem{schsou1}
S.~Schecter and C.~Sourdis.
\newblock {Heteroclinic orbits in slow-fast Hamiltonian systems with slow
  manifold bifurcations}.
\newblock {\em {Journal of Dynamics and Differential Equations}}, 22:629 --
  655, 2010.

\bibitem{Shatah2002572}
Jalal Shatah and Chongchun Zeng.
\newblock Periodic solutions for hamiltonian systems under strong constraining
  forces.
\newblock {\em Journal of Differential Equations}, 186(2):572 -- 585, 2002.

\bibitem{soufif1}
C.~Sourdis and P.~C. Fife.
\newblock Existence of heteroclinic orbits for a corner layer problem in
  anisotropic interfaces.
\newblock {\em Adv. Differential Equations}, 12:623--668, 2007.

\bibitem{kri1}
K.~Uldall~Kristiansen, P.~Palmer, and R.~M. Roberts.
\newblock A unification of models of tethered satellites.
\newblock {\em SIAM Journal of Applied Dynamical Systems}, 10:1042--1069, 2011.

\bibitem{kri2}
K.~Uldall~Kristiansen, P.~Palmer, and R.~M. Roberts.
\newblock The persistence of a slow manifold with bifurcation.
\newblock {\em SIAM Journal of Applied Dynamical Systems}, 11:661--683, 2012.

\bibitem{kri3}
K.~Uldall~Kristiansen and C.~Wulff.
\newblock Exponential estimates of slow manifolds.
\newblock {\em arXiv:1208.4219v1 [math.DS]}, 2012.

\bibitem{nei99}
D.~L. Vainshtein, L.~M. Zelenyi, A.~I. Neishtadt, and B.V Savenkov.
\newblock Jumps in an adiabatic invariant with small initial values.
\newblock {\em Plasma Physics Reports}, 25:299--303, 1999.

\bibitem{nei06}
V.~V. Vasiliev, A.~I. Neishtadt, C.~Sim\'o, and D.~V. Treschev.
\newblock Stability islands in domains of separatrix crossings in slow-fast
  hamiltonian systems.
\newblock {\em Proceedings of the Steklov Institute of Mathematics},
  259:237--247, 2007.

\bibitem{nei09}
V.~V. Vasiliev, A.~I. Neishtadt, C.~Sim\'o, and D.~V. Treschev.
\newblock Periodic orbits and stability islands in chaotic seas created by
  separatrix crossings in slow-fast systems.
\newblock {\em Discrete and Continuous Dynamical Systems Series B},
  10:621--665, 2008.

\end{thebibliography}
\bibliographystyle{plain}




%
%
%
\end{document}